\documentclass[12pt]{amsart}

\usepackage{amsmath}
\usepackage{amsthm}
\usepackage{amssymb}
\usepackage{caption}
\usepackage[curve]{xypic}
\usepackage{cite}
\usepackage{enumitem}
\usepackage{color}
\usepackage{url}
\usepackage[usenames,dvipsnames]{xcolor}
\usepackage{graphicx}
\usepackage{tikz}
\usepackage{tikz-cd}
\usepackage{hyperref} 
\usepackage{verbatim}

\setlist{itemsep=4pt, topsep=4pt}

\makeatletter

\def\chaptermark#1{}

\def\chapter{%
  \if@openright\cleardoublepage\else\clearpage\fi
  \thispagestyle{plain}\global\@topnum\z@
  \@afterindenttrue \secdef\@chapter\@schapter}

\def\@chapter[#1]#2{\refstepcounter{chapter}%
  \ifnum\c@secnumdepth<\z@ \let\@secnumber\@empty
  \else \let\@secnumber\thechapter \fi
  \typeout{\chaptername\space\@secnumber}%
  \def\@toclevel{0}%
  \ifx\chaptername\appendixname \@tocwriteb\tocappendix{chapter}{#2}%
  \else \@tocwriteb\tocchapter{chapter}{#2}\fi
  \chaptermark{#1}%
  \addtocontents{lof}{\protect\addvspace{10\p@}}%
  \addtocontents{lot}{\protect\addvspace{10\p@}}%
  \@makechapterhead{#2}\@afterheading}
\def\@schapter#1{\typeout{#1}%
  \let\@secnumber\@empty
  \def\@toclevel{0}%
  \ifx\chaptername\appendixname \@tocwriteb\tocappendix{chapter}{#1}%
  \else \@tocwriteb\tocchapter{chapter}{#1}\fi
  \chaptermark{#1}%
  \addtocontents{lof}{\protect\addvspace{10\p@}}%
  \addtocontents{lot}{\protect\addvspace{10\p@}}%
  \@makeschapterhead{#1}\@afterheading}
\newcommand\chaptername{Chapter}

\def\@makechapterhead#1{\global\topskip 7.5pc\relax
  \begingroup
  \fontsize{\@xivpt}{18}\bfseries\centering
    \ifnum\c@secnumdepth>\m@ne
      \leavevmode \hskip-\leftskip
      \rlap{\vbox to\z@{\vss
          \centerline{\normalsize\mdseries
              \uppercase\@xp{\chaptername}\enspace\thechapter}
          \vskip 3pc}}\hskip\leftskip\fi
     #1\par \endgroup
  \skip@34\p@ \advance\skip@-\normalbaselineskip
  \vskip\skip@ }
\def\@makeschapterhead#1{\global\topskip 7.5pc\relax
  \begingroup
  \fontsize{\@xivpt}{18}\bfseries\centering
  #1\par \endgroup
  \skip@34\p@ \advance\skip@-\normalbaselineskip
  \vskip\skip@ }
\def\appendix{\par
  \c@chapter\z@ \c@section\z@
  \let\chaptername\appendixname
  \def\thechapter{\@Alph\c@chapter}}

\newcounter{chapter}

\newif\if@openright

\makeatother

\makeatletter 
\def\@cite#1#2{{\m@th\upshape\bfseries%
[{#1\if@tempswa{\m@th\upshape\mdseries, #2}\fi}]}}
\makeatother 

\theoremstyle{plain}
\newtheorem{thm}{Theorem}[section]
\newtheorem{cor}[thm]{Corollary}

\newtheorem{prop}[thm]{Proposition}
\newtheorem{lem}[thm]{Lemma}
\newtheorem{sublem}[thm]{Sublemma}

\theoremstyle{definition}
\newtheorem{defn}[thm]{Definition}

\theoremstyle{remark}
\newtheorem{rem}[thm]{Remark}

\numberwithin{equation}{subsection}
\renewcommand{\bold}[1]{\medskip \noindent {\bf #1 }\nopagebreak}

\newcommand{\nc}{\newcommand}
\newcommand{\rnc}{\renewcommand}

\nc\bA{\mathbb{A}}
\nc\bB{\mathbb{B}}
\nc\bC{\mathbb{C}}
\nc\bD{\mathbb{D}}
\nc\bE{\mathbb{E}}
\nc\bF{\mathbb{F}}
\nc\bG{\mathbb{G}}
\nc\bH{\mathbb{H}}
\nc\bI{\mathbb{I}}
\nc{\bJ}{\mathbb{J}} 
\nc\bK{\mathbb{K}}
\nc\bL{\mathbb{L}}
\nc\bM{\mathbb{M}}
\nc\bN{\mathbb{N}}
\nc\bO{\mathbb{O}}
\nc\bP{\mathbb{P}}
\nc\bQ{\mathbb{Q}}
\nc\bR{\mathbb{R}}
\nc\bS{\mathbb{S}}
\nc\bT{\mathbb{T}}
\nc\bU{\mathbb{U}}
\nc\bV{\mathbb{V}}
\nc\bW{\mathbb{W}}
\nc\bY{\mathbb{Y}}
\nc\bX{\mathbb{X}}
\nc\bZ{\mathbb{Z}}
\nc\cA{\mathcal{A}}
\nc\cB{\mathcal{B}}
\nc\cC{\mathcal{C}}
\rnc\cD{\mathcal{D}}
\nc\cE{\mathcal{E}}
\nc\cF{\mathcal{F}}
\nc\cG{\mathcal{G}}
\rnc\cH{\mathcal{H}}
\nc\cI{\mathcal{I}}
\nc{\cJ}{\mathcal{J}} 
\nc\cK{\mathcal{K}}
\rnc\cL{\mathcal{L}}
\nc\cM{\mathcal{M}}
\nc\cN{\mathcal{N}}
\nc\cO{\mathcal{O}}
\nc\cP{\mathcal{P}}
\nc\cQ{\mathcal{Q}}
\rnc\cR{\mathcal{R}}
\nc\cS{\mathcal{S}}
\nc\cT{\mathcal{T}}
\nc\cU{\mathcal{U}}
\nc\cV{\mathcal{V}}
\nc\cW{\mathcal{W}}
\nc\cY{\mathcal{Y}}
\nc\cX{\mathcal{X}}
\nc\cZ{\mathcal{Z}}
\nc\bfA{\mathbf{A}}
\nc\bfB{\mathbf{B}}
\nc\bfC{\mathbf{C}}
\nc\bfD{\mathbf{D}}
\nc\bfE{\mathbf{E}}
\nc\bfF{\mathbf{F}}
\nc\bfG{\mathbf{G}}
\nc\bfH{\mathbf{H}}
\nc\bfI{\mathbf{I}}
\nc{\bfJ}{\mathbf{J}} 
\nc\bfK{\mathbf{K}}
\nc\bfL{\mathbf{L}}
\nc\bfM{\mathbf{M}}
\nc\bfN{\mathbf{N}}
\nc\bfO{\mathbf{O}}
\nc\bfP{\mathbf{P}}
\nc\bfQ{\mathbf{Q}}
\nc\bfR{\mathbf{R}}
\nc\bfS{\mathbf{S}}
\nc\bfT{\mathbf{T}}
\nc\bfU{\mathbf{U}}
\nc\bfV{\mathbf{V}}
\nc\bfW{\mathbf{W}}
\nc\bfY{\mathbf{Y}}
\nc\bfX{\mathbf{X}}
\nc\bfZ{\mathbf{Z}}

\nc{\dmo}{\DeclareMathOperator}
\nc{\wt}{\widetilde}
\rnc{\Re}{\operatorname{Re}}
\rnc{\Im}{\operatorname{Im}}
\rnc{\span}{\operatorname{span}}
\dmo{\rank}{rank}
\dmo{\End}{End}
\dmo{\Hom}{Hom}
\dmo{\Jac}{Jac}
\dmo{\Id}{Id}
\dmo{\Ann}{Ann}
\dmo{\Area}{Area}
\dmo{\CP}{\bC P^1}
\dmo{\rk}{rk}
\dmo{\rel}{rel}
\dmo{\ra}{\rightarrow}
\dmo{\Twist}{\mathrm{Twist}}
\dmo{\TwistX}{\mathrm{Twist}(X, \omega)}

\rnc{\Col}{\operatorname{Col}}

\nc{\ColOne}{\Col_{\bfC_1}}
\nc{\ColOneX}{\ColOne(X,\omega)}

\nc{\ColTwo}{\Col_{\bfC_2}}
\nc{\ColTwoX}{\ColTwo(X,\omega)}

\nc{\ColThree}{\Col_{\bfC_3}}
\nc{\ColThreeX}{\ColThree(X,\omega)}

\nc{\ColOneTwo}{\Col_{\bfC_1, \bfC_2}}
\nc{\ColOneTwoX}{\ColOneTwo(X,\omega)}

\nc{\ColOneThree}{\Col_{\bfC_1, \bfC_3}}
\nc{\ColOneThreeX}{\ColOneThree(X,\omega)}

\nc{\MOne}{\cM_{\bfC_1}}
\nc{\MTwo}{\cM_{\bfC_2}}
\nc{\MOneTwo}{\cM_{\bfC_1, \bfC_2}}

\nc{\MThree}{\cM_{\bfC_3}}

\nc{\MOneThree}{\cM_{\bfC_1, \bfC_3}}

\dmo{\For}{\cF}

\nc{\GL}{\mathrm{GL}^+(2, \bR)}

\title[Billiards in Right Triangles]{Billiards in right triangles and orbit closures in genus zero strata}
\author[Apisa]{Paul~Apisa}

\subjclass[2010]{32G15, 37D40, 14H15}

\begin{document}
\maketitle
\thispagestyle{empty}

\begin{abstract}
We classify $\mathrm{GL}(2, \mathbb{R})$ orbit closures of rank at least two in hyperelliptic components of strata of Abelian and quadratic differentials. As a consequence, the orbit closure of the unfolding of every rational right and isosceles triangle is computed and the asymptotic number of periodic billiard trajectories in these triangles is deduced. Additionally, given a fixed set of angles, the orbit closure of the unfolding of all unit area rational parallelograms, isosceles trapezoids, and right trapezoids outside of a discrete set is determined.
\end{abstract}

\setcounter{tocdepth}{2} 
\tableofcontents

\section{Introduction}\label{S:intro}

A \emph{rational polygon} is a polygon with connected\footnote{A generalization to polygons with disconnected boundaries appears in Masur-Tabachnikov \cite{MT}.} boundary in the plane all of whose angles are rational multiples of $\pi$. Katok and Zemlyakov \cite{KatokZemlyakov} observed that every rational polygon has an associated translation surface, called its \emph{unfolding}. Straight line flow on the unfolding corresponds to billiard flow on the original polygon. Our first goal in the sequel is to exploit this connection to study billiards in rational right and isosceles triangles\footnote{The dynamics of two elastic point masses confined to a line segment with barriers at the endpoints is isomorphic to billiards in a right triangle (see Masur-Tabachnikov \cite[Section 1.2]{MT}) so this analysis admits another application to physics.}. Underlying the analysis will be the determination of the $\mathrm{GL}(2,\mathbb{R})$ orbit closure of the unfolding of every rational right and isosceles triangle.

In rational polygons, periodic billiard trajectories are always shadowed by bands of nearby periodic billiard trajectories of the same length. Given a rational polygon $P$ we will let $N_P(L)$ denote the number of (bands of) periodic billiard trajectories in $P$ of length at most $L$. Masur \cite{Masur-LowerBounds, Masur-UpperBounds} showed that $\displaystyle{\liminf_{L \ra \infty} \frac{N_P(L)}{L^2}}$ and $\displaystyle{\limsup_{L \ra \infty} \frac{N_P(L)}{L^2}}$ are positive and finite for all $P$, i.e. that $N_P(L)$ has quadratic asymptotics. If these two limits agree we will say that $N_P(L)$ has \emph{exact quadratic asymptotics}. 

In general, it is unknown whether $N_P(L)$ always has exact quadratic asymptotics. Nevertheless, $N_P(L)$ always satisfies a weak form of quadratic asymptotics by work of Eskin-Mirzakhani \cite{EM} and Eskin-Mirzakhani-Mohammadi \cite[Theorem 2.12]{EMM}. Following Athreya-Eskin-Zorich \cite{AEZ}, we will say that $N_P(L) ``\sim" c \cdot L^2$ means that 
\[ \lim_{L \ra \infty} \frac{1}{L} \int_0^L N_P(e^t) e^{-2t} dt = c. \]
These limits are sometimes referred to as the \emph{weak asymptotics of $N_P$}. The main theorem is the following; note that we will occasionally abbreviate $\mathrm{gcd}(a,b)$ to just $(a,b)$. Moreover, $a$, $b$, and $n$ will always be taken to be positive integers in the sequel.

\begin{thm}\label{T:Counting:ClosedTrajectories}
Suppose that $P$ is a rational right triangle with angles $\left( \frac{a}{2n}, \frac{b}{2n}, \frac{1}{2}\right)\pi$ with $\mathrm{gcd}(a,b,2n) =1$ and $\min(a,b) > 2$, then
\[ N_P(L) ``\sim" \frac{1}{16\pi} \left( 1 - \frac{1}{n} \right)\left( 1 + \frac{2}{ab} \right) \frac{L^2}{\mathrm{area}(P)}.\]
Suppose that $T$ is a rational isosceles triangle with angles $\left( \frac{a}{n}, \frac{a}{n}, \frac{b}{n} \right)\pi$ with $\mathrm{gcd}(a,b,n) =1$, $a \ne 1$, and $b \notin \{1,2,4\}$, then
\[ N_T(L) ``\sim" \frac{1}{8\pi} \left( 1-\frac{(n,2)}{n} \right) \left( 1 + \frac{2}{ab (n,2)}\right) \frac{L^2}{\mathrm{area}(T)}. \]
\end{thm}
%
%
%
%
%
%
%
%
%

We will say that a rational right or isosceles triangle is \emph{unexceptional} if it satisfies the hypotheses of Theorem \ref{T:Counting:ClosedTrajectories}.

One striking observation is the following. Let $c(\alpha)$ be the number, wherever defined, so that $N_T(L) ``\sim" c(\alpha) \frac{L^2}{\mathrm{area}(T)}$ where $T$ is an isosceles triangle with angles $(\alpha, \alpha, \pi - 2\alpha)$. Then $c(\alpha)$ is discontinuous at every rational multiple of $\pi$ in $(0, \frac{\pi}{2})$ with a discrete set of possible exceptions. A similar statement holds for right triangles.


Theorem \ref{T:Counting:ClosedTrajectories} completes our understanding of the weak asymptotics of $N_P(L)$ for $P$ a rational right or isosceles triangle. Veech \cite[Theorem 1.2]{V} computed the asymptotics of $N_P(L)$ for $P$ an isosceles triangle with angles $( \frac{1}{n}, \frac{1}{n}, \frac{n-2}{n})\pi$ and Eskin-Marklof-Witte-Morris \cite[Theorem 1.4]{EMWM} computed it for $P$ an isosceles triangle with angles $( \frac{n-2}{2n}, \frac{n-2}{2n}, \frac{4}{2n})\pi$ with $n \ne 5$ odd. In both cases, the authors showed exact (not just weak) quadratic asymptotics. The asymptotics of the exceptional right and isosceles triangles not covered by Theorem \ref{T:Counting:ClosedTrajectories} can be derived from these two results. 

%
%



Another question of interest to a billiard player is the number of ways to shoot a billiard ball between two points on a billiard table. Given two distinct points $p_1$ and $p_2$ on a rational polygon $P$, let $N_{P; p_1, p_2}(L)$ denote the number of billiard trajectories of length at most $L$ from $p_1$ to $p_2$.

\begin{thm}\label{T:Counting:GeneralizedDiagonals}
Suppose that $P$ is a rational right triangle with angles $\left( \frac{a}{2n}, \frac{b}{2n}, \frac{1}{2}\right)\pi$ with $\mathrm{gcd}(a,b,2n) =1$ and $\min(a,b) > 2$. If $p_1$ and $p_2$ are distinct points on $P$, then
\[ N_{P; p_1,p_2}(L) ``\sim" \frac{n_1 n_2 c_{p_1, p_2}}{4n^2} \frac{\pi L^2}{\mathrm{area}(P)}.\]
where we make the following definitions: $d_i$ is $a-2$ (resp. $b-2$, resp. $-1$) if $p_i$ is the vertex of angle $\frac{a\pi}{2n}$ (resp. $ \frac{b\pi}{2n}$, resp. $ \frac{\pi}{2}$) and $0$ otherwise; $n_i$ is $1$ if $p_i$ is a vertex whose angle is different from $\frac{\pi}{2}$, $n$ if $p_i$ is any other point on the boundary of $P$, and $2n$ if $p_i$ is any point in the interior of $P$; and 
\[ c_{p_1, p_2} := \frac{(d_1+d_2+2)!!(d_1+1)!!(d_2+1)!!}{(d_1+d_2+1)!!d_1!!d_2!!}\cdot \begin{cases}
                                  \frac{2}{\pi^2} & \text{ $d_1$ and $d_2$ both odd} \\
                                  \frac{1}{2} & \text{ either $d_1$ or $d_2$ even} 
  \end{cases}.\]
In particular, when $p_1$ is not a vertex, 
\[ N_{P; p_1,p_2}(L) ``\sim" \frac{\theta_1 \theta_2 L^2}{2\pi \mathrm{area}(P)}\]
where $\theta_i$ is the amount of angle, measured in radians, in $P$ around $p_i$, i.e. if $p_i$ is an interior point it is $2\pi$, if $p_i$ is a boundary point that is not a vertex, it is $\pi$, and if $p_i$ is a vertex it is the angle of the vertex. 
%
%
\end{thm}

The explicit constants in Theorems \ref{T:Counting:ClosedTrajectories} and \ref{T:Counting:GeneralizedDiagonals} are the Siegel-Veech constants computed in Athreya-Eskin-Zorich \cite{AEZ}. 

Curiously, by Stirling's formula, when $p_1$ and $p_2$ are the two vertices not corresponding to the right angle,
\[ N_{P; p_1, p_2}(L) ``\sim"  \frac{(1+o(1))}{\sqrt{n}} \frac{\mathrm{const}_1(p_1, p_2) \sqrt{\theta_1 \theta_2} L^2}{ \mathrm{area}(P)} \]
where $\mathrm{const}_1(p_1, p_2)$ only depends on the parity of $\frac{2n\theta_1}{\pi}$ and $\frac{2n\theta_2}{\pi}$. Similarly, if $p_1$ is the vertex of angle $\frac{\pi}{2}$ and $p_2$ is a distinct vertex, then 
\[ N_{P; p_1, p_2}(L) ``\sim"  \frac{(1+o(1))}{n} \frac{\mathrm{const}_2(p_2) \theta_2 L^2}{ \mathrm{area}(P)} \]
where $\mathrm{const}_2(p_2)$ only depends on the parity of $\frac{2n\theta_2}{\pi}$. In both cases we see that as the height of the rational numbers that define the vertices of the triangle become larger, there are fewer billiard trajectories between the vertices. 


%

Our last result on billiards in rational right and isosceles triangles is on the \emph{finite blocking problem}, which, given a rational polygon $P$, asks for all pairs of finitely blocked points on $P$. Two (not necessarily distinct) points $p_1$ and $p_2$ are said to be \emph{finitely blocked} if there is a finite collection of points $B \subseteq P - \{p_1, p_2\}$ so that all billiard paths from $p_1$ to $p_2$ pass through points in $B$.

\begin{thm}\label{T:Main:FiniteBlocking}
If $P$ is an unexceptional rational right or isosceles triangle then no two points of $P$ are finitely blocked from each other. 
\end{thm}

A complete classification of finitely blocked pairs of points for every rational right and isosceles triangle is given in Theorem \ref{T:Main:FiniteBlocking:Detailed}.

\subsection{The orbit closures of the unfoldings of rational right and isosceles triangles}

In the sequel translation surfaces will be written as $(X, \omega)$ where $X$ is a Riemann surface and $\omega$ a holomorphic $1$-form. The set of all genus $g$ translation surfaces is stratified by specifying the orders of zeros (with multiplicity) of the holomorphic $1$-form. These strata admit a $\mathrm{GL}(2, \mathbb{R})$ action, generated by Teichm\"uller geodesic flow and complex scalar multiplication. In the sequel, we will use \emph{the orbit closure of $(X, \omega)$} as shorthand for the ``$\mathrm{GL}(2, \mathbb{R})$-orbit closure of $(X, \omega)$ in the stratum containing it". Theorems \ref{T:Counting:ClosedTrajectories},  \ref{T:Counting:GeneralizedDiagonals}, and \ref{T:Main:FiniteBlocking} will all be deduced from the following.

\begin{thm}\label{T:Main:Unfolding}
The unfolding of an exceptional right or isosceles triangle is either Veech or the double cover of a Veech surface. The unfolding of an unexceptional right or isosceles triangle has dense orbit in a locus of holonomy double covers of a component of a stratum of quadratic differentials.
\end{thm}

The orbit closure of every rational right and isosceles triangle is explicitly determined in Theorems \ref{T:Unfolding:RightDetailed} and \ref{T:Unfolding:IsoscelesDetailed} respectively.

A fortiori, Theorem \ref{T:Main:Unfolding} pinpoints which rational right and isosceles triangles unfold to Veech surfaces, which is a result of Kenyon-Smillie \cite{KS}. However, this consequence is more accurately said to be an outcome of the techniques developed by Mirzakhani-Wright \cite{MirWri2}, which are used in the proof. Nonetheless, the algorithm devised in Mirzakhani-Wright \cite{MirWri2} for certifying that the orbit closure of the unfolding of a triangle is full rank is not sufficient to determine which rational right and isosceles triangles have an unfolding that is full rank and is not used in the proof of Theorem \ref{T:Main}.

\subsection{Orbit closures in genus zero}\label{SS:OC}

The main observations that facilitate the proof of Theorem \ref{T:Main:Unfolding} are the following. First, since every isosceles triangle is tiled by two isometric right triangles, it suffices to compute the orbit closure of the unfolding of every rational right triangle. Second, a triangle with angles $\left( \frac{a}{2n}, \frac{b}{2n}, \frac{1}{2}\right)\pi$, where $a, b$ and $n$ are positive integers with $\mathrm{gcd}(a,b,2n) =1$, can be partially unfolded to a quadratic differential in the stratum $\cQ(a-2, b-2, -1^n)$. This stratum is a genus zero stratum with at most two zeros. This motivates the following theorem.

\begin{thm}\label{T:Main}
In genus zero strata of quadratic differentials with two or fewer zeros, the only $\mathrm{GL}(2, \mathbb{R})$ orbit closures of rank at least two are full loci of covers of genus zero strata. 
\end{thm}

Intriguingly, nearly every argument in the proof of Theorem \ref{T:Main} holds for arbitrary genus zero strata. The only essential application of the hypothesis on the number of zeros occurs in Section \ref{S:CylinderRigidRankOne} where it is used to simplify the combinatorics in the classification of cylinder rigid rank one subvarieties. This suggests that the conclusion of Theorem \ref{T:Main} may hold for arbitrary genus zero strata and that the proof given here may be a good template for the general proof. 

By the classification of connected components of strata of Abelian differentials (by Kontsevich-Zorich \cite{KZ}) and quadratic differentials (by Lanneau \cite{Lconn}), the following is an equivalent formulation of Theorem \ref{T:Main}.

\begin{cor}\label{C:Main}
In hyperelliptic connected components of strata of Abelian and quadratic differentials, the only $\mathrm{GL}(2, \mathbb{R})$ orbit closures of rank at least two are full loci of covers of genus zero strata.
\end{cor}

In the case of hyperelliptic components of strata of Abelian differentials, Corollary \ref{C:Main} was the main result of Apisa \cite{Apisa2}. The proof given here provides a new proof of that result using the techniques developed in Apisa-Wright \cite{ApisaWrightDiamonds}, \cite{ApisaWrightGeminal}, \cite{ApisaWrightHighRank} and Apisa \cite{Apisa-MHD}. 

By Chen-Gendron \cite[Theorem 1.1]{ChenGendron}, if a component of a stratum of holomorphic $k$-differentials is a \emph{hyperelliptic component}, i.e. a component consisting of hyperelliptic Riemann surfaces for which the $k$-differential is a $(-1)^k$-eigenform of the hyperelliptic involution, then it consists of double covers of a genus zero stratum of $k$-differentials for which all but at most two cone points have cone angles different from $\pi$. Theorem \ref{T:Main} therefore implies the following.


\begin{cor}\label{C:Main2}
The orbit closure of the unfolding of a holomorphic $k$-differential in a hyperelliptic component is either rank one or a full locus of covers of a genus zero stratum.
\end{cor}

\subsection{Outline of the proof of Theorems \ref{T:Main:Unfolding} and \ref{T:Main}}

We will now outline the strategy of the proof of Theorem \ref{T:Main}. In Section \ref{S:PeriodicStructure} we encode the combinatorics of horizontally periodic translation surfaces in hyperelliptic loci using graphs (this could be seen as an extension of the work of Lindsey \cite{Lindsey}, on hyperelliptic components of strata of Abelian differentials, to all hyperelliptic loci). In Section \ref{S:Joinings}, we show that an invariant subvariety in a product of genus zero strata that surjects onto each factor is, under mild hypotheses and up to scaling each component, a diagonal embedding. In Section \ref{S:CylinderRigidRankOne}, we show (using Section \ref{S:PeriodicStructure}) that the rank one cylinder rigid invariant subvarieties in genus zero strata with at most two zeros are loci of covers of strata of quadratic differentials. This is useful for running an inductive argument since it was shown in Apisa \cite{Apisa-MHD} that every component of the boundary of a rank two rel zero invariant subvariety is a rank one cylinder rigid invariant subvariety. This inductive argument is the content of Section \ref{S:ProofTMain:Rigid} (in the cylinder rigid case) and Section \ref{S:ProofTMain} (in general). The relevance of Section \ref{S:Joinings} is that the boundary of invariant subvarieties often consist of disconnected surfaces. 

Finally the passage from Theorem \ref{T:Main} to Theorem \ref{T:Main:Unfolding} is given in Section \ref{S:ProofTUnfolding}. It relies on the techniques of Mirzakhani-Wright \cite{MirWri2} to show that unexceptional triangles have an orbit closure of rank at least two and to constrain the location of the zeros of the holomorphic 1-forms in the tangent space of the orbit closure. Deep results on the structure of the bundle of first cohomology over an invariant subvariety due to Avila-Eskin-M\"oller \cite{AEM}, Wright \cite{Wfield}, and Filip \cite{Fi2} are all used in this deduction, as is the formula of Forni-Matheus-Zorich \cite{FMZ} for the second fundamental form of the Hodge bundle. 

The deduction of Theorems \ref{T:Counting:ClosedTrajectories}, \ref{T:Counting:GeneralizedDiagonals}, and \ref{T:Main:FiniteBlocking} from Theorem \ref{T:Main:Unfolding} is explained in Subsection \ref{SS:History}

\subsection{Further applications to billiards and to strata of holomorphic $k$-differentials on the sphere}

We will complete our discussion of billiards by considering the following polygons: parallelograms, isosceles trapezoids, right trapezoids, and polygons with at most two angles not integer multiples of $\frac{\pi}{2}$. This is a natural collection of polygons to consider since, along with right and isosceles triangles, they constitute the only polygons that unfold to hyperelliptic translation surfaces (by Apisa \cite[Theorem 1.5]{Apisa-PeriodicPointsG=2}).


\begin{defn}
Given a cyclically ordered set $\theta$ of $k$ positive rational multiples of $\pi$ that sum to $(k-2)\pi$, let $\cM(\theta)$ be the smallest affine invariant subvariety that contains every unfolding of a polygon with vertices whose angles (appearing clockwise) are the entries of $\theta$.
\end{defn}

By Mirzakhani-Wright \cite[Lemma 6.3]{MirWri2}, $\cM(\theta)$ is independent of the order in which the angles appear and almost every polygon with angles $\theta$ has orbit closure $\cM(\theta)$. 


\begin{thm}\label{T:Unfolding:Parallelogram}
If $\theta := \left( \frac{a}{n}, \frac{b}{n}, \frac{a}{n}, \frac{b}{n} \right)\pi$, where $a, b$ and $n$ are positive integers with $\mathrm{gcd}(a,b,n) = 1$ and $a+b = n$, then $\cM(\theta)$ is the quadratic double of $\cQ(2a-2, 2b-2, -1^{2n})$ if $n$ is odd and of $\cQ^{hyp}((a-2)^2, (b-2)^2)$ if $n$ is even. 

If $n \ne 2$,  $3$, $4$, or $6$, then the set of unit area quadrilaterals with angles $\theta$ whose unfolding has orbit closure $\cM(\theta)$ has a discrete complement. Let $Q$ be any polygon with angles $\theta$ whose unfolding has orbit closure $\cM(\theta)$.

If $n \ne 2$, then $Q$ has finitely blocked pairs of points if and only if $\min(a,b) = 1$, in which case the only blocking is between the two vertices of angle $\frac{\pi}{n}$ (resp. that are distinct) when $n$ is even (resp. odd). The blocking set is the barycenter for parallelograms and the two midpoints of edges joining vertices of equal angle for trapezoids. 

If $n \ne 2$ and $\min(a, b) \ne 1$, then, 
\[ N_{Q}(L) ``\sim" \frac{1}{2\pi} \left(1- \frac{(n,2)}{2n} \right)\left( 1 + \frac{1}{ab(n,2)}\right)\frac{L^2}{\mathrm{area}(Q)}. \]
\end{thm}

%
%

In the case that $n = 2$, the unfolding is a torus with four marked points and it is a simple matter to determine the orbit closure. In the case that $n = 3$ or $4$, $\cM(\theta)$ is $\cH(1,1)$ or consists of double covers of $\cH(2)$ (respectively) and so it is possible to determine the orbit closure of every polygon with these angles using the work of McMullen \cite{Mc5}. Moreover, the finite blocking problem can be solved for these polygons using M\"oller \cite{M2} and Apisa \cite{Apisa,Apisa-PeriodicPointsG=2}.

\begin{thm}\label{T:Main:SpecificSphere}
If $\theta := \left( \frac{a}{n}, \frac{b}{n}, \frac{1}{2}, \frac{1}{2} \right)\pi$, where $a, b$ and $n$ are positive integers with $\mathrm{gcd}(a,b,n) = 1$ and $a+b = n$, then $\cM(\theta)$ is the quadratic double of $\cQ(2a-2, 2b-2, -1^{2n})$ if $n$ is odd and of $\cQ(a-2, b-2, -1^n)$ when $n$ is even. 

If $n \ne 2$,  $3$, $4$, or $6$, then the set of unit area right trapezoids with angles $\theta$ whose unfolding has orbit closure $\cM(\theta)$ has a discrete complement. Let $R$ be any polygon with angles $\theta$ whose unfolding has orbit closure $\cM(\theta)$.
 
If $n \ne 2$, then $R$ has finitely blocked pairs of points if and only if $\min(a,b) = 1$, in which case the only such blocking is the vertex of angle $\frac{\pi}{n}$ from itself. The blocking set is empty.

If $n \ne 2$ and $\min(a,b) \ne 1$, then
\[N_{R}(L) ``\sim" \frac{1}{4\pi} \left( 1 - \frac{(n,2)}{2n} \right) \left( 1 + \frac{(n,2)^2}{2ab}\right) \frac{L^2}{\mathrm{area}(R)}. \]
\end{thm}
%
%
%
%
%

As before, when $n \in \{3, 4\}$, the orbit closure of the unfolding of every polygon with angles as in Theorem \ref{T:Main:SpecificSphere} can be deduced from results of McMullen \cite{Mc5}. 

More generally we have the following.

\begin{thm}\label{T:Main:GeneralSphere}
If a holomorphic $k$-differential on a sphere unfolds to a hyperelliptic curve, then one of the following occurs:
\begin{enumerate}
    \item The holonomy cover of almost every $k$-differential in the stratum has dense orbit in the hyperelliptic locus containing it.
    \item The cone angles of the $k$-differential are either $\left( \frac{a}{n}, \frac{a}{n}, \frac{b}{n}, \frac{b}{n} \right)2\pi$ or $\left( \frac{a}{n}, \frac{a}{n}, \frac{b}{n} \right)2\pi$ where $\mathrm{gcd}(a,b,n) = 1$ and $n$ is even. Then the holonomy cover of almost every $k$-differential in the stratum has dense orbit in a quadratic double of a hyperelliptic component of a stratum.
    \item The $k$-differential is the pillowcase double of an exceptional right or isosceles triangle.
\end{enumerate}
In particular, the unfolding of almost every rational polygon with at least five sides and at most two angles that are not multiples of $\frac{\pi}{2}$ has dense orbit in the hyperelliptic locus containing it.
\end{thm}

Work of Athreya-Eskin-Zorich \cite{AEZ} imply counting formulas, as in Theorems \ref{T:Counting:ClosedTrajectories} and \ref{T:Counting:GeneralizedDiagonals}, for a full measure set of polygons with angles described by Theorem \ref{T:Main:GeneralSphere}. Similarly a solution to the finite blocking problem can be deduced as in Theorem \ref{T:Main:FiniteBlocking}.

\subsection{History}\label{SS:History}

The existence of periodic orbits in rational polygons was established by Masur \cite{Masur-cylinder}, who deduced this fact from the existence of cylinders on every translation surface. Masur \cite{Masur-LowerBounds} \cite{Masur-UpperBounds} also showed that the number of cylinders of length at most $L$ was bounded below and above by a positive constant multiple of $L^2$. 

In Eskin-Masur \cite{EMa}, these quadratic upper and lower bounds were improved by showing that, given an $\mathrm{SL}(2, \mathbb{R})$-invariant probability measure $\mu$ on a stratum of translation surfaces, there are constants $c_{cyl, \mu}$ and $c_{sc, \mu}$ so that for $\mu$-almost every translation surface $(X, \omega)$, the number of cylinders (resp. saddle connections) of length at most $L$ is exactly asymptotic to $c_{cyl, \mu} L^2$ (resp. $c_{sc, \mu} L^2$). These formulas are called \emph{Siegel-Veech formulas}, see also Veech \cite{V4}. In Eskin-Mirzakhani \cite{EM} these probability measures were classified and shown to be supported on invariant subvarieties. Using the Siegel-Veech formulas and the measure classification in Eskin-Mirzakhani, Eskin-Mirzakhani-Mohammadi \cite{EMM} showed that every $\mathrm{GL}(2, \mathbb{R})$ orbit closure of a translation surface $(X, \omega)$ is an invariant subvariety and that this orbit closure completely determines the weak asymptotics of the number of cylinders and saddle connections on $(X, \omega)$ (see \cite[Theorem 2.12]{EMM}). This explains how the weak asymptotics in Theorems \ref{T:Counting:ClosedTrajectories} and \ref{T:Counting:GeneralizedDiagonals} can be deduced from computing the orbit closure of the unfolding in Theorem \ref{T:Main:Unfolding}.

In Eskin-Masur-Zorich \cite{EMZboundary}, the notion of a \emph{configuration} was defined in order to refine the Siegel-Veech formulas to allow for the counting of specific types of cylinders and saddle connections. Moreover, the authors related Siegel-Veech constants for surfaces with dense orbit in strata to ratios of Masur-Veech volumes\footnote{Masur-Veech volumes are finite by work of Masur \cite{Ma2} and Veech \cite{V2}; Goujard has explicitly computed them \cite{Goujard-Volume} and computed associated Siegel-Veech constants \cite{Goujard-SiegelVeech}; Aggarwal has computed the large-genus asymptotics of volumes \cite{Aggarwal-Volume} and Siegel-Veech constants \cite{Aggarwal-SiegelVeech}; Chen-M\"oller-Sauvaget-Zagier\cite{ChenMollerSauvagetZagier} have related these volumes to intersection theory.}. Boissy \cite{Boissy} and Masur-Zorich \cite{MZ} computed the configurations present on genus zero and general quadratic differentials respectively. Combined with the formulas for sums of Lyapunov exponents (and their connection to area Siegel-Veech constants) in Eskin-Kontsevich-Zorich \cite{EKZbig}, this body of work enabled Athreya-Eskin-Zorich \cite{AEZ} to compute Siegel-Veech constants for configurations on genus zero quadratic differentials. These constants are precisely the ones that appear in the formulas in Theorem \ref{T:Counting:ClosedTrajectories} and \ref{T:Counting:GeneralizedDiagonals}.

\subsubsection*{Orbit Closures and Compactifications}

Unlike McMullen’s work in genus two \cite{Mc5} or Apisa \cite{Apisa-Rk1Hyp}, the classification in Theorem \ref{T:Main} does not classify rank one orbit closures. An essential tool in its proof is the use of cylinder degenerations to pass to the boundary of an invariant subvariety. The utility of cylinder degenerations for probing an invariant subvariety was observed by Wright \cite{Wcyl}, whose work is based on Smillie-Weiss' work on horocycle flow \cite{SW2}. 

To run the inductive argument in the proof of Theorem \ref{T:Main}, we employ the theory of the boundary of an invariant subvariety developed, principally, in Mirzakhani-Wright \cite{MirWri}, Chen-Wright \cite{ChenWright}, and Bainbridge-Chen-Gendron-Grushevsky-M\"oller \cite{BCGGM-Incidence2}. 

\subsubsection*{The Finite Blocking Problem}

Many results on the finite blocking problem in translation surfaces were obtained for surfaces with nontrivial Veech groups in Gutkin \cite{GutkinBlocking1} \cite{GutkinBlocking2}, Hubert-Schmoll-Troubtezkoy \cite{HST}, and Monteil \cite{Mont1} \cite{Mont2}. In Leli\`evre-Monteil-Weiss \cite{LMW}, it was shown that almost every point illuminates almost every other. In Apisa-Wright \cite{ApisaWright}, it was shown that the finite blocking problem on a translation surface can be solved by determining the minimal half-translation cover and the set of periodic points (see Section \ref{S:FB} for details and definitions). Since Apisa-Wright \cite{ApisaWright} classified periodic points for strata of quadratic differentials and Apisa-Saavedra-Zhang \cite{ApisaSaavedraZhang} did the same for the Veech $n$-gon and double $n$-gon loci, Theorem \ref{T:Main:FiniteBlocking} follows from Theorem \ref{T:Main:Unfolding}. Periodic points were also computed in Apisa \cite{Apisa}, M\"oller \cite{M2}, and computed and connected to billiards in Apisa \cite{Apisa-PeriodicPointsG=2}.

A brief proof determining which vertices in rational right and isosceles triangles do not illuminate themselves (i.e. are not finitely blocked from themselves with empty blocking set) appears in Tokarsky-Boese \cite{TokarskyBoese}.

\subsubsection*{Billiards in rational polygons}

One of the earliest sources of inspiration in the study of billiards in triangles was Veech \cite{V}, who discovered an infinite family of triangles that unfold to Veech surfaces. Ward \cite{W}, Kenyon-Smillie \cite{KS}, and Hooper \cite{HooperVeechTriangle} discovered new Veech triangles. A new perspective on the Veech-Ward triangles was provided by Bouw-Moller \cite{BM} who also presented a family of quadrilaterals that unfold to Veech surfaces. Finally, Eskin-Marklof-Witte-Morris \cite{EMWM} studied a collection of isosceles triangles whose unfolding is not Veech, but which cover Veech surfaces.

Billiard dynamics in right-angled rational polygons were further studied in McMullen \cite{Mc} and Athreya-Eskin-Zorich \cite{AEZ}. A rich new family of polygons with surprisingly small orbit closures were constructed in Eskin-McMullen-Mukamel-Wright \cite{EMMW}. 

In Kenyon-Smillie \cite{KS} and Puchta \cite{Pu}, the acute, right, and isosceles triangles that unfold to Veech surfaces were computed. Using work of Mirzakhani-Wright \cite{MirWri2}, Larsen-Norton-Zykoski \cite{LarsenNortonZykoski} classified all the obtuse triangles with largest angle at least 135 degrees that have a Veech unfolding. 

\subsubsection*{Irrational Billiards}

While much less is known about irrational billiards, some results are still available for irrational right and isosceles triangles. For instance, Schwartz-Hooper \cite{SchwartzHooper} found that all triangles sufficiently close to isosceles triangles contain a periodic orbit and Hooper \cite{HooperRightTriangles} showed that any periodic trajectory on a right triangle is unstable, i.e. fails to persist on arbitrarily nearby triangles. For more work on billiards in irrational right triangles, see Cipra-Hanson-Kolan \cite{CipraHansonKolan}, Galperin-Stepin-Vorobets \cite{GalperinStepinVorobets}, Galperin-Zvonkine \cite{GalperinZvonkine}, and Troubetzkoy \cite{Troubetzkoy-RightTriangles}. 

\bold{Acknowledgments.}  
The author is grateful to Vincent Delecroix and Julian Rueth for generously performing computer experiments on the orbit closures of unfoldings of rational right triangles. During the preparation of this paper, the  author was partially supported by NSF Postdoctoral Fellowship DMS 1803625.


%
%
%
%

\section{The structure of horizontally periodic surfaces in hyperelliptic loci}\label{S:PeriodicStructure}

Throughout this section we will suppose that $\cL$ is a component of a hyperelliptic locus in a component $\cH$ of a stratum of Abelian differentials. Let $k$ denote the codimension of $\cL$ in $\cH$. Let $(X, \omega) \in \cL$ be any horizontally periodic surface in $\cL$. Let $J$ denote the hyperelliptic involution on $(X, \omega)$. We will permit $(X, \omega)$ to have marked points. If a fixed point of $J$ is marked, then we will say that its image on $(X, \omega)/J$ is a \emph{marked pole}.  

\begin{lem}\label{L:Codim=Zeros}
$k+1$ is the number of zeros, marked points, and marked poles on $(X, \omega)/J$. 
\end{lem}
\begin{proof}
Let $z_{odd}$ (resp. $z_{even}$) denote the number of odd (resp. even) order zeros, marked points, and marked poles on $(X, \omega)/J$. Since $\cL$ is full rank, its codimension is 
\[ \mathrm{rel}(\cH) - \mathrm{rel}(\cL) = (z_{odd} + 2z_{even}-1) - z_{even} = z_{odd} + z_{even} -1. \] (See Mirzakhani-Wright \cite{MirWri2} for a discussion of ``full rank" and Apisa-Wright \cite[Lemma 4.2]{ApisaWright} for the computation of $\mathrm{rel}(\cL)$). 
\end{proof}

\begin{defn}\label{D:CutReglue}
Given a pair of parallel cylinders $C_1$ and $C_2$ of equal circumference on a translation surface we will say that \emph{cutting and regluing $C_1$ and $C_2$} is the translation surface that results from deleting the core curves $\gamma_1$ and $\gamma_2$ of the cylinders $C_1$ and $C_2$ to form a translation surface with boundary and then identifying the boundary corresponding to $\gamma_1$ with the boundary corresponding to $\gamma_2$ to form a new translation surface (see Figure \ref{F:cut-reglue} for an illustration).  Since there is not a canonical way to reglue the boundaries, the surfaces formed from cutting and regluing are only well-defined up to shearing the cylinders formed from $C_1$ and $C_2$ on the reglued surfaces. 
\end{defn}

\begin{figure}[h]\centering
\includegraphics[width=0.75\linewidth]{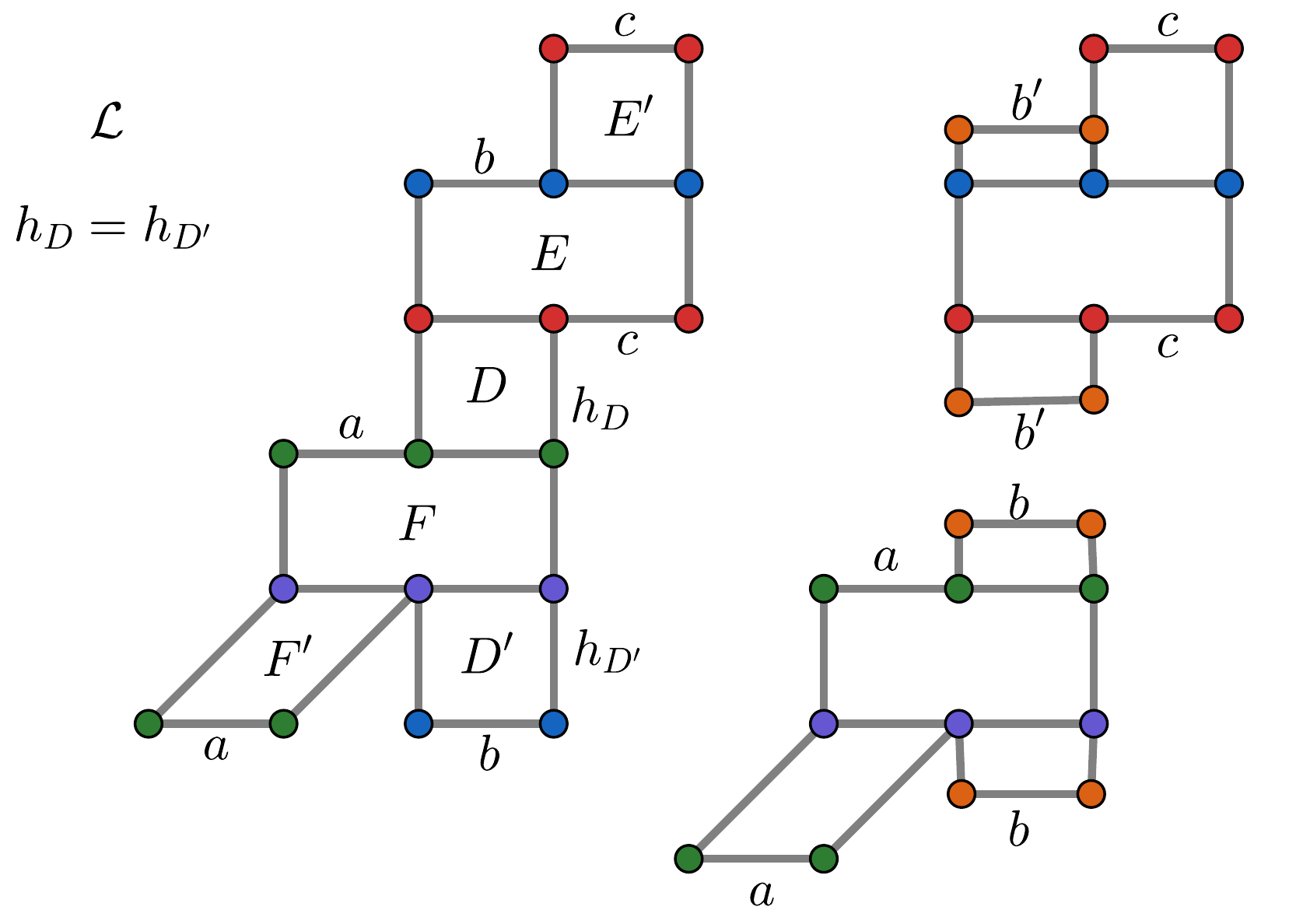}
\caption{The surface on the left belongs to a codimension one hyperelliptic locus $\cL \subseteq \cH(1^4)$ defined in period coordinates by the equation $h_{D} = h_{D'}$. The disconnected surface on the right is the result of cutting and regluing $\{D, D'\}$. A basis for the subspace of rel in the twist space of the horizontal cylinders is given by $\frac{\gamma_D^* + \gamma_{D'}^*}{2} - \gamma_F^* + \gamma_{F'}^*$ and $\frac{\gamma_D^* + \gamma_{D'}^*}{2} - \gamma_E^* + \gamma_{E'}^*$.}
\label{F:cut-reglue}
\end{figure}

\begin{lem}\label{L:CutReglue}
Cutting and regluing pairs of cylinders exchanged by $J$ produces a pair of surfaces in hyperelliptic loci of codimensions $k_1$ and $k_2$ where $k_1 + k_2 + 1 = k$.
\end{lem}
\begin{proof}
Two cylinders exchanged by the hyperelliptic involution on $(X, \omega)$ correspond to a separating cylinder on $(X, \omega)/J$. Therefore, cutting and regluing produces two surfaces $(X_i, \omega_i)$ with a hyperelliptic involution $J_i$ for $i \in \{1, 2\}$. If there are $z_i$ zeros, marked points, and marked poles on $(X_i, \omega_i)/J$, then $z_1 + z_2$ is the number of such points on $(X, \omega)/J$. The result now follows from Lemma \ref{L:Codim=Zeros}. 
%
%
\end{proof}

\begin{defn}\label{D:Graphlike}
Given a horizontally periodic surface with horizontal cylinders $\bfC$, the \emph{cylinder graph} is the directed graph whose vertices correspond to cylinders in $\bfC$ and where there is a directed edge between $C$ and $C'$ if the top boundary of $C$ shares a saddle connection with the bottom boundary of $C'$. If $C$ and $C'$ are horizontal cylinders, let $d(C, C')$ be the distance in the cylinder graph between $C$ and $C'$. 

Say that a horizontally periodic surface is \emph{graph-like} if there is an involution $J'$ that fixes each horizontal cylinder. The terminology is chosen since the directed cylinder graph of a graph-like surface can be taken to be undirected since $J'$ induces a bijection between edges joining vertices $v$ to $w$ and edges joining $w$ to $v$.

In the sequel, given a cylinder $C$, we will let $\gamma_C$ denote its (oriented) core curve and $\gamma_C^*$ its Poincare dual. If we are simultaneously considering a collection of parallel cylinders, all core curves will be assumed to be oriented in the same direction.
\end{defn}

%
%
%

\begin{lem}\label{L:GraphlikeRel}
If $(Y, \eta)$ is a graph-like surface in a stratum $\cH$, then $\Twist( (Y, \eta), \cH)$ contains a rel subspace of dimension at most one. 

If this rel subspace is nonempty and $\bfC$ denotes the horizontal cylinders on $(Y, \eta)$, then the rel subspace of $\Twist( (Y, \eta), \cH)$ is generated by 
\[ \sum_{C \in \bfC} (-1)^{d(C_0, C)} \gamma_C^* \]
where $C_0$ is a fixed cylinder in $\bfC$.
\end{lem}

Given an equivalence class of cylinders $\bfC$ on a translation surface $(Y, \eta)$ in an invariant subvariety $\cM$, recall that the twist space of $\bfC$, denoted by $\Twist(\bfC, \cM)$ and defined in Apisa \cite[Definition 2.2]{Apisa-MHD}, is defined to be the collection of linear combinations of Poincare duals of core curves of cylinders in $\bfC$ that belong to $T_{(Y, \eta)} \cM$. If $(Y, \eta)$ is horizontally periodic, $\Twist( (Y, \eta), \cM)$ denotes $\Twist(\bfC, \cM)$ where $\bfC$ is the collection of horizontal cylinders on $(Y, \eta)$. The \emph{support} of an element $v \in \Twist(\bfC, \cM)$ is the collection of cylinders $C \in \bfC$ so that $\gamma_C^*$ has nonzero coefficient in $v$.


\begin{proof}
Let $v = \sum_{C \in \bfC} a_C \gamma_C^*$ be a rel vector in $\Twist( (Y, \eta), \cH)$. Suppose that $C$ and $C'$ are two cylinders in $\bfC$ that share a boundary saddle connection. Since $(Y, \eta)$ is graph-like, $C$ shares a saddle connection with $C'$ in both its top and bottom boundary. Therefore, there is a simple closed curve $\gamma$ contained in $\overline{C \cup C'}$ that intersects the core curve of $C$ and $C'$ exactly once. Since $v$ is rel, $0 = v(\gamma) = a_C + a_{C'}$, i.e. $a_C = -a_{C'}$. Therefore, $v$ is a multiple of the indicated rel vector. 
\end{proof}

\begin{defn}\label{D:CodimesionTreeGraph}
The following data will be called a \emph{(codimension $k$) tree of graph-like surfaces}:
\begin{enumerate}
    \item A connected tree $T$ with $n$ vertices each of which is assigned a non-negative integer $\{k_1, \hdots, k_n\}$ so that $\sum_{i=1}^n (k_i+1)  = k+1$.
    \item A horizontally periodic graph-like surface $(X_i, \omega_i)$ in a codimension $k_i$ hyperelliptic locus associated to each vertex $i$.
    \item An edge joining vertex $i$ to vertex $j$ corresponds to a pair of horizontal cylinders of equal circumference on $(X_i, \omega_i)$ and $(X_j, \omega_j)$. Each cylinder on $(X_i, \omega_i)$ appears in at most one edge. 
\end{enumerate}
To each horizontally periodic surface $(X, \omega)$ in a codimension $k$ hyperelliptic locus we can associate a codimension $k$ tree of graph-like surfaces as follows:
\begin{enumerate}
    \item Cut and reglue all pairs of horizontal cylinders exchanged by the hyperelliptic involution to form a collection of surfaces $(X_i, \omega_i)$ for $i \in \{1, \hdots, n\}$. These surfaces belong to codimension $k_i$ hyperelliptic loci and $\sum_{i=1}^n (k_i+1)  = k+1$ by Lemma \ref{L:Codim=Zeros}. Moreover, these surfaces are graph-like since every horizontal cylinder is fixed by the hyperelliptic involution. 
    \item Let $T$ be a graph whose vertices correspond to the surfaces $(X_i, \omega_i)$ for $i \in \{1, \hdots, n\}$. Add an edge from $(X_i, \omega_i)$ to $(X_j, \omega_j)$ if they contain horizontal cylinders $C_i$ and $C_j$ that were formed by cutting and regluing a pair of horizontal cylinders on $(X, \omega)$. Since cylinders exchanged by the hyperelliptic involution are homologous, cutting any edge in the graph $T$ disconnects the graph, implying that $T$ is a tree. 
\end{enumerate}
In the construction above, the surfaces $(X_i, \omega_i)$ are well-defined only up to shearing the cylinders associated to the edges of the tree (see the final sentence in Definition \ref{D:CutReglue}). 

It may be useful to refer to Figure \ref{F:cut-reglue} where the horizontally periodic surface on the left is associated to a tree with two vertices - one for each of the two surfaces on the right.
\end{defn}

\subsection{Rel in the twist space of a horizontally periodic surface in a hyperelliptic locus}

In this subsection we will study the rel subspace of the twist space of a horizontally periodic surface $(X, \omega)$ in a hyperelliptic locus $\cL$. Let $T$ be the tree of graph-like surfaces associated to $(X, \omega)$ (as in Definition \ref{D:CodimesionTreeGraph}) and let $V(T)$ denote its set of vertices. Let $\bfC$ be the collection of horizontal cylinders on $(X, \omega)$.

\begin{lem}\label{L:HypTwist}
The element $\sum_{C \in \bfC} a_C \gamma_C^*$, where $a_C \in \mathbb{C}$, belongs to $\Twist((X, \omega), \cL)$ if and only if $a_C = a_{C'}$ whenever $C$ and $C'$ are exchanged by the hyperelliptic involution.
\end{lem}
\begin{proof}
If $\sum_{C \in \bfC} a_C \gamma_C^*$ belongs to $\Twist((X, \omega), \cL)$ and $B$ and $B'$ are two horizontal cylinders exchanged by the hyperelliptic involution, then, under any sufficiently small deformation, these cylinders must have the same heights. This implies that $a_B = a_{B'}$.

Conversely, if $B$ is a horizontal cylinder fixed by the hyperelliptic involution, then it is clear that $i\gamma_B^*$ belongs to $\Twist((X, \omega), \cL)$  since the hyperelliptic involution extends to the surface $(X, \omega) + it\gamma_B^*$ for sufficiently small $t$. Similarly, if $B$ and $B'$ are exchanged by the hyperelliptic involution, $i(\gamma_B^* + \gamma_{B'}^*)$ belongs to $\Twist((X, \omega), \cL)$. This implies that any element of the form $\sum_{C \in \bfC} a_C \gamma_C^*$ belongs to $\Twist((X, \omega), \cL)$ as long as $a_B = a_{B'}$ whenever $B$ and $B'$ are exchanged by the hyperelliptic involution.
\end{proof}

For each vertex $i$ of $T$, let $(X_i, \omega_i)$ be the corresponding surface in hyperelliptic locus $\cL_i$. Let $\bfC^F_i$ be the horizontal cylinders on $(X_i, \omega_i)$ that are not part of any edge relation in $T$. We may think of these cylinders as belonging to either $(X_i, \omega_i)$ or to $(X, \omega)$. Let $\bfC^{NF}_i$ be the horizontal cylinders on $(X, \omega)$ that correspond to edges in $T$ that have $i$ as an endpoint. (The superscripts ``F" and ``NF" indicate whether the cylinders are fixed or not fixed by the hyperelliptic involution.) We can identify each pair of cylinders exchanged by the hyperelliptic involution in $\bfC^{NF}_i$ with a cylinder in $(X_i, \omega_i)$. Let $d_i(\cdot, \cdot)$ denote the distance between two cylinders in the cylinder graph of $(X_i, \omega_i)$. Now if $\Twist( (X_i, \omega_i), \mathcal{L}_i)$ has rel, then define the following vector in $\Twist( (X, \omega), \cL)$, 
\[ \rho_i := \frac{1}{2} \sum_{C \in \bfC_i^{NF}} (-1)^{d_i(C_0, C)} \gamma_C^* + \sum_{C \in \bfC_i^{F}} (-1)^{d_i(C_0, C)} \gamma_C^*  \]
where $C_0$ is an arbitrary horizontal cylinder on $(X_i, \omega_i)$. If $\Twist( (X_i, \omega_i), \mathcal{L}_i)$ has no rel, set $\rho_i := 0$.



\begin{lem}\label{L:RelInHyp}
The nonzero elements of $\{\rho_i\}_{i \in V(T)}$ form a basis for the rel subspace of $\Twist( (X, \omega), \cL)$.
\end{lem}

It may be useful to refer to Figure \ref{F:cut-reglue} where we have already noted that, for the surface on the left, $\{\rho_1, \rho_2\}$ is a basis of the rel subspace of the twist space.

\begin{proof}
We will prove the statement by induction on the number $m$ of pairs of horizontal cylinders that are exchanged by the hyperelliptic involution. The base case of $m = 0$ is Lemma \ref{L:GraphlikeRel}. Suppose now that $m > 0$. 

Up to re-indexing let $(X_1, \omega_1)$ be a leaf of $T$. Let $\{B, B'\}$ be the pair of cylinders exchanged by the hyperelliptic involution on $(X, \omega)$ that correspond to the edge whose endpoint is the leaf. Note that $(X, \omega) - \{B, B'\}$ has two components, one of which corresponds to $(X_1, \omega_1)$. The horizontal cylinders on this (resp. the other) component we denote by $\bfC_1$ (resp. $\bfC_2$). After cutting and regluing $\{B, B'\}$, let $S_1$ and $S_2$ denote respectively the resulting translation surfaces in hyperelliptic loci $\cL_i$. We emphasize that $B$ and $B'$ do not belong to $\bfC_1 \cup \bfC_2$.

Recall that every element of $\Twist( (X, \omega), \cL)$ is a linear combination of elements of $\{\gamma_C^*\}_{C \in \bfC}$. Define a map $p_i: \Twist( (X, \omega), \cL) \ra \Twist(S_i, \cL_i)$ given by 
\[ p_i\left( \sum_{C \in \bfC} a_C \gamma_C^* \right) = \sum_{C \in \bfC_i} a_C \gamma_C^* \]
where the cylinders in $\bfC_i$ are taken to belong to $(X, \omega)$ on the left hand side and to $S_i$ on the right. Notice that Lemma \ref{L:HypTwist} implies that $p_i$ has the indicated domain and codomain.

Let $D_i$ be the horizontal cylinder on $S_i$ coming from cutting and regluing $\{B, B'\}$. Notice that $S_i - \gamma_{D_i}$ is a subsurface of both $S_i$ and $(X, \omega)$. By the Mayer-Vietoris sequence, 
\[ H_1(S_i) = H_1(S_i - \gamma_{D_i}) + \mathbb{Z}\eta_{D_i} \] where $\eta_{D_i}$ is any simple closed curve so that $\gamma_{D_i}^*(\eta_{D_i}) = 1$. Since $\eta_{D_i}$ is a simple closed curve on $(X_i, \omega_i)$ it corresponds to a simple arc on $(X, \omega)$ that joins the core curve of $B$ to the core curve of $B'$. Let $\eta$ denote the concatenation of these two arcs (possibly after adding in a segment of the core curves of $B$ and $B'$ in order to ensure that the endpoints of the two arcs agree).  

We additionally define a map $\pi_i: \Twist( (X, \omega), \cL) \ra \Twist(S_i, \cL_i)$ by 
\[ \pi_i(v) = p_i(v) - \left( p_i(v)\right)(\eta_i)\gamma_{D_i}^*. \]
Lemma \ref{L:HypTwist} again implies that $\pi_i$ has the indicated domain and codomain.

Define $R$ (resp. $R_i$) to be the rel subspace of $\Twist((X, \omega), \cL)$ (resp. $\Twist(S_i, \cL_i)$). We will show that $\pi_1 \oplus \pi_2$ induces an isomorphism between $R$ and $R_1 \oplus R_2$. 

\begin{sublem}\label{SL:PiStillRel}
If $v \in R$, then $\pi_i(v) \in R_i$. 
\end{sublem}
\begin{proof}
Let $\gamma$ be any closed curve in $S_i - \gamma_{D_i}$, which we can identify with a subsurface of $X$. Since the only curves in $\{\gamma_C\}_{C \in \bfC}$ that $\gamma$ intersects are core curves of cylinders in $\bfC_i$, we have
\[ \left(\pi_i(v)\right)(\gamma) = v(\gamma) = 0. \]
Moreover, 
\[  \left(\pi_i(v)\right)(\eta_i) = \left(p_i(v)\right)(\eta_i) - \left(p_i(v)\right)(\eta_i) \gamma_{D_i}^*(\eta_i) = 0.\]
Since, by the Mayer-Vietoris sequence, 
\[ H_1(S_i) = H_1(S_i - \gamma_{D_i}) + \mathbb{Z}\eta_{D_i} \]
it follows that since $v$ belongs to the rel subspace, so does $\pi_i(v)$. 
\end{proof}

We will define one final map, $\iota_i: \Twist(S_i, \cL_i) \ra \Twist((X, \omega), \cL)$, as follows:
\[ \iota_i\left( b\gamma_{D_i}^* + \sum_{C \in \bfC_i} a_C \gamma_C^* \right) = \frac{b}{2}\left( \gamma_B^* + \gamma_{B'}^* \right) + \sum_{C \in \bfC_i} a_C \gamma_C^*.  \]
where the cylinders in $\bfC_i$ belong to $S_i$ on the left-hand side and $X$ on the right. As above, Lemma \ref{L:HypTwist} implies that $\iota_i$ has the indicated domain and codomain.

\begin{sublem}\label{SL:IotaStillRel}
If $v \in R_i$, then $\iota_i(v) \in R$. 
\end{sublem}
\begin{proof}
The proof will be quite similar to that of Sublemma \ref{SL:PiStillRel}. By the Mayer-Vietoris sequence, $H_1(X)$ is contained in $H_1(S_1 - \gamma_{D_1}) + H_1(S_2 - \gamma_{D_2}) + \mathbb{Z}\eta$. As in Sublemma \ref{SL:PiStillRel}, any closed curve in $S_j - \gamma_{D_j}$ for $j \in \{1, 2\}$ pairs trivially with $\iota_i(v)$, so it suffices to consider the pairing of $\iota_i(v)$ and $\eta$. By construction, $\gamma_B^*(\eta) = \gamma_{B'}^*(\eta) = 1$. Since the restriction of $\eta$ to $S_i - \gamma_{D_i}$ can be identified with $\eta_i$ we have that $(\iota_i(v))(\eta) = v(\eta_i) = 0$. Therefore, $\iota_i(v) \in R$ as desired.
\end{proof}

By Sublemmas \ref{SL:PiStillRel} and \ref{SL:IotaStillRel} we have two linear maps $\pi = \pi_1 \oplus \pi_2: R \ra R_1 \oplus R_2$ and $\iota = \iota_1 + \iota_2: R_1 \oplus R_2 \ra R$. 

\begin{sublem}
$\pi$ and $\iota$ are inverses of each other. 
\end{sublem}
\begin{proof}
Fix $v \in R$. Notice that the coefficients of $\{\gamma_C^*\}_{C\in \bfC_1 \cup \bfC_2}$ are the same for $v$ and $\iota(\pi(v))$. Since $v$ and $\iota(\pi(v))$ belong to the rel subspace, and hence are zero in absolute cohomology, these coefficients determine the coefficients of $\gamma_B^*$ and $\gamma_{B'}^*$ (which are equal to each other by Lemma \ref{L:HypTwist}). Therefore, $v = \iota(\pi(v))$. The argument that $\pi\circ\iota$ is the identity is identical. 
\end{proof}

Therefore, if $\beta$ is a basis of $R_2$, then $\{ \iota_1(\rho_1) \} \cup \iota_2(\beta)$ is a basis of $R$. The induction hypothesis applied to $S_2$ now allows us to conclude.  
\end{proof}

\begin{lem}\label{L:NonAdjacentHomologous}
Suppose that $\cL$ is a codimension one hyperelliptic locus. If $\Twist((X, \omega), \cL)$ contains rel supported on every horizontal cylinder and $C$ and $C'$ are two horizontal cylinders exchanged by the hyperelliptic involution that share a boundary saddle connection that belongs to the top boundary of $C$, then the top boundary of $C$ coincides with the bottom boundary of $C'$. 
\end{lem}
\begin{proof}
Cut and reglue $(X, \omega)$ along $\{C, C'\}$. Let $(X_0, \omega_0)$ be the resulting component that contains the top boundary of $C$ and let $D$ denote the cylinder on $(X_0, \omega_0)$ that was formed by cutting and regluing $\{C, C'\}$. 

Since $\cL$ is codimension one, it follows that $(X_0, \omega_0)$ belongs to a codimension zero hyperelliptic locus (by Lemma \ref{L:CutReglue}). In particular, $(X_0, \omega_0)$ is graph-like and so, if its twist space contains rel, the rel vector is given by Lemma \ref{L:GraphlikeRel}; however this vector cannot be rel since $D$ contains a saddle connection on its top and bottom boundary (since $C$ and $C'$ share a saddle connection that belongs to the top boundary of $C$). 

Since $\Twist((X, \omega), \cL)$ contains rel supported on every horizontal cylinder, it follows that the twist space of $(X_0, \omega_0)$ also contains rel unless $\overline{D} = (X_0, \omega_0)$ (by Lemma \ref{L:RelInHyp}). (This uses that, since $\cL$ is a codimension one locus, at most one cylinder on $(X_0, \omega_0)$ can be part of an edge relation). We now conclude by noting that, since $\overline{D} = (X_0, \omega_0)$, the top boundary of $C$ coincides with the bottom boundary of $C'$. 
%
\end{proof}

\subsection{The boundary of hyperelliptic loci}

We conclude this section with a useful fact about the boundary of hyperelliptic loci in the Mirzakhani-Wright partial compactification (see \cite{MirWri}).

\begin{prop}\label{P:HypBdry}
Any component of the boundary of a codimension $k$ hyperelliptic locus $\cL$ is contained in a product of hyperelliptic loci of codimension at most $k$.
\end{prop}
\begin{proof}
Let $(X_n, \omega_n)$ be a sequence in $\cL$ converging to a limit $(X, \omega)$ in the Mirzakhani-Wright partial compactification. By Chen-Wright \cite[Lemma 2.1]{ChenWright}, we may pass to a subsequence to assume that $(X_n, \omega_n)$ converges to $(X', \omega')$ in the Deligne-Mumford compactification of the Hodge bundle over $\cL$ and so that $(X, \omega)$ is precisely the union of the components of $(X', \omega')$ on which $\omega'$ is nonzero. Let $\cH_g$ denote the hyperelliptic locus in the moduli space $\cM_g$ of genus $g$ Riemann surfaces. The Deligne-Mumford compactification of $\cH_g$ is isomorphic to $\overline{\cM_{0, 2g+2}}$ and consists precisely of stable curves that are double covers of stable rational curves (see Avritzer-Lange \cite[Corollary 2.5]{AvritzerLange}). In particular, if $(X_n, \omega_n)$ (resp. $(X', \omega')$) is the holonomy double cover of the (resp. stable) quadratic differential $(\mathbb{P}^1, q_n)$ (resp. $(Y, q)$), then $(\mathbb{P}^1, q_n)$ converges in the Deligne-Mumford compactification to $(Y, q)$. 

By lifting to the incidence variety compactification (see \cite{BCGGM-Incidence2}) and passing to a subsequence, we may suppose that $(\mathbb{P}^1, q_n)$ converges to a twisted quadratic differential $(Y, \eta)$ where $\eta$ coincides with $q$ on any irreducible component where $q$ is nonzero (by \cite[Theorem 1.5]{BCGGM-Incidence2}). 




Let $M$ denote the number of elements of $D(q_n)$, the set of zeros, poles, and marked points of $(\mathbb{P}^1, q_n)$. One advantage of passing to the incidence variety compactification is that the sequence $(\mathbb{P}^1 - D(q_n))$ converges in $\overline{\cM_{0, M}}$ to $Y - D$ where $D$ consists of the zeros, poles, and marked points of $\eta$ not occurring at nodes. Therefore, $Y$ is a union of $s$ marked spheres $Y_1, \hdots, Y_s$ joined along $s-1$ nodes, each of which is separating. The divisor of $\eta \restriction_{Y_i}$ is supported on points in $D \cup N$ where $N$ consists of the nodes of $Y$ (by \cite[Definition 3.2 (0)]{BCGGM-Incidence2}). 


\begin{sublem}\label{SL:SeparatingNode}
Fix $i$. Let $p$ be a node of $Y_i$ and let $Y^c_i$ denote the component of $Y-\{p\}$ that does not contain the subsurface corresponding to $Y_i$. 

If the order of vanishing of $\eta \restriction_{Y_i}$ at $p$ is at least $1$, then $Y^c_i$ contains a zero that belongs to $D$. 

If $q$ is nonvanishing on $Y_i$ and $\mathrm{ord}_p(\eta \restriction_{Y_i}) \in \{0, -1\}$, then if $p$ is marked on $(Y_i, q \restriction_{Y_i})$ , $Y^c_i$ contains a zero, marked point, or marked pole that belongs to $D$.
\end{sublem}
\begin{proof}
Suppose that $p$ connects $Y_i$ to $Y_j$. Suppose first that $\mathrm{ord}_p(\eta \restriction_{Y_i}) \geq 1$. By \cite[Definition 3.2 (1)]{BCGGM-Incidence2}, $\mathrm{ord}_p(\eta \restriction_{Y_i}) + \mathrm{ord}_{p}(\eta \restriction_{Y_j}) = -4$. Therefore, $\mathrm{ord}_{p}(\eta \restriction_{Y_j}) \leq -5$ and hence either $Y_j$ contains a zero located at either a point in $D$, in which case we are done, or at a node different from $p$, in which case we iterate the argument until we find the desired zero.

Suppose now that $q$ is nonvanishing on $Y_i$ and $\mathrm{ord}_p(\eta \restriction_{Y_i}) \in \{0, -1\}$. By definition of the Mirzakhani-Wright partial compactification, $p$ is marked only if $Y_i^c$ contains an zero, marked point, or marked pole belonging to $D$. 
\end{proof}

Notice there are $(k+1)$ zeros, marked points, and marked poles in $D$ by assumption (see Lemma \ref{L:Codim=Zeros}). Fix $i$ so that $q$ is nonvanishing on $Y_i$. By Sublemma \ref{SL:SeparatingNode}, for each node of $Y_i$ that is a zero, marked point, or marked pole, there is a corresponding zero, marked point, or marked pole in $D$ that belongs to $Y_j$ for $j \ne i$. Therefore, the number of zeros, marked points, and marked poles on $(Y_i, q \restriction_{Y_i})$ is at most $(k+1)$, which, by Lemma \ref{L:Codim=Zeros}, concludes the proof.
\end{proof}




\section{Joinings}\label{S:Joinings}

The main result of this section is the following special case of the quasi-diagonal conjecture in Apisa-Wright \cite[Conjecture 8.35]{ApisaWrightDiamonds}.

\begin{prop}\label{P:Joinings:Rigid}
Suppose that $\cM$ is a cylinder rigid prime invariant subvariety of $\cQ_1 \times \hdots \times \cQ_n$ where $\cQ_i$ is a genus zero stratum for all $i$ and where $\cM$ projects to a dense subset of each $\cQ_i$. If $\cM$ is at least three-dimensional, then $\cQ_1 = \hdots = \cQ_n$ and, up to rescaling components, $\cM$ is the diagonal embedding. 
\end{prop}

For a definition of \emph{prime} see Chen-Wright \cite{ChenWright}. For a definition of \emph{cylinder rigid} and the associated notion of a \emph{subequivalence class} see Apisa \cite[Section 4]{Apisa-MHD}.

\begin{rem}\label{R:Counterexample1}
The conclusion of Proposition \ref{P:Joinings:Rigid} does not hold when $\cM$ is two-dimensional. For example, let $\cN$ be a locus of surfaces in $\cH(\emptyset) \times \hdots \times \cH(\emptyset)$ so that, for any surface $(X, \omega) \in \cN$, there is a flat torus $(E, dz)$ so that $(X, \omega)$ is a translation cover of $(E, dz)$. Let $\cN' \subseteq \cH(0^4) \times \hdots \times \cH(0^4)$ consist of surfaces in $\cN$ with four two-torsion points marked. Finally, let $\cM$ be the subset of $\cQ(-1^4) \times \hdots \times \cQ(-1^4)$ whose holonomy double covers are surfaces in $\cN'$. $\cM$ is cylinder rigid, prime, and each projection is a surjection; however, $\cM$ is not necessarily a diagonal embedding. 
\end{rem}

%
%
%
%
%

Before proceeding we will need the following results.

\begin{lem}\label{L:HypParallelism}
In genus zero strata without marked points, two distinct saddle connections are generically parallel if and only if they bound a cylinder.

Moreover, in genus zero strata with marked points, two cylinders are equivalent if and only if their core curves are homotopic after forgetting marked points. 
\end{lem}
\begin{proof}
Suppose that $s_1$ and $s_2$ are two distinct generically parallel saddle connections on a genus zero quadratic differential without marked points. Cutting them produces a unique component with trivial linear holonomy (by Masur-Zorich \cite[Theorem 1]{MZ}). Glue together the boundary of this component to form a translation surface $(X, \omega)$ without boundary and let $S$ denote the collection of saddle connections on $(X, \omega)$ corresponding to $s_1$ and $s_2$. Note that $S$ contains at most two saddle connections, all of which have identical length. These saddle connections are either arcs, i.e. the endpoints are distinct, or loops. 

%
%
%



If $S$ consists of two loops, then, since cutting these loops does not disconnect $(X, \omega)$, $(X, \omega)$ has genus at least two. Therefore, our original surface is formed by cutting two nonhomologous nonseparating simple closed curves on $(X, \omega)$ and regluing the boundary components. This implies that our original surface also has genus at least two, which is a contradiction. 

If $(X, \omega) - S$ contains a nonseparating simple closed curve, then so does our original genus zero surface, which is impossible. It follows that $(X, \omega)$ is a torus and $S$ contains either one loop or two arcs whose concatenation is a loop. In either case, $(X, \omega) - S$ is a cylinder, as desired. (Note that this uses that all marked points on $(X, \omega)$ are endpoints of saddle connections in $S$, which follows since the surfaces in the genus zero stratum have no marked points.)

The first claim implies that in a genus zero stratum without marked points each equivalence class of cylinders consists of a single cylinder. To see this, perturb so that all the cylinders in an equivalence class are generic (see Apisa-Wright \cite{ApisaWrightHighRank} for a definition of ``generic equivalence classes") and then apply the result to their boundary saddle connections. This implies the second claim in strata without marked points. The result in strata with marked points is immediate.
\end{proof}

\begin{lem}\label{L:SameConstants}
Suppose that $\cM$ is a cylinder rigid prime invariant subvariety in a product of strata $\cH_1 \times \hdots \times \cH_n$. Let $\cM_i$ denote the closure of the projection of $\cM$ onto the $i$th factor. Then the rank, rel, and dimension of $\cM_i$ do not depend on $i$.
\end{lem}
\begin{proof}
%
`Since $\cM$ is prime, the rank of $\cM_i$ is independent of $i$ (by Chen-Wright \cite[Theorem 1.3 (3)]{ChenWright}). If $(X, \omega) \in \cM$, each $\cM$-subequivalence class contains a cylinder on each component of $(X, \omega)$ (by Chen-Wright \cite[Theorem 1.3 (2)]{ChenWright}). Therefore, the maximum number of disjoint $\cM_i$-subequivalence classes, which is $\mathrm{rank}(\cM_i) + \mathrm{rel}(\cM_i)$ (by Wright \cite{Wcyl}; see Apisa-Wright \cite[Theorem 7.10 (2)]{ApisaWrightHighRank}), is also independent of $i$. It follows that the dimension, rank, and rel of $\cM_i$ is independent of $i$. 
\end{proof}

\begin{proof}[Proof of Proposition \ref{P:Joinings:Rigid}:]
Let $(X, q)$ be a surface in $\cM$ and let $(X_i, q_i)$ be the $i$th component. Since $\cM$ projects densely to $\cQ_i$ for all $i$, it follows that no cylinder on $(X_i, q_i)$ is subequivalent to any other. Since $\cM$ is prime, the absolute periods on $(X_i, q_i)$ determine the absolute periods on every component of $(X, q)$ (by Chen-Wright \cite[Theorem 1.3 (2)]{ChenWright}. These two results imply that each $\cM$-subequivalence class contains exactly one cylinder on each component of $(X, q)$.

\begin{sublem}\label{SL:JointPeriodicity}
One component of $(X, q)$ is horizontally periodic if and only if every component is. 
\end{sublem}
\begin{proof}
Suppose to a contradiction that, up to re-indexing, $(X_1, q_1)$ is horizontally periodic and $(X_2, q_2)$ is not. Let $\bfC_i$ denote the collection of horizontal cylinders on $(X_i, q_i)$. Notice that every cylinder in $\bfC_1$ is subequivalent to a unique cylinder in $\bfC_2$ and vice versa. It follows from Smillie-Weiss \cite[Corollary 6]{SW2} that there is a cylinder $C$ in $(X_2, q_2) - \bfC_2$. The subequivalence class of $C$ necessarily contains a cylinder $C'$ on $(X_1, q_1)$, which cannot intersect any cylinder in $\bfC_1$ since $C$ does not intersect any cylinder in $\bfC_2$ (by the cylinder proportion theorem, see Nguyen-Wright \cite[Proposition 3.2]{NW}; note that while the result is stated for equivalence classes of cylinders, it holds for subequivalence classes as well). However, $(X_1, q_1)$ is covered by cylinders in $\bfC_1$, so this is a contradiction. 
%
\end{proof}


Since $\cQ_i$ is a genus zero stratum, every cylinder is generically either a simple envelope or a simple cylinder (see Apisa-Wright \cite[Section 4]{ApisaWrightDiamonds} for a review of this terminology). 

\begin{sublem}\label{SL:Joinings:SameType}
Given a subequivalence class $\bfC$, either every cylinder in $\bfC$ is generically a simple cylinder or every cylinder in $\bfC$ is generically a simple envelope.
%
\end{sublem}
\begin{proof}
Perturb $(X, q)$ to a nearby surface $(X', q') \in \cM$ on which the cylinders in $\bfC$ are generic (this is possible on each $\cQ_i$ and hence on $\cM$, by Chen-Wright \cite[Theorem 1.3 (1)]{ChenWright}, since it is prime and hence projects onto an open dense subset of each $\cQ_i$). There is an open neighborhood $U \subseteq \cM$ of $(X', q')$ on which the cylinders in $\bfC$ remain generic. Let $(X'', q'')$ be a surface in $U$ on which the first component is periodic in the direction of $\bfC$ (this is possible since there is a dense collection of surfaces whose holonomy double covers are square-tiled in $\cQ_1$). By Sublemma \ref{SL:JointPeriodicity}, $(X'', q'')$ is periodic in the direction of $\bfC$. If a cylinder in $\bfC$ is a simple envelope (resp. cylinder) it shares a boundary saddle connection with exactly one (resp. two) other subequivalence classes on $(X'', q'')$. However, if one cylinder in $\bfC$ shares a boundary saddle connection with a cylinder in a subequivalence class $\bfC'$, every cylinder in $\bfC$ must (by Apisa \cite[Lemma 4.4]{Apisa-MHD}). Therefore, either every cylinder in $\bfC$ is generically a simple cylinder or every cylinder in $\bfC$ is generically a simple envelope. 
\end{proof}

To run an inductive argument we will require the following result. Recall that an invariant subvariety $\cN$ is called \emph{free} if every cylinder $C$ on any surface $(Y, q_Y) \in \cN$ may be \emph{freely dilated}, i.e. dilated while fixing the remainder of the surface and remaining in $\cN$.


\begin{sublem}
Given a generic subequivalence class $\bfC$ that contains a cylinder $C_i$ on $(X_i, q_i)$ for each $i$, $\cM_{\bfC} \subseteq (\cQ_1)_{C_1} \times \hdots \times (\cQ_n)_{C_n}$ is prime, cylinder rigid, and its projection to each factor is a dense subset of a free invariant subvariety.
\end{sublem}

Recall that $\Col_\bfC(X, \omega)$ denotes the result of collapsing the cylinders in $\bfC$ and $\cM_{\bfC}$ is the component of the boundary of $\cM$ containing it (see Apisa-Wright \cite[Section 2]{ApisaWrightDiamonds} for precise definitions and properties).

\begin{proof}
Since $\cM_{\bfC}$ is prime (by Apisa-Wright\cite[Lemma 9.1]{ApisaWrightHighRank}) and cylinder rigid (by Apisa \cite[Proposition 4.12]{Apisa-MHD}), it remains to show that its projection to each factor is free. Let $(\cM_{\bfC})_i$ be the projection to the $i$th factor.

Let $D$ be a cylinder on $\Col_{C_1}(X_1, q_1)$. We will show that $D$ may be freely dilated on $(\cM_{\bfC})_1$. Let $(Y^m, q^m)$ be a sequence of surfaces in $\cM$ converging to $\Col_{\bfC}(X, q)$. (To avoid confusion, we will index the sequence with superscripts and index components of surfaces with subscripts). Let $f^m: (Y^m, q^m) \ra \Col_{\bfC}(X, q)$ be the collapse map (see Mirzakhani-Wright \cite[Proposition 2.4]{MirWri} for a definition). By Mirzakhani-Wright \cite[Lemma 2.15]{MirWri}, there are cylinders $D^m$ on $(Y^m, q^m)$ so that $f^m(D^m)$ has a core curve homotopic to that of $D$. Moreover, the circumferences and moduli of $D^m$ converge to those of $D$. This implies that the circumferences and moduli of all cylinders in the subequivalence class $\bfD^m$ of $D^m$ have circumferences and moduli that are bounded by definition of cylinder rigid (see Apisa \cite[Definition 4.1]{Apisa-MHD}). By Mirzakhani-Wright \cite[Lemma 2.15]{MirWri}, there is a collection of cylinders $\bfD$ on $\Col_{\bfC}(X, q)$ so that $f^m$ takes cylinders in $\bfD^m$ to $\bfD$ and induces an isomorphism from $\Twist(\bfD^m, \cM)$ to $\Twist(\bfD, \cM_{\bfC})$ (by Apisa \cite[Theorem 2.9]{Apisa-MHD}). In particular, $\bfD$ is a subequivalence class, which implies that $D$ may be freely dilated on $(\cM_{\bfC})_1$ as desired.
%
%
%
\end{proof}

\begin{cor}\label{C:DenseProjection}
Given a generic subequivalence class $\bfC$ that contains a cylinder $C_i$ on $(X_i, q_i)$ for each $i$, $\cM_{\bfC} \subseteq (\cQ_1)_{C_1} \times \hdots \times (\cQ_n)_{C_n}$ is prime, cylinder rigid, and its projection to the $i$th factor is dense in $(\cQ_i)_{C_i}$. 
\end{cor}
\begin{proof}
If $(\cQ_i)_{C_i}$ belongs to the boundary of $\cQ_i$, then $(\cM_{\bfC})_i$ and $(\cQ_i)_{C_i}$ have the same dimension (by Apisa-Wright \cite[Lemma 6.5]{ApisaWrightHighRank}) and there is nothing to show. Suppose in order to deduce a contradiction that this is not the case. The following sublemma will produce the desired contradiction.

\begin{sublem}
There are no free codimension one invariant subvarieties in genus zero strata.  
\end{sublem}
\begin{proof}
Suppose in order to derive a contradiction that $\cN$ is a free codimension one invariant subvariety in a genus zero stratum $\cQ$. Let $\wt{\cN}$ and $\wt{\cQ}$ denote the loci of holonomy double covers of surfaces in $\cN$ and $\cQ$ respectively. Let $g$ be the genus of the surfaces in these loci. Since $\cN$ is codimension one, 
\[ \mathrm{rank}(\cN) \geq \mathrm{rank}(\cQ)-1  = g-1. \]
Since $\cQ$ and $\cN$ are free and consist of genus zero surfaces, $\wt{\cN}$ and $\wt{\cQ}$ are $h$-geminal (see Apisa-Wright \cite{ApisaWrightGeminal} for a definition). In particular, $\wt{\cN}$ is a full locus of covers of a hyperelliptic locus in a stratum of Abelian differentials (noting that the only Abelian or quadratic doubles that are $h$-geminal are quadratic doubles of genus zero strata, this follows from Apisa-Wright \cite[Proposition 8.1]{ApisaWrightGeminal}). Since $\wt{\cN}$ is contained in a locus of covers its rank is strictly less than $\frac{g}{2} + \frac{1}{2}$ (this is by a Riemann-Hurwitz computation, see, for example, Apisa-Wright \cite[Lemma 2.1]{ApisaWrightHighRank}). In summary, 
\[ g-1 \leq \mathrm{rank}(\cN) \leq \frac{g}{2} + \frac{1}{2}, \]
which implies that $g \leq 3$. 

Assume without loss of generality that every fixed point of the hyperelliptic involution on surfaces in $\wt{\cQ}$ is marked. Since additionally the collection of marked points on surfaces in $\wt{\cN}$ is invariant by the hyperelliptic involution, two cylinders are subequivalent if and only if they are exchanged by the hyperelliptic involution. Therefore, the collection, $\For(\wt{\cN})$, of surfaces in $\wt{\cN}$ formed by forgetting the marked points, is also $h$-geminal. We will use this to ignore marked points when $g \geq 2$.

Suppose first that $g = 3$. Then $\For(\wt{\cN})$ has rank two and is contained in a locus of covers of surfaces in $\cH(2)$ or $\cH(1,1)$. The covers are necessarily unbranched double covers. This implies that $\cN$ is a locus of double covers of surfaces in a genus zero stratum $\cQ_0$, where $\cQ_0$ is $\cQ(1, -1^5)$ or $\cQ(2, -1^6)$, with branching occurring over $2$ odd order zeros or poles. By finding a simple envelope (see Apisa-Wright \cite[Section 4]{ApisaWrightDiamonds} for a definition) on a surface in $\cQ_0$ with a boundary saddle connection joining two poles that do not belong to the branch locus, its preimage is two $\cN$-generically parallel isometric cylinders, which contradicts the assumption that $\cN$ is free.



Suppose next that $g = 2$. In this case, every cylinder is fixed by the hyperelliptic involution. Since subequivalent cylinders are exchanged by the involution, $\For(\wt{\cN})$ is free. By Mirzakhani-Wright \cite{MirWri2}, the only such locus in components of strata of Abelian differentials is the entire component, a contradiction.

Suppose finally that $g = 1$. By Apisa-Wright \cite[Proposition 6.5]{ApisaWrightGeminal}, $\wt{\cN}$ is a quadratic double of $\cQ(-1^4, 0^k)$, for some nonnegative integer $k$. Hence $\wt{\cN} = \wt{\cQ}$, which is a contradiction.  
%
%
%
\end{proof}
\end{proof}

We will proceed by induction on rank beginning with the rank one case as the base case.

\begin{sublem}\label{SL:Joining:BaseCase}
Proposition \ref{P:Joinings:Rigid} holds when $\cM$ has rank one.
\end{sublem}
\begin{proof}
The only rank one genus zero stratum of dimension $2+m$ is $\cQ(0^m, -1^4)$. By Lemma \ref{L:SameConstants}, there is some constant $m$ so $\cQ_1 = \hdots = \cQ_n = \cQ(0^m, -1^4)$. For simplicity, we pass to the locus of holonomy double covers $\wt{\cM} \subseteq \cH(0^{2m}, 0^4) \times \hdots \times \cH(0^{2m}, 0^4)$, where the four distinguished points are fixed points of the involution. Fix $(Y, \eta) \in \wt{\cM}$ and let $(Y_i, \eta_i)$ denote the $i$th component. 

Since almost every surface in $(\wt{\cM})_i$ contains a direction that is covered by $2m+2$ cylinders, suppose that $(Y_1, \eta_1)$ is such a surface and that the horizontal is such a direction. By Sublemma \ref{SL:JointPeriodicity}, each component of $(Y, \eta)$ is covered by $2m+2$ cylinders. Label the horizontal cylinders by the equivalence class to which they belong. The cylinders that border the fixed points are distinguished by only bordering one other equivalence class. (Recall that two cylinders are said to \emph{border one another} if they share a boundary saddle connection). This observation and Apisa \cite[Lemma 4.4]{Apisa-MHD}, imply that the subequivalence classes of horizontal cylinders appear in the same cyclic order on each component of $(Y, \eta)$. 

We will now proceed by induction on $m$ to show that, up to scaling, each component of $(Y, \eta)$ is the same translation surface. Let $P_i$ denote the four preimages of poles on $(Y_i, \eta_i)$. These are fixed points of the holonomy involution $J_i$. 

We begin with the base case of $m=1$. In particular, there are pairs of points $Q_i$ exchanged by $J_i$ on $(Y_i, \eta_i)$. By Corollary \ref{C:DenseProjection}, if moving the points in $Q_i$ causes them to merge with points in $P_i$ for one fixed $i$, then this occurs for every $i$. Rescale all components of $(Y, \eta)$ so that the points in $Q_i$ and those in $Q_j$ move at slope $\pm 1$ (see Apisa-Wright \cite[Definition 2.3]{ApisaWright} for a definition of slope). We will now argue as in Apisa-Wright \cite[Proof of Theorem 2.8]{ApisaWright} that the absolute periods of $(Y_i, \eta_i)$ are independent of $i$. Let $\alpha_1$ be any saddle connection on $(Y_1, \eta_1)$ whose endpoints belong to $P_1$ and move $Q_1$ so that it lies on $\alpha_1$. The saddle connection $\alpha_1$ has two endpoints belonging to $P_1$. Move $Q_1$ to one of these endpoints. At this point, all of the $Q_i$ are a subset of $P_i$. Now move the points in $Q_1$ along $\alpha_1$ to the opposite endpoint. When the points in $Q_1$ reach the opposite endpoint, every point in $Q_i$ must coincide with a point in $P_i$ again. Since all points in $Q_i$ and $Q_j$ move at slope $\pm 1$ with respect to each other, it follows that each $(Y_i, \eta_i)$ contains a saddle connection that has the same period as $\alpha_1$. This argument shows that the absolute periods of $(Y_i, \eta_i)$ are independent of $i$. Finally, since the points in $Q_i$ and $Q_j$ move at slope $\pm 1$ and have the property that $Q_i \subseteq P_i$ if and only if $Q_j \subseteq P_j$, it follows that if the period of an arc from a point in $Q_i$ to a point in $P_i$ is $v$, then there is an arc from a point in $Q_j$ to a point in $P_j$ of period $v$. Therefore, the components of $(Y, \eta)$ are pairwise isomorphic and so $(Y, \eta)$ is a diagonal embedding, as desired. The base case is now established.

Suppose now that $m > 1$. One may collapse a horizontal subequivalence class of cylinders $\bfC$ to reduce the number of marked points by two on one component. By Corollary \ref{C:DenseProjection}, this reduces the number of marked points by two on every component. Since, up to a standard shear in a horizontal equivalence class, $\Col_{\bfC}(Y, \eta)$ has the same absolute periods as $(Y, \eta)$ we see by the induction hypothesis that, up to forgetting marked points and after perhaps rescaling, any two components of $(Y, \eta)$ are isometric to each other. Moreover, by applying the induction hypothesis to $\Col_{\bfC}(Y, \eta)$ (which is possible by Corollary \ref{C:DenseProjection}), any two subequivalent horizontal cylinders are isometric.



By shearing the horizontal subequivalence classes, we may suppose that on $(Y_1, \eta_1)$ there are exactly two vertical cylinders. Suppose moreover, that all marked points, except for two fixed points of the involution, lie on the same vertical line. Showing that each component of $(Y, \eta)$ is the same translation surface now reduces to showing that on each component of $(Y, \eta)$, all marked points, except for two fixed points of the involution, lie on the same vertical line.

Suppose in order to derive a contradiction that, after perhaps re-indexing, there are points $p_1$ and $p_2$ that are not fixed points of the holonomy involution on $(Y_2, \eta_2)$ and that do not lie on the same vertical line. Moving both points $\epsilon$ to the right (by shearing subequivalence classes of horizontal cylinders), for sufficiently small $\epsilon$, creates six vertical cylinders on $(Y_2, \eta_2)$, but the resulting surface only has four vertical cylinders on $(Y_1, \eta_1)$. This contradicts the fact that any surface in $\cM$ has the same number of cylinders in any given direction on each component.
\end{proof}

Suppose now that $\cM$ has rank at least two. After perhaps perturbing $(X, q)$ we may assume that there are two disjoint generic subequivalence classes of cylinders, $\bfC_1$ and $\bfC_2$ on $(X, q)$, that share no boundary saddle connections (by Apisa-Wright \cite[Lemma 3.31]{ApisaWrightDiamonds}). After perhaps moving marked points, we may additionally assume that $\bfC_1$ and $\bfC_2$ are not just subequivalence classes, but equivalence classes (by Lemma \ref{L:HypParallelism}). By Corollary \ref{C:DenseProjection}, $\cM_{\bfC_i}$ satisfies the hypotheses imposed on $\cM$ in Proposition \ref{P:Joinings:Rigid}. The induction hypothesis implies that any two cylinders in $\Col_{\bfC_i}(\bfC_j)$, and hence in $\bfC_i$, have identical moduli for $i \ne j$. Moreover, up to rescaling components of $(X, q)$, $\cM_{\bfC_i}$ is a diagonal embedding by the induction hypothesis.


Rescale each component of $(X, q)$ so that the cylinders in $\bfC_1$ are isometric. By Sublemma \ref{SL:Joinings:SameType}, every cylinder in $\bfC_1$ is a simple envelope or every cylinder is a simple cylinder. It follows that each saddle connection in $\Col_{\bfC_1}(\bfC_1)$ has the same length. Since $\bfC_1$ is an equivalence class containing one cylinder on each component of $(X, q)$, on any component of $\Col_{\bfC_1}(X, q)$ there is no saddle connection outside of $\Col_{\bfC_1}(\bfC_1)$ that is generically parallel to the saddle connections in $\Col_{\bfC_1}(\bfC_1)$ (by Lemma \ref{L:HypParallelism}). Thus, all saddle connections in $\Col_{\bfC_1}(\bfC_1)$ are identical when the components of $\Col_{\bfC_1}(X, q)$ are identified with each other (by Lemma \ref{L:HypParallelism}). Since these saddle connections all have the same length, $\cM_{\bfC_1}$ is a diagonal embedding (without having to rescale components). Since the cylinders in $\bfC_1$ are all isometric simple envelopes or all isometric simple cylinders (by Sublemma \ref{SL:Joinings:SameType}), it follows that each component of $(X, q)$ is the same half-translation surface and hence $\cM$ is a diagonal embedding into $\cQ_1 \times \hdots \times \cQ_n$ where $\cQ_1 = \hdots = \cQ_n$.
\end{proof}

\section{Cylinder rigid rank one subvarieties with rel}\label{S:CylinderRigidRankOne}

Before stating the main result of the section we will make the following two definitions. 

\begin{defn}\label{D:WP}
If $(X, \omega)$ is a translation surface in a hyperelliptic locus then we will let $(X, \omega)^{WP}$ denote $(X, \omega)$ with all of its Weierstrass points marked. This is well-defined if the genus of $X$ is at least two; however, if $X$ has genus one there may be multiple choices of a hyperelliptic involution. In the sequel, we will only work with genus one surfaces that have marked points, which specifies a choice of a hyperelliptic involution (see Apisa-Wright \cite[Lemma 4.5]{ApisaWrightDiamonds}) and allows us to meaningfully discuss $(X, \omega)^{WP}$.

If $\bfC$ is a collection of cylinders on $(X, \omega)$ then we will let $\bfC^{WP}$ denote the corresponding cylinders on $(X,\omega)^{WP}$. Finally, if $\cM$ is an invariant subvariety, then we will let $\cM^{WP} = \{ (X, \omega)^{WP} : (X, \omega) \in \cM\}$, which is also an invariant subvariety. Throughout the sequel we will often use the fact that $(\cM_{\bfC})^{WP} = (\cM^{WP})_{\bfC^{WP}}$.
\end{defn}

\begin{defn}
Given a translation surface $(X, \omega)$ that is not a torus cover, there is a unique map $\pi_{X_{min}}$ (resp. $\pi_{Q_{min}}$) to a translation (resp. half-translation) cover of which all other such covers are factors. This follows from M\"oller \cite[Theorem 2.6]{M2} and its extension in Apisa-Wright \cite[Lemma 3.3]{ApisaWright}. When $(X, \omega)$ is a torus cover, let $\pi_{X_{min}}$ denote the minimal degree map to a torus. We will say that a translation (resp. half-translation) cover with domain $(X, \omega) \in \cM$ is \emph{$\cM$-minimal} if it can be deformed to all nearby surfaces in $\cM$ and for any such surface with dense orbit in $\cM$ it is the minimal translation (resp. half-translation) cover. Let $\cM_{min} := \{ \pi(X, \omega) : (X, \omega) \in \cM \text{ and $\pi$ is the $\cM$-minimal translation cover}\}$.
\end{defn}

For a definition of (half)-translation cover see Apisa-Wright \cite[Definition 3.2]{ApisaWrightDiamonds}.

The main result of this section is the following.

\begin{prop}\label{P:Main:Rk1}
If $\cM$ is a rank one cylinder rigid subvariety with rel $r > 0$ contained in a codimension zero or one hyperelliptic locus, then the following hold.
\begin{enumerate}
    \item\label{I:P:Main:Rk1:Rel=1:IETCase} If $r = 1$ and some surface in $\cM$ contains two generic subequivalent cylinders that are glued together along an IET, then the $\cM^{WP}$-minimal translation cover satisfies Assumption CP, and $(\cM^{WP})_{min}$ is a quadratic double of $\cQ(0, -1^4)$. Moreover, $\cM_{min}$ is a quadratic double of $\cQ(0, -1^4)$ with exactly one preimage of a pole marked.
    \item\label{I:P:Main:Rk1:Rel=1:NoIETCase} If $r =1$ and $\cM$ does not contain a surface with two generic subequivalent cylinders that border one another, then the $\cM$-minimal translation cover satisfies Assumption CP and $\cM_{min}$ is either $\cH(0,0)$ or an Abelian double of it. 
    \item\label{I:P:Main:Rk1:Rel=2} If $r=2$ then no two generic subequivalent cylinders border one another, the $\cM$-minimal translation cover satisfies Assumption CP, and $\cM_{min}$ is a quadratic double of $\cQ(0^2, -1^4)$ with no preimage of a pole marked. 
    
    Moreover, any periodic direction on a surface $(X, \omega) \in \cM$ that is covered by subequivalence classes $\bfC_1, \bfC_2, \bfC_3$ has, up to re-indexing, $\bfC_1$ consisting of a pair of cylinders exchanged by the hyperelliptic involution and $\bfC_2$ and $\bfC_3$ consisting of cylinders that only border those in $\bfC_1$.
\end{enumerate}
\end{prop}

Two cylinders are said to be \emph{glued together along an IET} if there is a boundary component of one cylinder that is also a boundary component of the other. Assumption CP is defined in Apisa-Wright \cite{ApisaWrightDiamonds}. Given a collection of cylinders $\bfC$, define $\tau_{\bfC} = \sum_{C \in \bfC} \gamma_C^*$. We will prove the following first.

\begin{lem}\label{L:Main}
Suppose that $\cM$ is a rank one rel one cylinder rigid subvariety contained in a codimension zero or one hyperelliptic locus. Let $(X, \omega) \in \cM$ be a horizontally periodic surface with two horizontal subequivalence classes, $\bfA$ and $\bfB$. If $\bfA \cup \bfB$ contains a pair of cylinders exchanged by the hyperelliptic involution, then call them $\bfD = \{D, D'\}$ and, up to relabelling, suppose that $\bfD \subseteq \bfA$. One of the following holds:
\begin{enumerate}
    \item If every horizontal cylinder on $(X, \omega)$ is fixed by the involution then $\sigma_{\bfA} = \tau_{\bfA}$ and $\sigma_{\bfB} = \tau_{\bfB}$ (up to scaling); $\tau_{\bfA} - \tau_{\bfB}$ is rel; and no cylinder in $\bfA$ (or $\bfB$) borders another such cylinder.
    \item If the cylinders in $\bfD$ share no boundary saddle connections, then the previous conclusion holds.
    \item Suppose that the cylinders in $\bfD$ border each other. Then they are glued together along an IET; no two cylinders in $\bfB$ (resp. $\bfA)$ share boundary saddle connections (resp. except for the two cylinders in $\bfD$); and, up to scaling, 
\[ \sigma_{\bfB} = \tau_\bfB \qquad \sigma_{\bfA} = \frac{\gamma_D^* + \gamma_{D'}^*}{2} + \tau_{\bfA - \bfD}\]
and the difference of these two classes is rel.
\end{enumerate}
\end{lem}
Recall that $\sigma_{\bfC}$ is the notation for the \emph{standard shear} in a collection of cylinders (the primary motivation for considering this class comes from the cylinder deformation theorem of Wright \cite{Wcyl}; see Apisa-Wright \cite[Section 3]{ApisaWrightDiamonds} for a definition and discussion).
\begin{proof}
Since $\cM$ is rank one rel one and cylinder rigid there is some positive real number $c$ so that $\sigma_\bfA - c\sigma_\bfB = \sum_{C \in \bfA \cup \bfB} a_C \gamma_C^*$ is rel where $a_C$ is a real number for each $C \in \bfA \cup \bfB$. If each horizontal cylinder is fixed by the involution, then the result follows by Lemma \ref{L:GraphlikeRel}. If the cylinders in $\bfD$ share a boundary saddle connection then they are glued together along an IET (by Lemma \ref{L:NonAdjacentHomologous}) and so the claim follows by Lemma \ref{L:RelInHyp}.

Therefore, suppose that the cylinders in $\bfD$ do not share boundary saddle connections. By Lemma \ref{L:RelInHyp}, it follows that no two cylinders in $\bfA-\bfD$ (or $\bfB$) border one another. Let $h_C$ denote the height of $C \in \bfA \cup \bfB$ on $(X, \omega)$ in the sequel. 



\begin{sublem}\label{SL:Alternating}
$D$ does not border both a cylinder in $\bfA$ and $\bfB$.
\end{sublem}
\begin{proof}
Suppose not. Suppose without loss of generality that $E$ is a horizontal cylinder in $\bfB$ that borders the top boundary of $D$. By Lemma \ref{L:RelInHyp}, all the horizontal cylinders that border the top boundary of $D$ belong to $\bfB$ and have identical height. Therefore, by assumption, there is a cylinder $F$ that borders the bottom boundary of $D$ and that belongs to $\bfA$. We will now use $\sigma_{\bfB}$ to overcollapse $E$ into $\{D, D'\}$.

Recall that an overcollapse is defined as follows. Suppose that $(Y, \eta)$ is a surface with a collection of horizontal cylinders $\bfC$ so that $\sigma_{\bfC} \in T_{(Y, \eta)} \cN$ where $\cN$ is an invariant subvariety containing $(Y, \eta)$. Applying $\sigma_{\bfC}$ we may shear the cylinders in $\bfC$ so they have no vertical saddle connections. Let $(Y_t, \eta_t) := (Y, \eta) - it\sigma_{\bfC}$; this is called an \emph{overcollapse path}. Notice that $(Y, \eta) = (Y_0, \eta_0)$ and that on $(Y_1, \eta_1)$ the heights of the cylinders in $\bfC$ have all reached zero. Since no vertical saddle connection has collapsed, $(Y_t, \eta_t)$ continues to be defined for $t \in [1, 1+\epsilon)$ for some $\epsilon > 0$. The surfaces in $(Y_t, \eta_t)$ for $t > 1$ are called \emph{overcollapsed surfaces}. This construction is illustrated in Figure \ref{F:overcollapse1}. In our case, we will shear the cylinders in $\bfB$ so that there is a vertical line that begins at a singularity on the top boundary of $E$ and travels along a segment entirely contained in $E$ to arrive at a saddle connection on the top boundary of $D$. This ensures that the height of $D$ changes with $t$ for $t > 1$.


\begin{figure}[h]\centering
\includegraphics[width=0.75\linewidth]{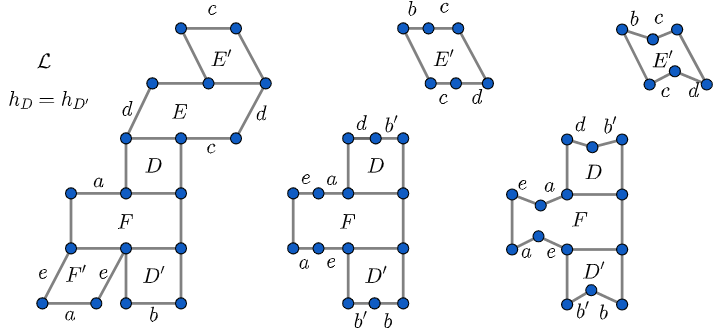}
\caption{The surface on the left corresponds to our original surface $(X, \omega)$ in a hyperelliptic locus $\cL$, which is defined, in this example, by the equation $h_D = h_{D'}$. In the notation of the previous paragraph, the surface in the middle corresponds to $(X_1, \omega_1)$ and the surface on the right to $(X_t, \omega_t)$ for some $t > 1$.}
\label{F:overcollapse1}
\end{figure}

Since $D$ only borders cylinders in $\bfB$ along its top boundary (and since all these cylinders have the same height and border no cylinders in $\bfB$), it follows that on $(X_t, \omega_t)$, for $t > 1$, $D$ has height $h_D - th_E$.

Since $D$ and $F$ are subequivalent they have a constant ratio of heights along the path $(X_t, \omega_t)$. Therefore, there must be a horizontal cylinder $F' \in \bfB$ that shares a boundary saddle connection with $F$ so that on $(X_t, \omega_t)$, for $t > 1$, $F$ has height $h_F - 2th_{F'}$ (the factor of $2$ appears since $F' \notin \{D, D'\}$ and hence shares a saddle connection with $F$ on its top and bottom boundary). 

Since $D$ and $F$ have a constant ratio of heights along the path $(X_t, \omega_t)$, this justifies the third equality in the following:
\[ \frac{a_D}{-a_{F'}} = \frac{a_D}{a_F} = \frac{h_D}{h_F} = \frac{h_E}{2h_{F'}} = \frac{a_E}{2a_{F'}} \] 
(the first equality follows from Lemma \ref{L:RelInHyp} and the second and fourth from the fact that $\sigma_\bfA - c\sigma_{\bfB}=\sum_{C \in \bfA \cup \bfB} a_C \gamma_C^*$). These equalities yield $a_D = -\frac{a_E}{2}$. However, since $\sigma_\bfA - c\sigma_{\bfB}$ is rel, Lemma \ref{L:RelInHyp} implies that 
\[ a_D = \frac{-a_E - a_F}{2}  \]
and so $a_F = 0$, implying that $F$ has height zero, a contradiction. 
%
\end{proof}


By Sublemma \ref{SL:Alternating} and Lemma \ref{L:RelInHyp}, it follows that no horizontal cylinder in $\bfA$ (resp. $\bfB$) shares a boundary saddle connection with any other such cylinder. In particular, all the cylinders that share boundary saddle connections with the cylinders in $\{D, D'\}$ belong to $\bfB$. Let $E$ (resp. $F$) be a horizontal cylinder that borders the top (resp. bottom) boundary of $D$. By Lemma \ref{L:RelInHyp}, it suffices to show that $a_E = a_F$. Suppose to a contradiction that $|a_E| > |a_F|$ (in particular this implies that $h_E > h_F$). Note that $a_E$ and $a_F$ are both negative real numbers.

Use $\sigma_{\bfA}$ to overcollapse $\{D, D'\}$ into $F$. As in Sublemma \ref{SL:Alternating}, let $(X_t, \omega_t)$ denote the overcollapse path. We may suppose without loss of generality (perhaps after reselecting $E$ and $F$) that $D$ contains a vertical line that begins at a singularity on the bottom (resp. top) boundary of $D$ and travels along a vertical segment entirely contained in $D$ to arrive at a saddle connection belonging to the bottom (resp. top) boundary of $E$ (resp. $F$). In particular, this assumption implies that the height of $E$ (resp. $F$) on $(X_t, \omega_t)$, for $t > 1$, is bounded above by $h_E - 2th_D$ (resp. $h_F - 2th_D$).

First observe that if $F' \notin \{D, D'\}$ is a horizontal cylinder bordering $F$, then 
\[ h_{F'} = |a_{F'}| = |a_F| < \frac{|a_E| + |a_F|}{2} =  |a_D| = h_D \] (the first and final equalities by definition; the second and third equalities are by Lemma \ref{L:RelInHyp}; the inequality follows from $|a_E| > |a_F|$). This shows that the height of $F$ on $(X_t, \omega_t)$, for $t > 1$, is $h_F - 2th_D$ (note that we have tacitly used the fact that no two distinct cylinders in $\bfA$ (or $\bfB$) border each other). Since $h_E > h_F$ and since the ratio of $h_E$ to $h_F$ is constant along $(X_t, \omega_t)$ it follows that on $(X_t, \omega_t)$, for $t>1$, the height of $E$ is strictly less than $h_E - 2th_D$. Therefore, there is a horizontal cylinder $E' \notin \{D, D'\}$ that borders $E$ on $(X, \omega)$ so that the height of $E$ on $(X_t, \omega_t)$, for $t > 1$, is $h_E - 2th_{E'}$. This implies the following,
\[\frac{a_D}{-a_E} = \frac{a_D}{a_{E'}} = \frac{h_D}{h_{E'}} = \frac{h_F}{h_E} = \frac{a_F}{a_E}\]
(the first equality is by Lemma \ref{L:RelInHyp}, the second and fourth equalities are by definition of the coefficients $\{a_C\}_{C \in \bfA \cup \bfB}$, and the third equality is by the fact that $h_E - 2th_{E'}$ and $h_F - 2th_D$ have constant ratio for $t > 1$). Therefore, $a_D = -a_F$, which contradicts the fact that $a_D = \frac{-a_E - a_F}{2}$ (by Lemma \ref{L:RelInHyp}) and the assumption that $a_E \ne a_F$.
%
%
%
\end{proof}


\begin{cor}\label{C:R1R1:Starter}
If $\cM$ is a cylinder rigid invariant subvariety of rank one rel one in a codimension zero or one hyperelliptic locus, then $\cM$ is a locus of torus covers.

Moreover, given a surface in $\cM$, if $\pi$ is the quotient by the absolute period lattice and $\bfA$ is a generic subequivalence class of cylinders, one of the following occurs:
\begin{enumerate}
    \item No two cylinders in $\bfA$ are glued together along an IET and there is an $\cM_{min}$-subequivalence class $\bfC_\bfA$ so that $\pi^{-1}(\bfC_{\bfA}) = \bfA$. $\bfC_\bfA$ consists of either one cylinder or two isometric cylinders that do not border each other.
    \item Two cylinders in $\bfA$ are glued together along an IET and there is an $\cM_{min}$-subequivalence class $\bfC_\bfA$ consisting of two isometric cylinders that border one another and so that $\pi^{-1}(\overline{\bfC_\bfA}) = \overline{\bfA}$.
\end{enumerate}
\end{cor}

Cylinders are assumed to be open. Thus, there is an important but subtle distinction between the two claims about the preimage of $\bfC_\bfA$. For a discussion and illustration see Apisa-Wright \cite[Remark 2.1]{ApisaWrightDiamonds}.

\begin{proof}
Fix $(X, \omega) \in \cM$. For the first claim, it suffices to show that the absolute periods of $(X, \omega)$ form a lattice. Let $\bfA$ be any subequivalence class of horizontal cylinders on $(X, \omega)$. By Lemma \ref{L:Main}, there is a constant $h_\bfA$ so that every cylinder in $\bfA$ has height $h_{\bfA}$ unless it is generically glued to another cylinder in $\bfA$ along an IET, in which case it has height $\frac{h_{\bfA}}{2}$. 

If $(X, \omega)$ is covered by $\bfA$, then the imaginary part of every absolute period is an integer multiple of $h_\bfA$ (note that if $\{A, A'\}$ are two cylinders in $\bfA$ glued together along an IET, then any closed curve that crosses $A$ also crosses $A'$). 

If there is a second subequivalence class $\bfB$ parallel to $\bfA$ then, defining $h_\bfB$ analogously to $h_\bfA$, the imaginary part of each absolute period is an integer multiple of $h_{\bfA} + h_{\bfB}$ (by Lemma \ref{L:Main}). This is because the rel vector identified in Lemma \ref{L:Main} implies that every absolute period crosses the same number of cylinders in $\bfA$ as cylinders in $\bfB$, up to considering two cylinders glued along an IET to be the same cylinder. 
By considering two transverse periodic directions we see that the absolute periods of $(X, \omega)$ form a lattice, as desired. 

Now we will turn to the claim in the second paragraph and assume without loss of generality that $\bfA$ is a generic subequivalence class and that $\bfA$ and $\bfB$ are the two subequivalence classes in the direction of $\bfA$. We have the following two cases.

First, suppose that there are horizontal cylinders, $A \in \bfA$ and $B \in \bfB$, that border one another and so that each is fixed by the hyperelliptic involution. Then, up to applying $\sigma_{\bfA}$ and $\sigma_{\bfB}$, there is a simple closed curve that is contained in $\overline{A \cup B}$, that crosses the core curve of $A$ and $B$ once, and that has period $i(h_\bfA + h_{\bfB})$. This is the smallest nonzero magnitude of the imaginary part of an absolute period of $(X, \omega)$. Therefore, the claim about the image of $\bfA$ under the quotient by the absolute period lattice is clear. An identical argument works in the case that $\bfA \cup \bfB$ contains two cylinders generically glued together along an IET by treating those two cylinders as a single cylinder fixed by the hyperelliptic involution. 


Finally, suppose that $\bfA \cup \bfB$ does not contain two horizontal cylinders glued together along an IET and does not contain two horizontal cylinders that are fixed by the hyperelliptic involution that border one another. Up to re-labelling, $\bfB$ consists of exactly two cylinders that are exchanged by the hyperelliptic involution. Since each absolute period must intersect each cylinder in $\bfB$ the same number of times, it follows that the imaginary part of each absolute period is an integral multiple of $2h_{\bfA} + 2h_{\bfB}$. After shearing, it is straightforward to see that we may take $i(2h_{\bfA} + 2h_{\bfB})$ to be the period of a closed curve. Therefore, if $\pi$ is the minimal degree map to a torus, there is a collection of two isometric cylinders $\bfC_{\bfA}$ on the codomain of $\pi$ that do not border one another and so that $\bfA = \pi^{-1}(\bfC_{\bfA})$. The same holds for $\bfB$.
\end{proof}

\begin{cor}\label{C:Rank1Rel1}
If $\cM$ is a cylinder rigid invariant subvariety of rank one rel one in a codimension zero or one hyperelliptic locus, then $\cM$ satisfies the conclusion of Proposition \ref{P:Main:Rk1}.
%
%
\end{cor}


\begin{proof}

By Corollary \ref{C:R1R1:Starter}, $\cM_{min}$ has rank one rel one and is a geminal locus of tori, which are classified in Apisa-Wright \cite[Proposition 6.5]{ApisaWrightGeminal}.

Suppose first that every periodic direction on a surface in $\cM_{min}$ contains at most two cylinders. Then each $\cM_{min}$-subequivalence class contains exactly one cylinder and so the $\cM$-minimal translation covers satisfies Assumption CP by Corollary \ref{C:R1R1:Starter} and $\cM_{min}$ is $\cH(0,0)$ by Apisa-Wright \cite[Proposition 6.5]{ApisaWrightGeminal}. This is possibility \eqref{I:P:Main:Rk1:Rel=1:NoIETCase} in Proposition \ref{P:Main:Rk1}.

Suppose next that there is a periodic direction with four cylinders on $\cM_{min}$, no two of which border each other. Since the surface in $\cM$ have at most four zeros the same is true of $\cM_{min}$, which implies that $\cM_{min}$ is an Abelian double of $\cH(0,0)$ by Apisa-Wright \cite[Proposition 6.5]{ApisaWrightGeminal}. It follows from Corollary \ref{C:R1R1:Starter} that the $\cM$-minimal translation cover satisfies Assumption CP. This is again possibility \eqref{I:P:Main:Rk1:Rel=1:NoIETCase} in Proposition \ref{P:Main:Rk1}.


The only remaining case to consider is the one in which $\cM_{min}$ contains a surface with a generic subequivalence class consisting of two cylinders glued together along an IET, in which case $\cM_{min}$ is a quadratic double of $\cQ(0, -1^4)$ by Apisa-Wright \cite[Proposition 6.5]{ApisaWrightGeminal}. By Corollary \ref{C:R1R1:Starter}, each cylinder direction on a surface in $\cM_{min}$ has at most three cylinders, so exactly one preimage of a pole is marked. To show that $\cM$ is described by Proposition \ref{P:Main:Rk1} \eqref{I:P:Main:Rk1:Rel=1:IETCase}, it remains to show that the $\cM^{WP}$-minimal translation cover $\pi$ satisfies Assumption CP and that $(\cM^{WP})_{min}$ is a quadratic double of $\cQ(0, -1^4)$.

\begin{sublem}\label{SL:CPforWP}
If $\bfC$ is a generic subequivalence class on $(X, \omega) \in \cM$, then there is a subequivalence class $\bfC'$ on $\pi((X, \omega)^{WP})$ consisting of two isometric cylinders so that $\bfC^{WP} = \pi^{-1}(\bfC')$. In particular, $\pi$ satisfies Assumption CP.
\end{sublem}
\begin{proof}
If $\bfC$ does not contain two cylinders exchanged by the hyperelliptic involution, then the claim is immediate from Corollary \ref{C:R1R1:Starter} (and the assumption that $\cM_{min}$ does not have a surface with a periodic direction comprised of four cylinders). Suppose therefore that $\bfC$ contains two cylinders exchanged by the hyperelliptic involution. If these two cylinders are glued together along an IET, then the result again follows from Corollary \ref{C:R1R1:Starter}. Therefore, suppose that $\bfC$ contains two cylinders that are exchanged by the hyperelliptic involution, but not glued together along an IET. By Corollary \ref{C:R1R1:Starter}, the image of $\bfC$ and the other subequivalence class of cylinders parallel to it under $\pi$ are both single cylinders. However, a surface in a quadratic double of $\cQ(0, -1^4)$ with the preimage of one pole marked does not have any periodic direction covered by two cylinders in two subequivalence classes. This is a contradiction.
\end{proof}

By Sublemma \ref{SL:CPforWP}, $(\cM^{WP})_{min}$ is geminal and, after forgetting marked points, is a quadratic double of $\cQ(0, -1^4)$ with one preimage of a pole marked. It follows that $(\cM^{WP})_{min}$ is also a quadratic double of $\cQ(0, -1^4)$ by Apisa-Wright \cite[Proposition 6.5]{ApisaWrightGeminal}.
\end{proof}


\begin{lem}\label{L:Rank1Rel2}
If $\cM$ is cylinder rigid of rank one rel two in a codimension one hyperelliptic locus, then $\cM$ satisfies the conclusion of Proposition \ref{P:Main:Rk1}.
%
%
\end{lem}
\begin{proof}
Let $(X, \omega)$ be any surface in $\cM$ with three horizontal subequivalence classes of cylinders $\{\bfC_i\}_{i=1}^3$. By assumption, there are positive real constants $c_2$ and $c_3$ so that $\sigma_{\bfC_1} - c_i \sigma_{\bfC_i}$ is rel for $i \in \{2, 3\}$. Since $\Twist((X, \omega), \cM)$ contains two dimensions of rel it follows from Lemma \ref{L:RelInHyp} that there is a pair of horizontal cylinders $\bfD = \{D, D'\}$ so that cutting and regluing along $\bfD$ produces two surfaces $(X_i, \omega_i)$ in hyperelliptic components admitting rel for $i \in \{1, 2\}$ (by Lemmas \ref{L:CutReglue} and \ref{L:RelInHyp}). Let $S_i$ denote the subsurfaces corresponding to $(X_i, \omega_i)$ on $(X, \omega) - \bfD$ for $i \in \{1, 2\}$. 

\begin{sublem}\label{SL:PureSubsurfaces}
$S_i$ does not contain two distinct horizontal cylinders that are each fixed by the hyperelliptic involution and that share a boundary saddle connection. 
\end{sublem}
\begin{proof}
Suppose not. Since the union of these two cylinders contains a closed curve that crosses each cylinder exactly once, it follows (from the two rel vectors that we identified in the twist space) that these two cylinders belong to distinct subequivalence classes, say $\bfC_1$ and $\bfC_2$. But then the closed curve that is contained in their union has its period increase under the rel vector $\sigma_{\bfC_1} - c_3 \sigma_{\bfC_3}$, which is a contradiction. 
\end{proof}

Sublemma \ref{SL:PureSubsurfaces} implies that, up to reindexing, the cylinders in $\bfC_i$ are precisely the horizontal cylinders contained in $S_i$ and $\bfC_3 = \bfD$. Moreover, cylinders in $\bfC_1 \cup \bfC_2$ only border cylinders in $\bfC_3$ and no cylinder in $\bfC_3$ borders any other cylinder in $\bfC_3$. 

These observations, along with the fact that $\sigma_{\bfC_1} - c_i \sigma_{\bfC_i}$ is rel for $i \in \{2, 3\}$, imply that all cylinders in $\bfC_i$ have identical heights, call it $h_i$, by Lemma \ref{L:RelInHyp}. Therefore, the imaginary part of every absolute period is an integer multiple of $h_1 + h_2 + 2h_3$. 

Any surface with dense orbit in a rank one rel $r$ cylinder rigid invariant subvariety $\cN$ will contain infinitely many periodic directions that consist of $r+1$ subequivalence classes (since there is an open set of cylindrically stable surfaces when $\cN$ has rank one; see Apisa-Wright \cite[Lemma 7.10]{ApisaWrightHighRank}). By considering two such transverse periodic directions, we see that the absolute periods form a lattice for any $(X, \omega) \in \cM$ with dense orbit in $\cM$ and hence $\cM$ is a locus of torus covers. 

The quotient by the absolute period lattice necessarily sends the previously considered surface to a torus with four horizontal cylinders - one of height $h_1$, one of height $h_2$, and two of height $h_3$ - and so the preimages of the cylinder(s) of height $h_i$ consists of the cylinders in $\bfC_i$. This implies that the $\cM$-minimal map satisfies Assumption CP and that $\cM_{min}$ is geminal in a locus of tori. Therefore, periodic directions consisting of generic cylinders on surfaces in $\cM_{min}$ are covered by two free cylinders and one subequivalence class consisting of two isometric non-adjacent cylinders. It follows that $\cM_{min}$ is the quadratic double of $\cQ(0^2, -1^4)$ with no preimages of poles marked (by Apisa-Wright \cite[Proposition 6.5]{ApisaWrightGeminal}).
%
%
%
%
\end{proof}

\begin{proof}[Proof of Proposition \ref{P:Main:Rk1}:]
This is immediate from Corollary \ref{C:Rank1Rel1} and Lemma \ref{L:Rank1Rel2}.
\end{proof}

%
%
%

\section{Proof of the Theorem \ref{T:Main} assuming Theorem \ref{T:Main:Rigid:Simple}}\label{S:ProofTMain}

This section is devoted to the proof of Theorem \ref{T:Main}. We will begin with the following preliminary result (see Figure \ref{F:Theorem1-5} for an illustration).

\begin{lem}\label{L:RelInEC}
Suppose that $\cM$ is an invariant subvariety of rank at least two in a codimension zero or one hyperelliptic locus. Suppose that $(X, \omega) \in \cM$ is horizontally periodic and cylindrically stable. Let $\bfC_1$ be a horizontal equivalence class so that $\Twist(\bfC_1, \cM)$ contains a nonzero purely imaginary rel vector $v$. Let $\bfC_+$ (resp. $\bfC_-$) denote the cylinders in $\bfC_1$ whose heights increase (resp. decrease) along the path $(X, \omega) + tv$ for $t > 0$. Then, up to replacing $v$ with $-v$, the following holds:
\begin{enumerate}
    \item\label{I:C+} $\bfC_+$ consists of two cylinders exchanged by the hyperelliptic involution.
    \item\label{I:C-} Every cylinder in $\bfC_-$ only borders cylinders in $\bfC_+$.
    \item\label{I:Rel} Up to scaling, $v = i(\tau_{\bfC_+} - \tau_{\bfC_-})$.
\end{enumerate}
\end{lem}

The terminology \emph{cylindrically stable} was introduced in Aulicino-Nguyen \cite[Definition 2.4]{AN2} (see Apisa-Wright \cite[Lemma 7.10]{ApisaWrightHighRank} for a summary of relevant properties). In the sequel we will use the following property of cylindrically stable surfaces.


\begin{lem}\label{L:MovingToCS}
If $(X, \omega)$ is a surface in an invariant subvariety $\cM$ with a horizontal equivalence class $\bfC$, then there is path from $(X, \omega)$ to a cylindrically stable surface $(Y, \eta)$ along which $\bfC$ persists and remains horizontal. 
\end{lem}
\begin{proof}[Sketch of proof:]
By Wright \cite[Proof of Theorem 1.1]{Wcyl}, there is a path in $\cM$ starting at $(X, \omega)$, along which the cylinders in $\bfC$ persist, remain horizontal, and that ends at a horizontally periodic surface $(X', \omega')$. By Wright \cite[Proof of Lemma 8.6]{Wcyl}, if $(X', \omega')$ is not cylindrically stable it is possible to perturb $(X', \omega')$ so that all of its horizontal cylinders remain horizontal, but no longer cover the surface. Iterating this construction produces a path from $(X, \omega)$ to a horizontally periodic cylindrically stable surface $(Y, \eta) \in \cM$ along which the cylinders in $\bfC$ persist. 
\end{proof}

\begin{proof}[Proof of Lemma \ref{L:RelInEC}:]
We begin with the following preliminary result.
\begin{sublem}\label{SL:StructureOfC1-0}
$\bfC_{\pm}$ contains a pair of cylinder exchanged by the hyperelliptic involution. One component of the complement of the core curves of these cylinders has the property that the horizontal cylinders contained in it are exactly the remaining cylinders in $\bfC_{\pm}$.  
\end{sublem}
\begin{proof}
If every horizontal cylinder is fixed by the hyperelliptic involution, then every horizontal cylinder has its height altered along $(X, \omega) + tv$ (by Lemma \ref{L:GraphlikeRel}) and hence belongs to $\bfC_1$, which contradicts the fact that $\cM$ has rank at least two and that $(X, \omega)$ is cylindrically stable (see Apisa-Wright \cite[Lemma 7.10 (4)]{ApisaWrightHighRank}). If the two horizontal cylinders exchanged by the hyperelliptic involution, call them $\{D, D'\}$ do not belong to $\bfC_1$, then $v$ is supported on every horizontal cylinder except those in $\{D, D'\}$ (by Lemma \ref{L:RelInHyp}). Since the core curves of $D$ and $D'$ are homologous and part of a homology relation with the core curves of the other horizontal cylinders (by the existence of rel in $\Twist(\bfC_1, \cM)$ and Lemma \ref{L:RelInHyp}), it follows that $\{D, D'\} \subseteq \bfC_1$, which is a contradiction. Therefore, $\bfC_1$ contains a pair of cylinders exchanged by the hyperelliptic involution and that belong to the support of $v$. The second claim is now immediate from Lemma \ref{L:RelInHyp}.
\end{proof}

Suppose without loss of generality that the pair of horizontal cylinders $\bfD := \{D, D'\}$ exchanged by the hyperelliptic involution belongs to $\bfC_+$. The following establishes the claim.  

\begin{sublem}\label{SL:StructureOfC1}
$\bfC_+ = \bfD$ and cylinders in $\bfC_-$ only border cylinders in $\bfD$.
%
\end{sublem}
\begin{proof}
Let $S$ be the component of the complement of the core curves of $\bfD$ that contains a cylinder $C_- \in \bfC_-$. Suppose to a contradiction that $S$ also contains a cylinder $C_+ \in \bfC_+$. Let $m_-$ and $m_+$ be the moduli of these two cylinders respectively. We may suppose without loss of generality that $m_+$ and $m_-$ are not rational multiples of each other. There is some equivalence class, say $\bfC_2$ that borders a cylinder $C_0$, of modulus $m_0$, in $\bfC_1$. By overcollapsing the cylinders in $\bfC_2$ we can continuously change the modulus of $C_0$ while fixing the moduli of $m_+$ and $m_-$ (since they do not border any cylinder in $\bfC_2$ since they belong to a component of $(X, \omega) - \bfD$ whose only horizontal cylinders belong to $\bfC_1$ by Sublemma \ref{SL:StructureOfC1-0}). Therefore, we can form a surface on which the cylinders in $\bfC_1$ persist and contain three cylinders of moduli $m_+$, $m_-$, and $m_0$ that satisfy no rational linear relations. Let $(X', \omega')$ denote this new surface and $\bfC_3$ the equivalence class containing these three cylinders. Without loss of generality, by Lemma \ref{L:MovingToCS}, we may suppose that $(X', \omega')$ is also cylindrically stable. Since $\Twist(\bfC_3, \cM)$ contains three cylinders with moduli satisfying no rational linear relation it follows that $\Twist(\bfC_3, \cM)$ is at least three dimensional and has two dimensions of rel (by Mirzakhani-Wright \cite[Theorem 1.5]{MirWri}). But then every horizontal cylinder on $(X', \omega')$ belongs to $\bfC_3$ (by Lemma \ref{L:RelInHyp}), contradicting the fact that $\cM$ has rank at least two and that $(X', \omega')$ is cylindrically stable (see Apisa-Wright \cite[Lemma 7.10 (4)]{ApisaWrightHighRank}). Therefore, $\bfC_+ = \bfD$; all the horizontal cylinders in $S$ belong to $\bfC_-$; and, by Sublemma \ref{SL:StructureOfC1-0}, every cylinder in $\bfC_-$ belongs to $S$.
\end{proof}
\end{proof}

\begin{lem}\label{L:Rigid:NotAdjacent}
If $\cM$ is a cylinder rigid invariant subvariety of dimension at least three in a codimension zero or one hyperelliptic locus, then two generic subequivalent cylinders are glued together along an IET if and only if they share a boundary saddle connection. 
\end{lem}

In the sequel we will employ both cylinder collapses, which we already encountered in Section \ref{S:Joinings}, and more general cylinder degenerations. Given a collection of cylinders $\bfC$ on a translation surface $(X, \omega)$ in an invariant subvariety $\cM$ we will denote the more general cylinder degenerations by $\Col_v(X, \omega) \in \cM_v$ where $v \in \mathrm{Twist}(\bfC, \cM)$ (see Apisa-Wright \cite[Section 4.2]{ApisaWrightHighRank} for a precise definition and properties). 

\begin{proof}
Proceed by induction on the dimension of $\cM$. For rank one cylinder rigid invariant subvarieties of dimension at least three this result holds by Proposition \ref{P:Main:Rk1} (it is also explicitly stated for rank one rel one invariant subvarieties in Lemma \ref{L:Main}; note that Proposition \ref{P:Main:Rk1} implies that for rank one rel two invariant subvarieties, generic subequivalent cylinders never border each other).

Suppose now that $\cM$ has rank at least two. Let $\bfC$ be a subequivalence class of cylinders on $(X, \omega) \in \cM$. After perhaps perturbing $(X, \omega)$, let $\bfD$ be a generic subequivalence class that shares no boundary saddle connections with $\bfC$ (this exists by Apisa-Wright \cite[Lemma 3.28]{ApisaWrightDiamonds}). Then $\cM_{\bfD}$ is cylinder rigid (by Apisa \cite[Proposition 4.12]{Apisa-MHD}), prime (by Apisa-Wright \cite[Lemma 9.1]{ApisaWrightHighRank}), and of dimension exactly one less than $\cM$ (by Apisa-Wright \cite[Lemma 6.5]{ApisaWrightHighRank}; this uses the fact that in cylinder rigid subvarieties, subequivalent cylinders have a constant ratio of heights as long as they persist), and hence of dimension at least three (since $\cM$ has rank at least two). The same holds for the projection of $\cM_{\bfD}$ to any component (by the definition of cylinder rigid and Chen-Wright \cite[Theorem 1.3 (2)]{ChenWright}). Finally, the projection of $\cM_{\bfD}$ to any component still belongs to a codimension zero or one hyperelliptic locus (by Proposition \ref{P:HypBdry}). By the induction hypothesis, two cylinders in $\Col_{\bfD}(\bfC)$ share a boundary saddle connection if and only if they are glued together along an IET. Since $\bfD$ is disjoint from $\bfC$ and shares no boundary saddle connections with it, the same is true of $\bfC$, as desired. 
\end{proof}

\begin{cor}\label{C:Not-Rigid:C-SameHeight}
If $\cM$ is an invariant subvariety of rank at least two in a codimension zero or one hyperelliptic locus and $\bfC$ is an equivalence class of cylinders on $(X, \omega) \in \cM$ so that $\Twist(\bfC, \cM)$ contains a rel vector, then, using the notation of Lemma \ref{L:RelInEC}, all the cylinders in $\bfC_-$ have identical heights.
\end{cor}

It may be useful to keep the surface depicted on the top of Figure \ref{F:Theorem1-5} in mind for the following proof.

\begin{proof}
Let $v$ be the rel vector as in the notation of Lemma \ref{L:RelInEC}. By Lemma \ref{L:MovingToCS} we may assume that $(X, \omega)$ is horizontally periodic and cylindrically stable. Assume additionally that two cylinders in $\bfC$ have the same height on $(X, \omega)$ if and only if they have generically identical heights on $\cM$. This perturbation of the original surface may cause the original equivalence class to be contained in a larger equivalence class, which will now be called $\bfC$. The fact that $\mathrm{Twist}(\bfC, \cM)$ contains a rel vector has not changed. 

Since the cylinders in $\bfC$ are generic, the same holds for $\Col_v(\bfC)$ (by Apisa-Wright \cite[Lemma 6.5]{ApisaWright} and the fact that $v = i(\tau_{\bfC_+} - \tau_{\bfC_-})$ by Lemma \ref{L:RelInEC} \eqref{I:Rel}). Moreover, $\cM_v$ has dimension exactly one less than that of $\cM$. Since $\cM_v$ has the same rank as $\cM$ (since $(X, \omega)$ is cylindrically stable and every horizontal equivalence class persists on $\Col_v(X, \omega)$), it follows that $\cM_v$ has rel exactly one less than that of $\cM$.

Suppose first that $\cM$ has rel one. Then $\cM_v$ has rel zero and the surfaces in $\cM_v$ are connected (by Lemma \ref{L:RelInEC}). Since $\Col_v(\bfC_+)$ contains two cylinders that border one another, these two cylinders are glued together along an IET by Lemma \ref{L:Rigid:NotAdjacent} (note that any rel zero invariant subvariety is cylinder rigid by Mirzakhani-Wright \cite{MirWri}, see Apisa \cite[Lemma 4.2]{Apisa-MHD}). This implies that every cylinder in $\bfC_-$ has the same height. 

Suppose now that $\cM$ has rel two. By Lemma \ref{L:RelInEC}, since $\cM$ has rank at least two, $\mathrm{Twist}(\bfC, \cM)$ contains exactly one dimension of rel. Since $(X, \omega)$ is cylindrically stable there is a horizontal equivalence class $\bfD \ne \bfC$ that is involved in rel (see Apisa-Wright \cite[Definition 11.3]{ApisaWrightHighRank}). Perturb $(X, \omega)$ so that $\bfC$ and $\bfD$ become generic and let $w \in \mathrm{Twist}(\bfD, \cM)$ be typical (such a $w$ exists by Apisa-Wright \cite[Lemma 11.2]{ApisaWrightHighRank}, see Apisa-Wright \cite[Lemma 11.1]{ApisaWrightHighRank} for a definition). Therefore, $\cM_w$ has the same rank as $\cM$ (by Apisa-Wright \cite[Lemma 11.4]{ApisaWrightHighRank}) and is prime (by \cite[Lemma 9.1]{ApisaWrightHighRank}). In particular, the projection of $\cM_w$ to each component has rank at least two (by Chen-Wright \cite[Theorem 1.3 (3)]{ChenWright}) and is contained in a hyperelliptic locus of codimension zero or one (by Proposition \ref{P:HypBdry}). Since $\mathrm{Twist}(\Col_w(\bfC), \cM_w)$ can be identified with $\mathrm{Twist}(\bfC, \cM)$, we can identify $v$ with a rel deformation in $\mathrm{Twist}(\Col_w(\bfC), \cM_w)$. It follows from the rel one case established in the previous paragraph that all the cylinders in $\Col_w(\bfC_-)$ have the same height. Since $\bfD$ is disjoint from $\bfC$, all the cylinders in $\bfC_-$ have the same height.
\end{proof}

We will now record the main result of the next section. 

\begin{thm}\label{T:Main:Rigid:Simple}
If $\cM$ is a cylinder rigid invariant subvariety of rank at least two in a codimension zero or one hyperelliptic locus, then $\cM$ is a full locus of covers of a genus zero stratum. 

Moreover, if $\bfC$ is a generic subequivalence class of cylinders on $(X, \omega)$ containing two cylinders glued together along an IET and $\pi$ is the $\cM$-minimal translation cover, then $\pi(\bfC)$ consists of two simple cylinders that border one another and where the only marked point on their common boundary is a Weierstrass point.
\end{thm}

The proof of Theorem \ref{T:Main:Rigid:Simple} will be deferred to the next section and we will now show how Theorem \ref{T:Main:Rigid:Simple} implies Theorem \ref{T:Main}.

\begin{proof}[Proof of Theorem \ref{T:Main} assuming Theorem \ref{T:Main:Rigid:Simple}:]
By passing to the locus of holonomy double covers, it suffices to show that if $\cM$ is a rank at least two invariant subvariety in a hyperelliptic locus of codimension at most one, then $\cM$ is a locus of covers. We will proceed by induction with Theorem \ref{T:Main:Rigid:Simple} providing the base case and allowing us to assume that $\cM$ is not cylinder rigid. 



Since $\cM$ is not cylinder rigid we may suppose without loss of generality that there is an equivalence class of cylinders $\bfC_1$ on a surface $(X, \omega)$ in $\cM$ so that $\Twist(\bfC_1, \cM)$ contains rel. By Lemma \ref{L:MovingToCS}, we may assume that $(X, \omega)$ is a horizontally periodic cylindrically stable surface with horizontal equivalence classes $\bfC_1, \hdots, \bfC_n$. Let $v$ be a purely imaginary rel vector in $\Twist(\bfC_1, \cM)$. Let $\bfC_+$ (resp. $\bfC_-$, $\bfC_0$) be the cylinders in $\bfC_1$ whose heights increase (resp. decrease, remain the same) along the path $(X, \omega) + tv$ for $t > 0$. By Lemma \ref{L:RelInEC}, up to replacing $v$ with $-v$, we have that: $\bfC_+$ consists of two cylinders exchanged by the hyperelliptic involution; every cylinder in $\bfC_-$ only borders cylinders in $\bfC_+$; and $v = \tau_{\bfC_+} - \tau_{\bfC_-}$. Moreover, every cylinder in $\bfC_-$ has the same height (by Corollary \ref{C:Not-Rigid:C-SameHeight}).

\begin{lem}\label{L:StructureOfMv}
$\cM_v$ has the same rank as $\cM$ and is cylinder rigid. Moreover, the surfaces in $\cM_v$ are connected.
\end{lem}
\begin{proof}
The final claim is immediate. The rank of $\cM_v$ is the same as that of $\cM$ since $(X, \omega)$ is cylindrically stable and every horizontal equivalence class persists on $\Col_v(X, \omega)$. By Apisa-Wright \cite[Lemma 6.5]{ApisaWrightHighRank}, $\cM_v$ has dimension exactly one less than that of $\cM$ and hence has rel exactly one less than that of $\cM$. It remains to show that $\cM_v$ is cylinder rigid.

If $\cM_v$ has rel zero then it is cylinder rigid (by Mirzakhani-Wright \cite{MirWri}, see Apisa \cite[Lemma 4.2]{Apisa-MHD}). If not, then $\cM_v$ has rel one and $\cM$ has rel two (note that two is the maximum rel of an invariant subvariety in a codimension zero or one hyperelliptic locus). Since $\cM$ has rank two, Lemma \ref{L:RelInEC} implies that a linear combination of $\{\sigma_{\bfC_i}\}_{i=1}^n$ is rel on $(X, \omega)$ and remains rel on $\cM_v$. 

The codomain of the $\cM_v$-minimal half-translation cover has no free marked points, since, if it did, no rel vector would be a linear combination of standard shears in multiple parallel equivalence classes. By the induction hypothesis, $\cM_v$ is a full locus of covers of a genus zero stratum $\cQ$ of rank at least two. We have just shown that $\cQ$ has no free marked points and hence (by Apisa-Wright \cite[Theorem 1.4 and Corollary 5.2]{ApisaWright} no marked points. 

By Lemma \ref{L:HypParallelism}, since $\cQ$ has no marked points, each cylinder is its own equivalence class. It follows that each equivalence class of cylinders on surfaces in $\cM_v$ are the preimage of one cylinder under the $\cM_v$-minimal half-translation cover. In particular, the twist space of every equivalence class of cylinders on a surface in $\cM_v$ is one-dimensional and hence $\cM_v$ is cylinder rigid. 
\end{proof}

For the remainder of the proof, it will be useful to keep Figure \ref{F:Theorem1-5} in mind.

We will perturb slightly in order to assume that the cylinders in $\bfC_1$ are generic, which implies that the same holds for the cylinders in $\Col_v(\bfC_1)$ (by Apisa-Wright \cite[Lemma 6.5]{ApisaWright}). Since $\cM_v$ is cylinder rigid, the $\cM_v$-minimal translation cover $\pi: \Col_v(X, \omega) \ra (Y, \eta)$ sends $\Col_v(\bfC_1)$ to two simple cylinders that border one another and where the only marked point on their common boundary is a Weierstrass point $p$ (by Theorem \ref{T:Main:Rigid:Simple}). Let $C$ denote the cylinder formed from these two cylinders once the marked point $p$ forgotten.

On $(X, \omega)$, the cylinders in $\bfC_+$ (resp. $\bfC_-$) all have height $h_+$ (resp. $h_-$). Let $(Y', \eta')$ be the surface $(Y, \eta)$ with two points $\{q, q'\}$ lying at height $h_-/2$ above and below $p$ on the vertical separatrix passing through it marked and with $p$ forgotten. Let $D$ be the subcylinder contained in $C$ on $(Y', \eta')$ with boundary formed by the loops joining $q$ to itself and $q'$ to itself. We will now construct a translation cover $f: (X, \omega) \ra (Y', \eta')$. 

\begin{figure}[h]\centering
\includegraphics[width=\linewidth]{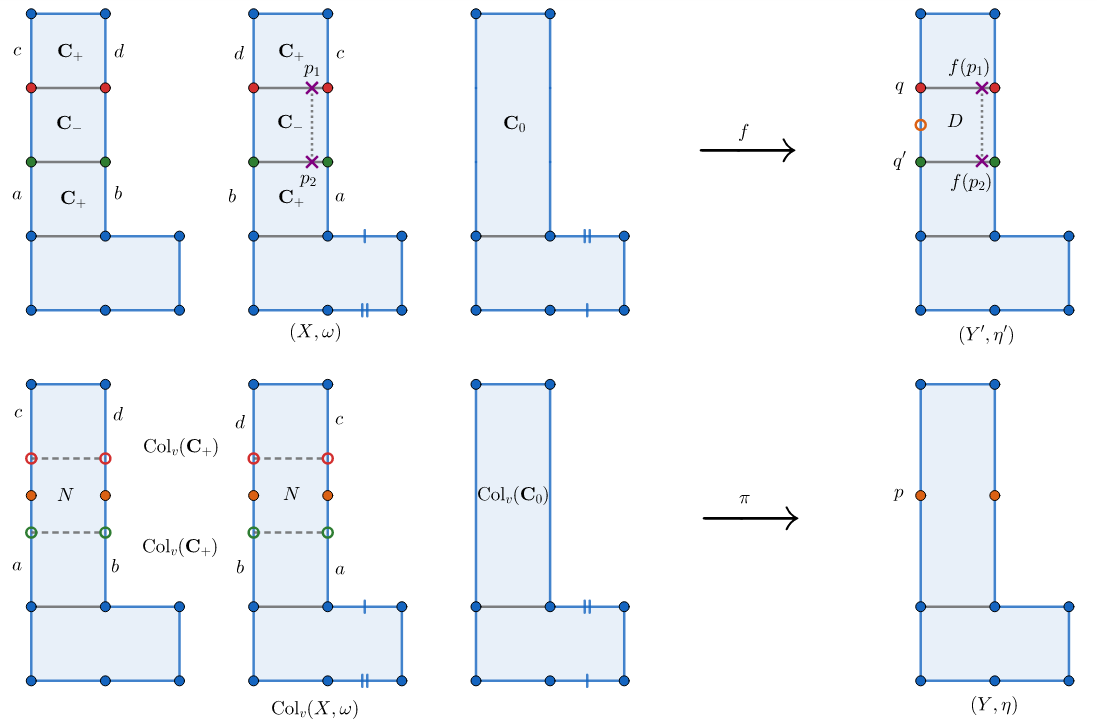}
\caption{This figure depicts the proof of Theorem \ref{T:Main} in a concrete setting. The surface $(X, \omega)$ on the top left belongs to a locus $\cM$ of covers of $\cH(2,0^2)$ whose branch locus consists of the zero of order two and two points exchanged by the hyperelliptic involution. The codomain of this cover is the surface on the top right. All sides are identified to the opposite side unless otherwise indicated. The locus $\cM$ is contained in the codimension one hyperelliptic locus in $\cH(8,1,1)$. On $(X, \omega)$, cylinders belonging to an equivalence class $\bfC$ are labelled. Note that $\mathrm{Twist}(\bfC, \cM)$ contains a purely imaginary rel vector $v$. Whether the heights of cylinders in $\bfC$ increase, decrease, or remain constant along the path $(X, \omega) + tv$ for $t > 0$ are indicated by the subscript. The surface on the bottom left is $\Col_v(X, \omega)$, which belongs to a locus $\cM_v$ of covers of $\cH(2,0)$ whose branch locus consists of the zero of order two and a distinct Weierstrass point. Unshaded points and points marked with an ``x" are unmarked. Dashed and dotted lines indicate ``saddle connections" that join unmarked points.}
\label{F:Theorem1-5}
\end{figure}

Let $N$ be set of points of distance at most $\frac{h_-}{2}$ from the shared boundary of the two cylinders in $\Col_v(\bfC_+)$. Notice that, by Lemma \ref{L:RelInEC} \eqref{I:Rel}, we may identify $\Col_v(X, \omega) - N$ with $(X, \omega) - {\bfC_-}$. Therefore, using this identification, we will define $f$ to agree with $\pi$ on this subset. At this point, $f: (X, \omega) - {\bfC_1} \ra (Y', \eta')$ is a translation cover. (Note that no marked point or zero is sent to $p$ under $f$ since, up to the action of the hyperelliptic involution, there is only one zero or marked point on $(X, \omega)$ that does not lie on the shared boundary of cylinders in $\bfC_-$ and $\bfC_+$; it is straightforward to see that this point is not sent to $p$.)

By construction if two points $p_1$ and $p_2$ on the boundary of a cylinder $\bfC_-$ are joined by a vertical line segment contained in $\bfC_-$, then this line segment has length $h_-$ and $f(p_1)$ and $f(p_2)$ are opposite endpoints of a vertical line segment of length $h_-$ contained in $\overline{D}$. This observation implies that our partial definition of $f$ can be extended over $\bfC_1$ by sending the line segment joining $p_1$ to $p_2$ on $(X, \omega)$ to the one joining $f(p_1)$ to $f(p_2)$ in $\overline{D}$. This extension is the desired translation cover.  
%
\end{proof}

%
%
%

\section{Proof of Theorem \ref{T:Main:Rigid:Simple}}\label{S:ProofTMain:Rigid}

The entirety of this section will be devoted to a proof of Theorem \ref{T:Main:Rigid:Simple}. 

Say that two generic subequivalent cylinders $C$ and $C'$ are \emph{well-behaved} if one of the following occurs:
\begin{enumerate}
    \item Neither cylinder is glued to another cylinder along an IET and both cylinders have the same height.
    \item Both cylinders are glued to each other along an IET and have the same height.
    \item $C$ is glued to another cylinder, $C'' \ne C'$, along an IET and has half the height of $C'$.
\end{enumerate}

\begin{lem}\label{L:Well-behaved}
In a cylinder rigid invariant subvariety $\cM$ of dimension at least three in a codimension zero or one hyperelliptic locus, any two generic subequivalent cylinders are well-behaved.
\end{lem}
\begin{proof}
Proceed by induction on the dimension of $\cM$. When $\cM$ has rank one, the result follows from Proposition \ref{P:Main:Rk1} (note that the in the rank one rel one case the result is contained more explicitly in Lemma \ref{L:Main}; in the rank one rel two case, distinct generic subequivalent cylinders do not border one another and they have the same height). Therefore, assume that $\cM$ has rank at least two. Let $\bfC$ be a generic equivalence class of cylinders on a surface $(X, \omega) \in \cM$.

By definition of cylinder rigid, the cylinders in $\bfC$ have a constant ratio of heights along any path in $\cM$ along which the cylinders persist. By Lemma \ref{L:MovingToCS} there is a path from $(X, \omega)$ to a cylindrically stable surface $(Y, \eta)$ along which $\bfC$ persists and, consequently, so that on $(Y, \eta)$ two cylinders in $\bfC$ have the same ratios of heights as on $(X, \omega)$. 

Suppose that there are $n$ horizontal subequivalence classes of cylinders on $(Y, \eta)$ and label them $\{1, \hdots, n\}$. Suppose that $\bfC$ corresponds to the cylinders labelled $1$. Let $\Gamma$ be the cylinder graph of $(Y, \eta)$ (see Definition \ref{D:Graphlike}). Label each vertex of $\Gamma$ according to the subequivalence class to which the corresponding cylinder belongs. If $\alpha$ is a number in $\{1, \hdots, n\}$, then let $\Gamma_\alpha$ denote $\Gamma$ with all the vertices labelled $\alpha$ deleted. 

If a vertex labelled $i$ is connected by an edge in $\Gamma$ to a vertex labelled $j$, for $i \ne j$, then every vertex labelled $i$ is connected by an edge to a vertex labelled $j$ (by Apisa \cite[Lemma 4.4]{Apisa-MHD}). It follows that there are two distinct numbers $\alpha$ and $\beta$ in $\{1, \hdots, n\}$ so that $\Gamma_\alpha$ and $\Gamma_\beta$ are a union of connected components each of which contains at least one vertex labelled by each number in $\{1, \hdots, n\} - \alpha$ (resp. $\{1, \hdots, n\} - \beta$).  

Let $\bfC_\alpha$ and $\bfC_\beta$ denote the subequivalence classes of cylinders corresponding to $\alpha$ and $\beta$. Perturb $(Y, \eta)$ slightly to a surface $(Y', \eta')$ on which $\bfC_\alpha$ and $\bfC_\beta$ are generic and all the cylinders that were horizontal on $(Y, \eta)$ persist; call them $\bfH$. Notice that every component of $(Y', \eta') - \bfC_\alpha$ contains a cylinder in $\bfH$, since, if not, this component contains a cylinder disjoint from any in $\bfH$ (by Apisa-Wright \cite[Lemma 8.3]{ApisaWrightHighRank}), which implies that there is a cylinder subequivalence class disjoint from those in $\bfH$ (by Apisa-Wright \cite[Corollary 8.4]{ApisaWrightHighRank}), which contradicts the fact that $(Y, \eta)$ is cylindrically stable (see Apisa-Wright \cite[Lemma 7.10 (2)]{ApisaWrightHighRank}). Therefore, by our choice of $\alpha$ and $\beta$, every component of $(Y', \eta') - \bfC_\alpha$ contains a cylinder from every subequivalence class in $\bfH - \bfC_\alpha$ (the same holds with $\alpha$ replaced by $\beta$). 

Since $\cM$ belongs to a codimension at most one hyperelliptic locus, at most one of $\bfC_\alpha$ and $\bfC_\beta$ contain a pair of cylinders exchanged by the hyperelliptic involution (by Lemma \ref{L:CutReglue}). Without loss of generality $\bfC_\alpha$ contains no such pair of cylinders and hence contains no pair of cylinders glued together along an IET. By Lemma \ref{L:Rigid:NotAdjacent}, no two cylinders in $\bfC_\alpha$ share boundary saddle connections. 

First, we observe that if two subequivalent cylinders belong to the same component of $(Y', \eta') - \bfC_\alpha$ (and are not equivalent to $\bfC_\alpha$) then they are well-behaved. This follows from the induction hypothesis, since they end up on the same component of $\Col_{\bfC_\alpha}(Y', \eta')$ (notice that if the cylinders were equivalent to $\bfC_\alpha$ then, a priori, it would be possible for them to be glued together along an IET on $\Col_{\bfC_\alpha}(Y', \eta')$, but not on $(Y', \eta')$; this is the reason we exclude cylinders equivalent to $\bfC_\alpha$ from the analysis here). Note that $\cM_{\bfC_{\alpha}}$ has codimension one (by Apisa-Wright \cite[Lemma 6.5]{ApisaWrightHighRank}, since $\bfC_\alpha$ is a generic subequivalence class).

Now suppose that $S_1$ and $S_2$ are two components of $(Y', \eta') - \bfC_\alpha$ and suppose that both components border a cylinder $C \in \bfC_{\alpha}$. Suppose that $S_1$ borders $C$ on its top boundary and $S_2$ borders $C$ along its bottom boundary. There is a shear $u$ in the direction of $\bfC_{\alpha}$ so $S_1$ and $S_2$ belong to the same component of $\Col_{\bfC_{\alpha}}(u \cdot (Y', \eta'))$. The induction hypothesis therefore implies that a cylinder in $S_1$ and a subequivalent cylinder (not equivalent to $\bfC_\alpha$) in $S_2$ are well-behaved when the cylinders are not equivalent to those in $\bfC_\alpha$. It follows that if two components of $(Y', \eta') - \bfC_\alpha$ border the same cylinder in $\bfC_\alpha$, then two subequivalent cylinders (not equivalent to $\bfC_\alpha$) in each component are well-behaved. Since no two cylinders in $\bfC_\alpha$ share a boundary saddle connection it follows that any two subequivalent cylinders (not equivalent to $\bfC_\alpha$) in $\bfH - \bfC_{\alpha}$ are well-behaved. It remains to show that any two subequivalent cylinders equivalent to those in $\bfC_\alpha$ are well-behaved.

If $\bfC_\beta$ is not equivalent to $\bfC_\alpha$ and does not contain two cylinders glued together along an IET, then this is immediate. Suppose now that $\bfC_\beta$ is equivalent to $\bfC_\alpha$. By Lemma \ref{L:RelInEC}, up to exchanging $\alpha$ and $\beta$, deleting the vertices labelled $\bfC_\alpha$ from the cylinder graph produces two types of components - those consisting solely of cylinders in $\bfC_\beta$ and those with no cylinders in $\bfC_\beta$. This contradicts the way that we chose $\alpha$ and $\beta$. Therefore, we may suppose that $\bfC_\beta$ is not equivalent to $\bfC_\alpha$ and that it contains two cylinders glued together along an IET.

We note that if $\cM$ has rel then $\Twist((X, \omega), \cM)$ has one dimension of rel (by Lemma \ref{L:RelInHyp} and the fact that $\bfC_\beta$ contains two cylinders glued together along an IET). By definition of cylinder rigidity, the rel vector in the tangent space of $\cM$ is a linear combination of $\{\sigma_{\bfC_i}\}_{i=1}^n$, no coefficient of which is zero (by Lemma \ref{L:RelInHyp}). Therefore, any two subequivalent cylinders outside of $\bfC_\beta$ have identical heights (by Lemma \ref{L:RelInHyp}). 

It remains to consider the case that $\cM$ has no rel. In this case there is a purely imaginary element $v \in T_{(X, \omega)} \cM$ that pairs trivially with every horizontal cylinder on $(X, \omega)$ except those in $\bfC_\beta$ (by Avila-Eskin-M\"oller \cite{AEM}). Therefore, for sufficiently small $t$, $(X', \omega') := (X, \omega) + tv$ is a surface on which all the horizontal cylinders on $(X, \omega)$ persist and, with the exception of those in $\bfC_\beta$, remain horizontal. The cylinders in $\bfC_\beta$ cease to be horizontal on $(X', \omega')$. By applying a shear to the surface to form a new surface $(X'', \omega'')$ we may suppose without loss of generality that  $\bfC_\beta$ has become vertical and, after applying the standard shear in $\bfC_\beta$, that $\bfC_\beta$ contains a horizontal saddle connection $s$. Suppose without loss of generality that $s$ joins two distinct zeros that are not exchanged by the hyperelliptic involution; which is possible since $\bfC_\beta$ contains two cylinders glued together along an IET.  

\begin{sublem}
On $(X'', \omega'')$ there is a horizontal equivalence class $\bfC_\gamma$ disjoint from those in $\bfH - \bfC_\beta$.
\end{sublem}
\begin{proof}
It suffices to show that $\Col_{\bfC_\beta}(X'', \omega'')$ is horizontally periodic. If not, then it is possible to find a cylinder $C$ in the complement of $\Col_{\bfC_\beta}(\bfH - \bfC_\beta)$ (by Smillie-Weiss \cite[Corollary 6]{SW2}). Since $C$ intersects no cylinder in $\Col_{\bfC_\beta}(\bfH - \bfC_\beta)$, the same is true for every cylinder subequivalent to $C$ (by the cylinder proportion theorem, see Nguyen-Wright \cite[Proposition 3.2]{NW}). Call this subequivalence class $\bfC_\delta$. 

Since cylinders in $\Col_{\bfC_\beta}(\bfH - \bfC_\beta)$ are equivalent if and only if the corresponding cylinders in $\bfH$ were equivalent, it follows that the standard shears on the subequivalence classes in $\Col_{\bfC_\beta}(\bfH - \bfC_\beta)$ project to a Lagrangian subspace of the projection $p$ of $T_{\Col_{\bfC_\beta}(X, \omega)} \cM_{\bfC_\beta}$ to absolute homology (since the rank of $\cM_{\bfC_\beta}$ is at most $\mathrm{rank}(\cM)-1$ since $\cM$ is rel zero). Therefore, $p(\sigma_{\bfC_\delta})$ is a linear combination of $\{ p(\sigma_{\bfC}) \}_{\bfC \in \Col_{\bfC_\beta}(\bfH)}$. Hence, the cylinders in $\bfC_\delta$ must be horizontal since the cylinders in $\Col_{\bfC_\beta}(\bfH)$ are. This is a contradiction. \end{proof}

Since $(X'', \omega'')$ has $\mathrm{rank}(\cM)$ many horizontal equivalence classes of cylinders, it follows that $(X'', \omega'')$ is horizontally periodic. Since $\bfC_\beta$ is a subset of $\overline{\bfC_\gamma}$, it follows that $s$ - the horizontal saddle connection joining two distinct zeros not exchanged by the involution - lies on the boundary of some cylinder in $\bfC_\gamma$. However, if two horizontal cylinders on $(X'', \omega'')$ were glued together along an IET, then no such saddle connection could belong to the boundary of a horizontal cylinder (since $\cM$ belongs to a codimension zero or one hyperelliptic locus). It follows that every horizontal cylinder on $(X'', \omega'')$ is fixed by the hyperelliptic involution. 

Notice that, by construction, the cylinder graph of $(X'', \omega'')$ with the vertices corresponding to $\bfC_\gamma$ deleted is precisely the cylinder graph of $(X, \omega)$ with the vertices corresponding to $\bfC_\beta$ deleted (since $\cM$ has rel zero, the perturbation from $(X, \omega)$ to $(X'', \omega'')$ did not cause new cylinders to become equivalent to those in $\bfH - \bfC_\beta$). In particular, each component contains at least one vertex labelled by every number in $\{1, \hdots, n\} - \beta$. Since we have shown that $\bfC_\gamma$ does not contain two cylinders glued together along an IET, all the hypotheses that held for $\bfC_\alpha$ on $(X, \omega)$ hold for $\bfC_\gamma$ on $(X'', \omega'')$. Therefore, we have already seen that any two subequivalent horizontal cylinders on $(X'', \omega'')$ that do not belong to $\bfC_\gamma$ are well-behaved. In particular, any two cylinders in $\bfC_\alpha$ are well-behaved, as desired.
\end{proof}

Before proceeding we will require a short technical lemma.

\begin{lem}\label{L:CRCover:PreimageOfImage}
Suppose that $\cM$ is a cylinder rigid invariant subvariety of rank at least two that is a full locus of covers of a stratum under its $\cM$-minimal (half)-translation cover. Under this cover, every subequivalence class of cylinders is sent to a single cylinder and its closure is the full preimage of its image.
\end{lem}
\begin{proof}
Let $\bfC$ be a subequivalence class of cylinders on $(X, \omega) \in \cM$. Let $\pi: (X, \omega) \ra (Y, q)$ be the $\cM$-minimal (half)-translation cover. Since $\cM$ is a full locus of covers of a stratum, after perturbing, we may assume that the cylinders in $\pi(\bfC)$ have moduli that satisfy no rational linear relation. The moduli of the cylinders in $\bfC$ are integral linear combinations of these moduli. By definition of subequivalence class, all cylinders in $\bfC$ have commensurable moduli. Since the preimage of each cylinder in $\pi(\bfC)$ appears as a subcylinder of some cylinder in $\bfC$, it follows that for each cylinder $C$ in $\bfC$, $\pi: \overline{C} \ra \overline{\pi(\bfC)}$ is a surjection. 

In particular, the interior of $\overline{\pi(\bfC)}$ is a cylinder (although not necessarily a maximal cylinder) after forgetting marked points. Call this cylinder $D$ and let $h$ denote its height. Without loss of generality, after perhaps perturbing, $D$ is generic and so, since $\cM$ has rank at least two, $\overline{D} \ne (Y, q)$. Since each cylinder in $\bfC$ surjects onto $D$ and since $\overline{D} \ne (Y, q)$, each cylinder in $\bfC$ has height $h$ or $2h$ (the second possibility may occur when $D$ is an envelope). 


Let $h_0$ denote the height of the tallest cylinder parallel to $\bfC$ but not contained in it. By applying a standard dilation to $\bfC$ we may assume that $\frac{h}{n} > h_0$ where $n$ is the number of cylinders in $\pi(\bfC)$. This implies that there is a cylinder $E$ in $\pi(\bfC)$ whose height is larger than $h_0$ and hence so that $\pi^{-1}(E) \subseteq \bfC$. It suffices to show that $E = D$. If this is not the case, then there is a point $p$ on the boundary of $E$ that lies in the interior of $D$. Without loss of generality suppose that $p$ lies on the top boundary of $E$. 

There is at least one point $P$ in $\pi^{-1}(p)$ that is marked. Let $E'$ be a component of $\pi^{-1}(E)$ so that $P$ belongs to its top boundary. Let $C$ be the cylinder in $\bfC$ containing $E'$. Since the top boundary of $C$ contains $P$, it follows that $\pi(C)$ only contains points in $D$ that lie below $p$. This contradicts the fact that $\pi: \overline{C} \ra \overline{\pi(C)}$ is a surjection. Therefore, $E = D$ as desired. 

To summarize, $\overline{E} = \overline{\pi(\bfC)}$ and $\pi^{-1}(E) \subseteq \bfC$. Therefore, $\pi^{-1}(\pi(\overline{\bfC})) = \overline{\bfC}$. Moreover, for each cylinder $C \in \bfC$, the interior of $\pi(C)$ is $D$.
\end{proof}

We will now require a slight modification of the definition of $\cM$-minimal. 

\begin{defn}\label{D:M-littlest}
Let $\cM$ be an invariant subvariety containing connected surfaces. Given $(X, \omega)$ in $\cM$ the \emph{$\cM$-littlest translation cover} will be defined to be the $\cM$-minimal translation cover except in the case that $\cM_{min}$ is an Abelian double of $\cH(0,0)$, in which case, it will be defined to be the composition of the $\cM$-minimal translation cover with the double cover. 

Now suppose that $\cM \subseteq \cH_1 \times \hdots \times \cH_n$ is an invariant subvariety where $\cH_i$ is a connected component of a stratum for all $i$. Let $\cM_i$ denote the closure of the projection of $\cM$ onto the $i$th component (as in Chen-Wright \cite[Theorem 1.3 (1)]{ChenWright}). If $(X, \omega) \in \cM$ we will say that the $\cM$-littlest translation cover is the  map whose codomain also has $n$ connected components and so that its restriction to the $i$th component is simply the $\cM_i$-littlest translation cover whose domain is that component. 

Define $\cM_{lit}$ to be the collection of codomains of $\cM$-littlest maps with domain in $\cM$.
\end{defn}

The following result is a strengthening of Theorem \ref{T:Main:Rigid:Simple} and the remainder of the section will be devoted to its proof.

\begin{thm}\label{T:Main:Rigid}
If $\cM$ is a cylinder rigid invariant subvariety of dimesion at least three in a codimension zero or one hyperelliptic locus, then the following hold.
\begin{enumerate}
    \item\label{I:FullLocus} $\cM_{lit}$ is a quadratic double of a genus zero stratum with at most two zeros.
    \item\label{I:PiMinDegen} If $\cM$ is at least four-dimensional, $\bfC$ is a generic subequivalence class of cylinders on $(X, \omega) \in \cM$, and $\pi$ is the $\cM$-littlest translation cover, then, when restricted to any component of $\Col_{\bfC}(X, \omega)$, $\Col_{\bfC}(\pi)$ is the $\cM_{\bfC}$-littlest translation cover. 
\end{enumerate}
\end{thm}
\begin{rem}\label{R:PreimageOfImage}
Note that in \eqref{I:PiMinDegen}, for $\Col_{\bfC}(\pi)$ to be defined we must have that the closure of every generic subequivalence class $\bfC$ is the preimage of its image under $\pi$ (see Apisa-Wright \cite[Lemma 2.2]{ApisaWrightDiamonds}). In the case that $\cM$ has rank at least two, this follows from \eqref{I:FullLocus} and Lemma \ref{L:CRCover:PreimageOfImage}. When $\cM$ is rank one and four-dimensional, this follows from Proposition \ref{P:Main:Rk1}.
\end{rem}

First we will deduce Theorem \ref{T:Main:Rigid:Simple}. 



\begin{proof}[Proof of Theorem \ref{T:Main:Rigid:Simple} assuming Theorem \ref{T:Main:Rigid}:]
The claim that $\cM$ is a full locus of covers of a genus zero stratum with at most two zeros is Theorem \ref{T:Main:Rigid} \eqref{I:FullLocus}. The claim about the image of a generic subequivalence class under the $\cM$-minimal map follows from Lemma \ref{L:CRCover:PreimageOfImage}.
\end{proof}

\begin{proof}

We will proceed by induction on the dimension of $\cM$. The base case is Proposition \ref{P:Main:Rk1}. The following result will power the induction.

\begin{lem}\label{L:SidesAreCovers}
If $\bfC$ is a generic subequivalence class of cylinders on $(X, \omega) \in \cM$, then $\cM_{\bfC}$ is a full locus of covers of a genus zero stratum. Moreover, letting $(\cM_{\bfC})_i$ denote the closure of the projection of $\cM_{\bfC}$ onto the $i$th component of $\Col_{\bfC}(X, \omega)$ (as in Definition \ref{D:M-littlest}), the codomain of the $(\cM_{\bfC})_i$-littlest translation cover is independent of $i$ up to forgetting fixed points of the hyperelliptic involution. 
\end{lem}
\begin{proof}
Suppose that $\cM_\bfC \subseteq \cH_1 \times \hdots \times \cH_n$ where $\cH_i$ is a component of a stratum. Since $\cM_{\bfC}$ is prime (by Apisa-Wright \cite[Lemma 9.1]{ApisaWrightHighRank}) and cylinder rigid (by Apisa \cite[Proposition 4.12]{Apisa-MHD}), it follows that the rank, rel, and dimension of $(\cM_{\bfC})_i$ is independent of $i$ (by Lemma \ref{L:SameConstants}).


Since $\cM_{\bfC}$ has dimension exactly one less than $\cM$ (by Apisa-Wright \cite[Lemma 6.5]{ApisaWrightHighRank}), the same holds for $(\cM_{\bfC})_i$ for all $i$. In particular, each $(\cM_{\bfC})_i$ has dimension at least three and hence, by the induction hypothesis, specifically by \eqref{I:FullLocus}, $((\cM_{\bfC})_i)_{lit}$ is a locus $\wt{\cQ}_i$ of holonomy double covers of a genus zero stratum $\cQ_i$. Define the \emph{$\cM_{\bfC}$-littlest half-translation cover} to be the composition of the $\cM_{\bfC}$-littlest translation cover with the holonomy double cover on each component.


\begin{sublem}
$(\cM_{\bfC})_{lit} \subseteq \wt{\cQ}_1 \times \hdots \times \wt{\cQ}_n$ is a cylinder rigid prime invariant subvariety whose projection to each factor is dense.
\end{sublem}
\begin{proof}
We have already shown that $(\cM_{\bfC})_{lit}$ projects densely onto each factor. The fact that it is prime follows from the fact that $\cM_{\bfC}$ is prime. Let $\bfA$ be a generic equivalence class of horizontal cylinders on a surface $(Y, \eta) \in \cM_{\bfC}$. Let $\pi: (Y, \eta) \ra (Z, \zeta)$ be the $\cM_{\bfC}$-littlest translation cover. To show that $(\cM_{\bfC})_{lit}$ is cylinder rigid it will suffice to show that $\bfB := \pi(\bfA)$ is a subequivalence class. 

Let $(Y_t, \eta_t)$ (resp. $(Z_t, \zeta_t)$) be the result of applying the matrix $u_t := \begin{pmatrix} 1 & t \\ 0 & 1 \end{pmatrix}$ to $\bfA$ (resp. $\bfB$) while fixing the rest of the surface.  By the induction hypothesis (see Remark \ref{R:PreimageOfImage}) $\overline{\bfA} = \pi^{-1}(\overline{\bfB})$. This implies that the $\cM_{\bfC}$-littlest map $\pi_t: (Y_t, \eta_t) \ra (Z_t, \zeta_t)$ is defined piecewise as $\pi$ on $(Y_t, \eta_t) - \overline{\bfA}$ and the composition of $u_t$ and $\pi$ on $\overline{\bfA}$. In particular, this implies that the path $(Z_t, \zeta_t)$ belongs to $(\cM_{\bfC})_{lit}$ and hence that the standard shear in $\bfB$ belongs to the tangent space. Since the cylinders in $\bfB$ all have commensurable moduli (since this is true of cylinders in $\bfA$), it follows that $\bfB$ is a subequivalence class. 
\end{proof}


Therefore, up to rescaling the components of the surfaces in it, $(\cM_{\bfC})_{lit}$ is the holonomy double cover of the diagonal embedding of $\cQ$ into a product $\cQ \times \hdots \times \cQ$ ($n$ copies) where $\cQ$ is a genus zero stratum (by Proposition \ref{P:Joinings:Rigid}). 

It remains to show that $(\cM_{\bfC})_{lit}$ is a diagonal embedding without rescaling any component of $\Col_{\bfC}(X, \omega)$. Let $\bfC'$ be any subequivalence class disjoint from $\bfC$. Let $\pi$ be the $\cM$-littlest translation cover on $\Col_{\bfC}(X, \omega)$. Let $\Col_{\bfC}(\bfC')_i$ be the cylinders in $\Col_{\bfC}(\bfC')$ on the $i$th component and let $D_i$ be their image under the $\cM_{\bfC}$-littlest half-translation cover. It suffices to show that the height of $D_i$ is independent of $i$. 

Suppose first that two cylinders in $\Col_{\bfC}(\bfC')$ are glued along an IET if and only if the same two cylinders are glued along an IET in $\bfC'$ (we can always arrange for this to be the case if $\cM$ has rank at least two, see Apisa-Wright \cite[Lemma 3.31]{ApisaWrightDiamonds}, by finding $\bfC'$ that shares no boundary saddle connections with cylinders in $\bfC$). By Lemma \ref{L:Well-behaved}, any two cylinders in $\Col_{\bfC}(\bfC')$ are well-behaved. In other words, there is a constant $h$ so that every cylinder in $\Col_{\bfC}(\bfC')$ has height $h$ (provided that we pretend that two cylinders glued along an IET are a single cylinder). This shows that $\pi(\bfC_i)$ consists of cylinders of height $h$ (again pretending that two cylinders glued along an IET are a single cylinder). In particular, if $D_i$ is a simple cylinder then it has height $h$ and if it is a simple envelope it has height $\frac{h}{2}$. Since every $D_i$ is a simple envelope or every $D_i$ is a simple cylinder, the result follows.

Now suppose that $\cM$ has rank one rel two and that two cylinders in $\bfC'$ are not glued along an IET, but become so in $\Col_{\bfC}(\bfC')$. By Proposition \ref{P:Main:Rk1}, $\bfC'$ consists of two cylinders exchanged by the hyperelliptic involution and so $\cM_{\bfC}$ is connected, so there is nothing to prove. This completes the proof.
%
%
%
\end{proof}

\begin{lem}\label{L:I:FullLocus:BaseCase}
If $\cM$ has rank two rel zero, then it is a full locus of covers of $\cH(2)$ or of the locus in $\cH(2, 0)$ with a Weierstrass point marked.
\end{lem}
\begin{proof}

Let $((X, \omega), \cM, \bfC_1, \bfC_2)$ be a generic diamond (this exists by Apisa-Wright \cite[Lemma 3.31]{ApisaWrightDiamonds}). Without loss of generality, suppose that $\bfC_1$ (resp. $\bfC_2$) is horizontal (resp. vertical) and, after perhaps applying the standard shear, contains a vertical (resp. horizontal) saddle connection. By Apisa \cite[Lemma 9.1]{Apisa2} this implies that there is a second horizontal (resp. vertical) equivalence class $\bfC_1'$ (resp. $\bfC_2'$). See the top surface in Figure \ref{F:LabellingScheme}.

\begin{sublem}
If none of the equivalence classes $\bfC_1, \bfC_1', \bfC_2, \bfC_2'$ contain two cylinders glued together along an IET then $\cM_{min} = \cH(2)$.
\end{sublem}
\begin{proof}
Suppose that none of the equivalence classes $\bfC_1, \bfC_1', \bfC_2, \bfC_2'$ contain two cylinders glued together along an IET. By Lemma \ref{L:Rigid:NotAdjacent}, every cylinder in $\ColTwo(\bfC_1)$ (resp. $\ColOne(\bfC_2)$) only borders cylinders in $\ColTwo(\bfC_1')$ (resp.  $\ColOne(\bfC_2')$). 

By Proposition \ref{P:Main:Rk1}, for each component $j$ of $\Col_{\bfC_i}(X, \omega)$, $(\cM_{\bfC_i})_j$ is a full locus of covers of one of the following: a quadratic double of $\cQ(-1^4, 0)$ with the preimage of exactly one pole marked, $\cH(0,0)$, or an Abelian double of $\cH(0,0)$. In the first case, every cylinder direction comprised of two subequivalence classes contains two distinct subequivalent cylinders glued together along an IET. Our assumptions, imply that this case does not occur. Therefore, by Lemma \ref{L:SidesAreCovers}, $\cM_{\bfC_i}$ is a full locus of covers of a diagonal embedding of $\cH(0,0)$ with the covering map being given by the $\cM_{\bfC_i}$-littlest map. Let $\pi_i$ denote the map whose domain is $\Col_{\bfC_i}(X, \omega)$ that is a composition of the $\cM_{\bfC_i}$-littlest map, whose codomain may have multiple components, with the map that identifies all these components with each other. This map satisfies Assumption CP by Proposition \ref{P:Main:Rk1} \eqref{I:P:Main:Rk1:Rel=1:NoIETCase}. We are done by Apisa-Wright \cite[Lemma 8.31]{ApisaWrightDiamonds}.
%
%
\end{proof}



\begin{sublem}
Without loss of generality $\bfC_2'$ contains two cylinders glued together along an IET.
\end{sublem}
\begin{proof}
We have already seen that we may suppose that at least one of $\bfC_1, \bfC_1', \bfC_2, \bfC_2'$ contains two cylinders glued together along an IET. If this collection is either $\bfC_1'$ or $\bfC_2'$ then up to swapping indices we are done. Suppose therefore, that, up to swapping indices, $\bfC_1$ contains two cylinders glued together along an IET. 

Since $\bfC_1$ contains two cylinders glued together along an IET, $\ColTwo(\bfC_1)$ does as well. Let $(X', \omega')$ denote this component of $\ColTwoX$ and $\cM'$ the projection of $\MTwo$ to the factor corresponding to $(X', \omega')$. By Proposition \ref{P:Main:Rk1} \eqref{I:P:Main:Rk1:Rel=1:IETCase}, $(X', \omega')$ covers a surface $(Y, \eta)$ that is the holonomy double cover of a surface in $\cQ(0, -1^4)$ with the preimage, call it $p$, of exactly one pole (and both preimages of the marked point) marked.  Under this map, $\ColTwo(\bfC_1)$ is sent to two simple horizontal cylinders glued together along an IET and whose boundaries contain $p$. By Proposition \ref{P:Main:Rk1} \eqref{I:P:Main:Rk1:Rel=1:IETCase}, $\overline{\ColTwo(\bfC_1)}$ is the full preimage of its image on $(Y, \eta)$, and so no saddle connection in $\ColTwo(\bfC_2)$ projects to one on $(Y, \eta)$ with $p$ as an endpoint. 


Therefore, after perhaps applying the standard shear in $\bfC_1$ we can ensure that $(Y, \eta)$ is covered by two vertical cylinders that are exchanged by the holonomy involution on $(Y, \eta)$, call them $V_1$ and $V_2$. In other words, after shearing the cylinders in $\bfC_1$, we may assume that the image of $\ColTwo(\bfC_2')$ on $(Y, \eta)$ consists of $V_1$ and $V_2$, both of which have a boundary component consisting of a saddle connection $s$ joining $p$ to itself. We have already observed that the saddle connections in the preimage of $s$ on $(X', \omega')$ contain no saddle connections in $\ColTwo(\bfC_2)$ and so there are two cylinders in $\bfC_2'$ that are glued together along an IET. 
\end{proof}



By Lemma \ref{L:SidesAreCovers}, $(\cM_{\bfC_i})$ is a full locus of covers of $\cQ(-1^4, 0)$. Let 
\[ \pi_i: \Col_{\bfC_i}(X, \omega) \ra (Z_i, \zeta_i) \]
be the covering map and $\pi_i^{WP}$ the same map with domain $\Col_{\bfC_i}(X, \omega)^{WP}$.

\begin{sublem}\label{SL:ACP1}
$(\cM^{WP})_{\bfC_i^{WP}}$ is a full locus of covers of $\cQ(-1^4, 0)$ satisfying Assumption CP. 
\end{sublem}

Recall that given a collection of cylinders $\bfC$ on $(X, \omega)$, $\bfC^{WP}$ refers to the corresponding cylinders on $(X, \omega)^{WP}$ (see Definition \ref{D:WP}).

\begin{proof}
By assumption, $\bfC_2'$ contains two cylinders that are generically glued together along an IET. Call them $\bfD$. Let $(Y, \eta)$ denote the component of $\ColOneX$ containing $\Col_{\bfC_1}(\bfD)$. By Proposition \ref{P:Main:Rk1} \eqref{I:P:Main:Rk1:Rel=1:IETCase}, the $(\cM^{WP})_{\bfC_1^{WP}}$-littlest map restricted to $(Y, \eta)$ satisfies Assumption CP. 

Let $(Y', \eta')$ be any component of $\ColOneX$ distinct from $(Y, \eta)$. Let $C'$ be any cylinder on $(Y', \eta')$. By Mirzakhani-Wright \cite[Lemma 2.15]{MirWri}, there is a cylinder $C$ on $(X, \omega)$ in the complement of $\bfC_1$ so that $\ColOne(C) = C'$. Since $C$ belongs to the complement of $\bfD$ (and since $\cM$ is contained in a codimension at most one hyperelliptic locus), it follows that $C$ and hence also $C'$ is fixed by the hyperelliptic involution (see for example Lemma \ref{L:CutReglue}). Therefore, we have shown that every cylinder on $(Y', \eta')$ is fixed by the hyperelliptic involution. 

By Proposition \ref{P:Main:Rk1} \eqref{I:P:Main:Rk1:Rel=1:NoIETCase}, the $\cM_{\bfC_1}$-littlest map restricted to $(Y', \eta')$ satisfies Assumption CP. Since each cylinder on $(Y', \eta')$ is fixed by the hyperelliptic involution, the $(\cM^{WP})_{\bfC_1^{WP}}$-littlest map restricted to $(Y, \eta)$ also satisfies Assumption CP. Therefore, $\pi_i^{WP}$ also satisfies Assumption CP as desired.
\end{proof}



\begin{sublem}\label{SL:Subtle}
$\Col(\pi_1^{WP}) = \Col(\pi_2^{WP})$.
\end{sublem}  
By Apisa-Wright \cite[Theorem 8.2]{ApisaWrightDiamonds}, the codomains of $\Col(\pi_1^{WP})$ and $\Col(\pi_2^{WP})$ are isomorphic and when the two maps are restricted to any component of $\ColOneTwoX$ they have the same fibers. This makes the proof of Sublemma \ref{SL:Subtle} subtle. 
\begin{proof}
It may be useful to refer to Figure \ref{F:LabellingScheme} for this proof.

Label the zeros and marked points on $(X, \omega)^{WP}$ with $2'$ (resp. $1'$) if, generically, the only vertical (resp. horizontal) cylinders to whose boundary they belong are those in $(\bfC_2')^{WP}$ (resp. $(\bfC_1')^{WP}$); label any other points $2$ (resp. $1$). Note that all points will have two labels.  Since the cylinders in $(\bfC_2')^{WP}$ (resp. $(\bfC_1')^{WP}$) persist on $\Col_{\bfC_i^{WP}}(X, \omega)^{WP}$ and $\Col_{\bfC_1^{WP}, \bfC_2^{WP}}(X, \omega)^{WP}$, there is a labelling of zeros and marked points on each of these surfaces so that the maps $\Col_{\bfC_i^{WP}}: (X, \omega)^{WP} \ra \Col_{\bfC_i^{WP}}(X, \omega)^{WP}$ are label-preserving.

Under $\pi_i^{WP}$, $\Col_{\bfC_i^{WP}}((\bfC_2')^{WP})$ is sent to two vertical cylinders that are glued together along an IET. Moreover, under all small perturbations of the codomain that remain in $\cQ(0, -1^4)$, there are precisely two points that remain on the boundary of these vertical cylinders and, generically, no other. Label these two points $2'$ and the other marked points by $2$. By the same construction using $(\bfC_1')^{WP}$ label points $1'$ and $1$.

\begin{figure}[h]\centering
\includegraphics[width=\linewidth]{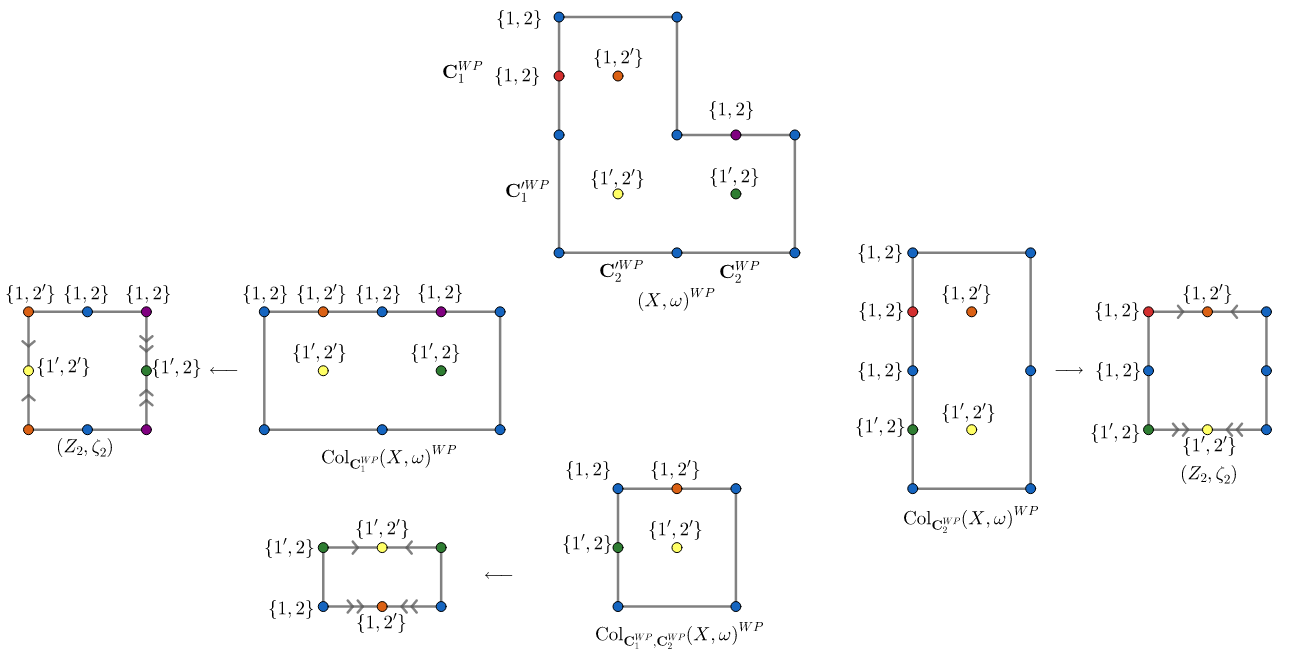}
\caption{An example of the labelling convention where $\cM = \cH(2)$. Unless otherwise indicated, opposite sides are identified on all polygons. The labels $\bfC_i^{WP}$ and $(\bfC_i^{WP})'$ indicate horizontal and vertical equivalence classes of cylinders. The horizontal arrows indicate the maps $\pi_1^{WP}$ and $\pi_2^{WP}$ for the middle surfaces and $\Col(\pi_1^{WP})$ on the bottom.}
\label{F:LabellingScheme}
\end{figure}

By construction, $\pi_i^{WP}$ is label-preserving. Moreover, the labelling of marked points on $(Z_i, \zeta_i)$ induces a labelling of marked points on $\Col_{\pi_i^{WP}(\Col_{\bfC_i^{WP}}((\bfC_{i+1})^{WP}))}((Z_i, \zeta_i))$ so that $\Col(\pi_i)$ is label-preserving. 

Note that all the points on $\Col_{\bfC_1^{WP}, \bfC_2^{WP}}(X, \omega)^{WP}$ are labelled as one of the following $\{\{1,2\}, \{1,2'\}, \{1',2\}, \{1',2'\}\}$. Similarly, each of these labels labels one of the four points on $\Col (Z_i, \zeta_i)^{WP}$. 

By Apisa-Wright \cite[Theorem 8.2]{ApisaWrightDiamonds}, $\Col (Z_1, \zeta_1)^{WP}$ and $\Col (Z_2, \zeta_2)^{WP}$ are isometric half-translation surfaces. Since both surfaces have points labelled by $1$ (resp. $2$) lying on horizontal (resp. vertical) lines, there is an label-preserving isometry $\iota$ from $\Col (Z_1, \zeta_1)$ to $\Col (Z_2, \zeta_2)$. Therefore, $\iota \circ \Col(\pi_1^{WP})$ and $\Col(\pi_2^{WP})$ are both label-preserving holomorphic maps, meaning that they are equal, as desired. 
\end{proof}

Therefore, Sublemmas \ref{SL:ACP1} and \ref{SL:Subtle} imply that $\cM^{WP}$ is a full locus of covers of $\cH(2)$ or a rank two rel zero stratum of quadratic differentials (by Apisa-Wright \cite[Theorem 8.2]{ApisaWrightDiamonds}). In the latter case, since $\cM$ is contained in a codimension zero or one hyperelliptic locus, $\cM_{min}$ is a quadratic double of genus zero stratum and hence a double of $\cQ(1, -1^5)$ (the codimension zero or one condition implies that the preimage of at most one pole is marked).  
\end{proof}

\begin{cor}\label{C:MinimalCoverCP}
If $\cM$ has rank two rel zero then either the $\cM$ or $\cM^{WP}$ minimal translation cover satisfies Assumption CP.
\end{cor}
\begin{proof}
If the codomain of the $\cM$-minimal translation cover belongs to $\cH(2)$, then the cover must satisfy Assumption CP. Suppose therefore that $(X, \omega)$ is a surface in $\cM$ and that the codomain of its $\cM$-minimal translation cover $\pi$ is a surface in $\cH(2)$ with a Weierstrass point $p$ marked. Let $\pi^{WP}$ denote this map with domain $(X, \omega)^{WP}$. If $\bfC$ is a generic equivalence class, then $\pi(\bfC)$ is either a simple cylinder or two simple cylinders that share a boundary consisting of a loop joining $p$ to itself. We will show that every cylinder in $\bfC^{WP}$ has the same height and that Assumption CP is satisfied. 

Suppose first that $\pi(\bfC)$ is a simple cylinder. It is clear that each cylinder in $\bfC$ has the same height. Moreover, if all of the cylinders in $\bfC$ are fixed by the hyperelliptic involution, then $\pi^{WP}$ takes every cylinder in $\bfC^{WP}$ to a cylinder of the same height. Suppose in order to derive a contradiction that $\bfC$ contains two cylinders that are exchanged by the hyperelliptic involution. Since $\cM$ belongs to a codimension zero or one hyperelliptic locus, the boundaries of this exchanged pair of cylinders contain every zero and marked point on surfaces in $(X, \omega)$, all of which map to a single point since $\pi(\bfC)$ is a simple cylinder. However, this contradicts the fact that the codomain of $\pi$ has one zero and one marked point. 

Suppose now that $\pi(\bfC)$ consists of two simple cylinders glued together along one boundary. In this case, $\bfC$ necessarily contains two cylinders that are glued together along an IET. Since $\cM$ belongs to a codimension zero or one hyperelliptic locus, every other cylinder in $\bfC$ is fixed by the hyperelliptic involution and borders no other cylinders in $\bfC$ (since, by Lemma \ref{L:Rigid:NotAdjacent}, this would imply that these two cylinders were glued together along an IET and hence not fixed by the hyperelliptic involution). Therefore, every cylinder in $\bfC^{WP}$ has the same height and $\pi^{WP}$ takes every cylinder in $\bfC^{WP}$ to a cylinder of the same height.
\end{proof}

\begin{lem}\label{L:PiMinDegenBaseCase}
If $\cM$ is either rank two rel zero or rank one rel two, $\bfC$ is a generic subequivalence class of cylinders on $(X, \omega) \in \cM$, and $\pi$ is the $\cM$-minimal translation cover, then, when restricted to the $i$th component of $\Col_{\bfC}(X, \omega)$, $\Col_{\bfC}(\pi)$ is the $(\cM_{\bfC})_i$-littlest map.
\end{lem}
\begin{proof}
By Lemma \ref{L:SidesAreCovers}, $\cM_{\bfC}$ is a full locus of covers of a genus zero stratum $\cQ$. Since $\bfC$ is generic and $\cM$ is either rank one rel two or rank two rel zero, $\cQ$ is rank one rel one.

By Corollary \ref{C:MinimalCoverCP} (in the case that $\cM$ has rank two) and Proposition \ref{P:Main:Rk1} \eqref{I:P:Main:Rk1:Rel=2} (in the case that $\cM$ has rank one), after perhaps replacing $(X, \omega)$ with $(X, \omega)^{WP}$ and $\cM$ with $\cM^{WP}$, the $\cM$-minimal map satisfies Assumption CP and its codomain belongs to a quadratic double of either $\cQ(0^2, -1^4)$ or $\cQ(1, -1^5)$. Moreover, the codomain of $\Col_{\bfC}(\pi)$ is a surface in a quadratic double of $\cQ(0, -1^4)$ and the map satisfies Assumption CP (by Apisa-Wright \cite[Proof of Lemma 8.6]{ApisaWrightDiamonds}). 

Let $(Y, \eta)$ be the $i$th component of $\Col_{\bfC}(X, \omega)$ and let $p: (Y, \eta) \ra (Z, \zeta)$ be the minimal degree map to a torus. By Proposition \ref{P:Main:Rk1}, $(Z, \zeta)$ belongs to either a quadratic double of $\cQ(0, -1^4)$ or an Abelian double of $\cH(0,0)$. Let $q: (Z, \zeta) \ra (Z', \zeta')$ be the map between tori so that $\Col_{\bfC}(\pi) \restriction_{(Y, \eta)} = q \circ p$. Notice that $q$ satisfies Assumption CP. 

Consider any direction containing a cylinder on $(Z, \zeta)$. Suppose without loss of generality that it is the horizontal direction and that, after perturbing slightly, the direction is covered by two subequivalence classes of cylinders, which we will label $1$ and $2$.

Suppose first that $(Z, \zeta)$ belongs to a quadratic double of $\cQ(0, -1^4)$. The subequivalence classes in the horizontal direction on $\Col_{\bfC}(X, \omega)$ appear vertically in one of the following cyclic orders $(1,2)$, $(1,2,1)$, $(2,1,2)$, or $(1,2,2,1)$. Since horizontal cylinders are sent to horizontal cylinders of the same height under maps satisfying Assumption CP, the preimage under $q$ of each horizontal cylinder on $(Z', \zeta')$ is a single cylinder on $(Z, \zeta)$. 

Since this holds for every periodic direction, $q$ is a ``good" map (see Apisa-Wright \cite[Definition 4.10]{ApisaWrightGeminal}) and so $q$ must be the identity by Apisa-Wright \cite[Lemma 4.16]{ApisaWrightGeminal}. Therefore, $\Col_{\bfC}(\pi) \restriction_{(Y, \eta)}$ is the $(\cM_{\bfC})_i$-littlest map.


Suppose now that $(Z, \zeta)$ belongs to an Abelian double of $\cH(0,0)$. After deforming $(Z, \zeta)$ in the Abelian double, we may suppose that it is covered by one free vertical cylinder and four horizontal simple cylinders appearing in cyclic order $(1,2,1,2)$. We see that $q$ is either the identity or the quotient by the translation involution. Since the codomain of $\Col_{\bfC}(\pi)$ is an element of the quadratic double of $\cQ(0, -1^4)$, it follows that $q$ is the quotient by the translation involution and hence that $\Col_{\bfC}(\pi) \restriction_{(Y, \eta)}$ is the $(\cM_{\bfC})_i$-littlest map. 
\end{proof}

Notice that Lemmas \ref{L:I:FullLocus:BaseCase} and \ref{L:PiMinDegenBaseCase} establish Theorem \ref{T:Main:Rigid} when $\cM$ is four-dimensional. Therefore, suppose now that $\cM$ has dimension at least five.


\begin{lem}\label{L:AgreementAtBase}
Suppose that $((X, \omega), \cM, \bfC_1, \bfC_2)$ is a generic diamond and that $\cM$ has dimension at least five. Let $\pi_i: \Col_{\bfC_i}(X_i, \omega_i) \ra (Y_i, \eta_i)$ be the half-translation cover of a connected surface in a genus zero stratum produced in Lemma \ref{L:SidesAreCovers}. Then $\Col(\pi_2) = \Col(\pi_1)$.
\end{lem}

For a definition of generic diamond see Apisa-Wright \cite[Definition 3.26]{ApisaWrightDiamonds}.

\begin{proof}
By Lemma \ref{L:SidesAreCovers}, $\cM_{\bfC_i}$ is a full locus of covers of a genus zero stratum $\cQ_i$. The covering map on $\Col_{\bfC_i}(X, \omega)$ is $\pi_i$. It follows that $\MOneTwo$ is a full locus of covers of $(\cQ_i)_{\pi_i(\bfC_{i+1})}$ with $\Col(\pi_i)$ the covering map on $\ColOneTwoX$. By the induction hypothesis (see Remark \ref{R:PreimageOfImage}), $\pi_i(\bfC_{i+1})$ is a simple cylinder or simple envelope and so the surfaces in $(\cQ_i)_{\pi_i(\bfC_{i+1})}$ are connected. By the induction hypothesis, specifically by \eqref{I:PiMinDegen}, the restriction of $\Col(\pi_i)$ to the $i$th component of $\ColOneTwoX$ is the composition of the $(\MOneTwo)_i$-littlest map with the quotient by the holonomy involution (note that $\cM_{\bfC_i}$ is at least four-dimensional since $\cM$ has dimension at least five and since we chose the diamond to be generic). This shows that the fibers of $\Col(\pi_1)$ and $\Col(\pi_2)$ are identical on each component of $\ColOneTwoX$.

If $\Col(\pi_1) \ne \Col(\pi_2)$, then there is a marked-point preserving affine symmetry of the codomain. However, the only strata in which the generic element has such affine symmetries are hyperelliptic components by Apisa-Wright \cite[Proof of Lemma 9.5]{ApisaWrightDiamonds}. The only hyperelliptic component of a genus zero stratum is $\cQ(-1^4)$, which is two-dimensional, whereas $\MOneTwo$ has dimension at least three (since $\cM$ has dimension at least five and since we chose the diamond to be generic). Therefore, $\Col(\pi_1) = \Col(\pi_2)$.  
\end{proof}

Suppose that $\cM$ has dimension at least five. Let $\bfC_1$ be any generic subequivalence class on any surface $(X, \omega) \in \cM$. By Apisa-Wright \cite[Lemma 3.31]{ApisaWrightDiamonds}, after perturbing, there is another subequivalence class $\bfC_2$ so that $((X, \omega), \cM, \bfC_1, \bfC_2)$ is a generic diamond. Notice that $\cM_{\bfC_i}$ is a full locus of covers of a genus zero stratum (by Lemma \ref{L:SidesAreCovers}) and let $\pi_i$ denote the corresponding cover with domain $\Col_{\bfC_i}(X, \omega)$. The hypotheses of the diamond lemma, Apisa-Wright \cite[Lemma 2.3]{ApisaWrightDiamonds}, hold by Lemma \ref{L:AgreementAtBase} and the induction hypothesis (see Remark \ref{R:PreimageOfImage}), so there is a half-translation cover $\pi: (X, \omega) \ra (Y, \eta)$ so that $\Col_{\bfC_i}(\pi) = \pi_i$. Letting $\cN$ be the orbit closure of $\pi(X, \omega)$, we have a generic diamond $(\pi(X, \omega), \cN, \pi(\bfC_1), \pi(\bfC_2))$ where $\cN_{\pi(\bfC_i)}$ is a stratum of genus zero quadratic differentials. By Apisa-Wright \cite[Theorem 7.5]{ApisaWrightDiamonds}, $\cN$ is a component of a stratum of quadratic differentials. Notice that $\cN$ cannot be a hyperelliptic component, since then $\cN_{\pi(\bfC_i)}$ would be as well, but the only hyperelliptic genus zero stratum is $\cQ(-1^4)$ and $\cN_{\pi(\bfC_i)} \ne \cQ(-1^4)$ since it has dimension at least four. Since the surfaces in $\cM$ are hyperelliptic they admit maps to genus zero surfaces and hence the surfaces in $\cN$ do as well, which implies that $\cN$ is a genus zero stratum and hence that $\pi$ is the $\cM$-minimal half-translation cover. This shows \eqref{I:FullLocus}. Since this holds for arbitrary $\bfC_1$ we have \eqref{I:PiMinDegen}. 
\end{proof}

\section{The unfoldings of rational right and isosceles triangles - Proof of Theorem \ref{T:Main:Unfolding}}\label{S:ProofTUnfolding}

We begin by fixing a rational right triangle with angles $(\frac{a}{e}, \frac{b}{e}, \frac{1}{2})\pi$ where $a, b$ and $e$ are positive integers with $\mathrm{gcd}(a,b,e) = 1$.

\begin{lem}\label{L:BasicNumberTheory1}
There is an integer $n$ so that $e = 2n$ and, up to exchanging $a$ with $b$, $\mathrm{gcd}(a, 2n) = 1$.
\end{lem}
\begin{proof}
Since the sum of the angles is $\pi$ we have $n:= a + b = \frac{e}{2}$, which establishes the first claim and shows that $b = n - a$. Since $\mathrm{gcd}(a,b,2n) = 1$, it follows that $a$ and $b$ cannot both be even, therefore suppose, after perhaps exchanging $a$ and $b$, that $a$ is odd. The fact that $a$ is odd and $b = n-a$ shows that $1 = \mathrm{gcd}(a,b,2n) = \mathrm{gcd}(a,2n)$.
\end{proof}

In light of the lemma, we will think of the rational right triangle, which we denote $Q_{a,n}$ as being specified by two positive integers $a$ and $n$ so that $\mathrm{gcd}(a,2n)=1$.

The pillowcase double $P_{a,n}$ of $Q_{a,n}$ is a $(2n)$-differential on the sphere with three cone points of cone angle $(\frac{a}{n}, \frac{n-a}{n}, 1)\pi$. We may suppose that these singularities lie at $0, \infty$, and $1$ respectively. Pulling back the pillowcase double under the map $z \ra z^n$ produces a quadratic differential on the sphere with cone points of cone angle $(a, n-a, 1^n)\pi$ where the superscript denotes multiplicity. Let $(S, q)$ denote this quadratic differential (the dependence on $n$ and $a$ will now be tacit) and note that it has at most two zeros. Let $(X, \omega)$ denote its holonomy double cover with no preimages of poles marked. Since $(S, q)$ had at most two zeros, it follows that $(X, \omega)$ belongs to a hyperelliptic locus of codimension at most one (by Lemma \ref{L:Codim=Zeros}).  

Notice that $(X, \omega)$ is a cyclic cover of $P_{a,n}$. Following Mirzakhani-Wright \cite[Section 6]{MirWri2} we will choose a generator $T$ of the deck group, which is isomorphic to $\mathbb{Z}/2n\mathbb{Z}$, so that $T^* \omega = \zeta \omega$ where $\zeta = \mathrm{exp}\left( \frac{2\pi i}{2n} \right)$. Let $H^1_\ell$ (resp. $H^{1,0}_\ell$) denote the eigenspace of eigenvalue $\zeta^\ell$ in $H^1(X; \mathbb{C})$ (resp. $H^{1,0}(X; \mathbb{C})$). 

\begin{lem}\label{L:EigenspaceDimension}
$\dim H^{1,0}_\ell = 1$ if and only if $\ell a$ is an odd number congruent to an element of $\{1, \hdots, n-1\}$ mod $2n$. The dimension is zero otherwise. 
\end{lem}
\begin{proof}
Define
\[ t(\ell) := \mathrm{frac}\left(\frac{\ell}{2} \right) + \mathrm{frac}\left(\frac{\ell a}{2n} \right) + \mathrm{frac}\left(\frac{\ell(n-a)}{2n} \right)\]
where $\mathrm{frac}(x) := x - \lfloor x \rfloor$ denotes the positive fractional part of a real number. By Wright \cite[Lemma 2.6]{W1} (see also Mirzakhani-Wright \cite[Lemma 6.1]{MirWri2}), $\dim H^{1,0}_\ell = t(-\ell) -1$. Since $\mathrm{frac}\left(\frac{\ell}{2} \right) = 0$ when $\ell$ is even, we see that $\dim H^{1,0}_\ell$ is only positive dimensional when $\ell$ is odd, in which case,
\[ \dim H^{1,0}_\ell = \mathrm{frac}\left(\frac{-\ell a}{2n} \right) + \mathrm{frac}\left(\frac{-\ell(n-a)}{2n} \right) - \frac{1}{2}. \]
Since $\ell$ is odd, $-\ell n \equiv n$ mod $2n$, so 
\[ \dim H^{1,0}_\ell = \mathrm{frac}\left(\frac{-\ell a}{2n} \right) + \mathrm{frac}\left(\frac{n+\ell a}{2n} \right) - \frac{1}{2}. \]
If $-\ell a$ is an element of $\{0, \hdots, n\}$ mod $2n$, then the first term is less than or equal to $\frac{1}{2}$ and hence the right hand side is necessarily less than $1$ (and hence equal to zero). Conversely, if $-\ell a$ is an element of $\{n+1, \hdots, 2n-1\}$ mod $2n$ then the right hand side is greater than zero (and hence equal to $1$).
\end{proof}

Let $\omega_\ell$ denote the unique (up to scaling) nonzero element of $H^{1,0}_\ell$ when this space has positive dimension. Let $\cM$ denote the orbit closure of $(X, \omega)$ and let $H^1_\cM$ denote the bundle over $\cM$ whose fiber over a translation surface $(X', \omega')$ is $H^1(X; \mathbb{C})$. By Wright \cite{Wfield}, this bundle decomposes into a sum of symplectically orthogonal flat bundles
\[ H^1_\cM = \bigoplus_{\rho \in \mathrm{Gal}(\mathbf{k}(\cM)/\mathbb{Q})} p(T\cM)^\rho \oplus \bigoplus_s \mathbb{W}_s \]
where $p$ is the projection from relative to absolute cohomology, $k(\cM)$ is the field of definition of $\cM$, the summands in the first factor are Galois conjugates of $p(T\cM)$, and the summands in the second factor are the remaining isotypic components for the monodromy representation. 

\begin{lem}\label{L:DirectSum}
Each summand in the decomposition above is a direct sum of $T$-eigenspaces.
\end{lem}
\begin{proof}
By Mirzakhani-Wright \cite[Lemma 7.7]{MirWri2}, there is a loop $\gamma$ in $\cM$ based at $(X, \omega)$ whose monodromy is $T$.  Since the decomposition of $H^1_{\cM}$ is monodromy-invariant, each summand is a $T$-representation. It follows that each summand is a direct sum of $T$-eigenspaces.  
%
\end{proof}

By Filip \cite{Fi2}, each summand respects the Hodge decomposition in that its intersection with $H^{1,0}$ and $H^{0,1}$ are each half-dimensional. As explained in Mirzakhani-Wright \cite[Lemma 7.2]{MirWri2}, by Forni-Matheus-Zorich \cite[Section 2.3]{FMZ}, this implies that if $\omega_1$ and $\omega_2$ are holomorphic one-forms on $X$ in different summands of the above decomposition then 
\[ B_\omega(\omega_1, \omega_2) := \frac{i}{2} \int_X \omega_1 \omega_2 \frac{\overline{\omega}}{\omega} = 0.\]

\begin{lem}\label{L:UnfoldingRank}
$\cM$ has rank one if and only if $a \in \{1, n-1, n-2\}$. 
\end{lem}
\begin{proof}
%
Suppose that $\cM$ has rank one. Since $p(T\cM)$ contains $H^{1,0}_1$, it must coincide with $H^1_{\pm 1} = H^1_1 \oplus H^1_{-1}$. By Mirzakhani-Wright \cite[Lemma 7.6]{MirWri2}, 
\[ \bigoplus_{\rho \in \mathrm{Gal}(\mathbf{k}(\cM)/\mathbb{Q})} p(T\cM)^\rho = \bigoplus_{k \in (\mathbb{Z}/2n\mathbb{Z})^\times} H^1_{\pm k}.\]
Therefore, any element in $H^{1,0}_k$ for $k \in (\mathbb{Z}/2n\mathbb{Z})^\times$ is $B_\omega(\cdot, \cdot)$-orthogonal to any element of $H^{1,0}_\ell$ for $k \ne \ell$.  By Lemma \ref{L:EigenspaceDimension} and Mirzakhani-Wright \cite[Proposition 7.3]{MirWri2}, we have the following:

\begin{sublem}
If $a\alpha$ and $a\beta$ are odd numbers congruent to a number in $\{1, \hdots, n-1\}$ mod $2n$, then $B_\omega(\omega_\alpha, \omega_\beta) \ne 0$ if and only if $a\alpha + a\beta = 2a$ mod $2n$.
\end{sublem}

If $a \notin \{1, n-1, n-2\}$ we can find $0 \ne \omega_\alpha$ where $\alpha \in (\mathbb{Z}/2n\mathbb{Z})^\times$ and $\beta \ne \alpha$ so that $\omega_\beta \ne 0$ and $B_\omega(\omega_\alpha, \omega_\beta) \ne 0$ (for instance, set $a\alpha = 1$ and $a\beta = 2a-1$ when $4 \leq 2a \leq n$ and set $a\alpha = n-\epsilon$ and $a\beta = 2a+\epsilon-n$ when $n \leq 2a \leq 2n-4$ where $\epsilon$ is $1$ if $n$ is even and $2$ if $n$ is odd. It is relevant that $a$ is a unit in $(\mathbb{Z}/2n\mathbb{Z})^\times$ by Lemma \ref{L:BasicNumberTheory1}). This is a contradiction.
\end{proof}

\begin{thm}\label{T:Unfolding:RightDetailed}
If the smallest angle in $Q_{a,n}$ has the form $\frac{\pi}{m}$ where $m$ is a positive integer, then $\cM$ is the regular $m$-gon (resp. double regular $m$-gon) locus if $m$ is even (resp. odd). Otherwise, $\cM$ is a quadratic double of $\cQ(a-2, b-2, -1^n)$. 
\end{thm}
\begin{proof}
The first claim is due to Veech \cite{V}. So suppose that $a \notin \{1, n-1, n-2\}$. By Lemma \ref{L:UnfoldingRank}, $\cM$ has rank at least two. Since we have already observed that $\cM$ belongs to a hyperelliptic locus of codimension at most one, $\cM$ is a full locus of covers of a genus zero stratum (by Theorem \ref{T:Main}). 

Let $\pi: (X, \omega) \ra (Y, \eta)$ be the corresponding translation cover with domain $(X, \omega)$ and let $\cL$ denote the hyperelliptic locus containing $(Y, \eta)$. It follows that $T_{(X, \omega)} \cM = \pi^* T_{(Y, \eta)} \cL$. Let $f:(S, q) \ra (Y, \eta)/J$ be the map induced by $\pi$ where $J$ is the hyperelliptic involution on $(Y, \eta)$. 

By Mirzakhani-Wright \cite[Lemma 7.6]{MirWri2}, since $\mathbf{k}(\cM) = \mathbb{Q}$, $p(T_{(X, \omega)} \cM)$ contains 
\[ \bigoplus_{k \in (\mathbb{Z}/2n\mathbb{Z})^\times} H^1_k.\]

By Wright \cite[Lemma 2.6]{W1} (see also Mirzakhani-Wright \cite[Lemma 6.1]{MirWri2}), up to scaling, $(\omega_\ell)^{2n}$ is the pullback of the $(2n)$-differential on the thrice-punctured sphere
\[ \frac{dz^{2n}}{z^{[\ell a]} (z-1)^n} \]
to $X$ where $[\ell a]$ denotes the value in $\{1, \hdots, n-1\}$ that is congruent to $\ell a$. This $(2n)$-differential on the thrice punctured sphere has three cone points at $0$, $1$, and $\infty$ of cone angle $\frac{[\ell a]}{n}  \pi$, $\pi$, and $\frac{n-[\ell a]}{n}  \pi$. Therefore, the pullback to $S$, which we will denote $q_\ell$, has cone points of cone angle $([\ell a], n - [\ell a], 1^n)\pi$. Moreover, the holonomy double cover of $(S, q_\ell)$ belongs to $T_{(X, \omega)} \cM$ provided that $\ell$ is coprime to $2n$ and that $\ell a$ is congruent to some odd number in $\{1, \hdots, n-1\}$ mod $2n$. 


Since $a$ is coprime to $2n$, let $\ell_0$ be its inverse in $(\mathbb{Z}/(2n)\mathbb{Z})^\times$. Note that $(S, q_{\ell_0})$ has a pole at $0$. Since 
\[ \omega_{\ell_0} \in T_{(X, \omega)} \cM = \pi^* T_{(Y, \eta)} \cL, \]
it follows that $f$ does not have a ramification point at $0$. 

Similarly, if $n$ is even (resp. odd) there is $\ell_1 \in (\mathbb{Z}/(2n)\mathbb{Z})^\times$ so that $\ell_1 a = n-1$ (resp. $\ell_1 a = n-2$). Since $(S, q_{\ell_1})$ has a pole (resp. marked point) at $\infty$, it follows, as above, that either $f$ does not have a ramification point at $\infty$ (when $n$ is even) or $f$ is at most doubly ramified at $\infty$ (when $n$ is odd).


Suppose first that $\infty$ is not a ramification point. Since $\{0, \infty\}$ are the only preimages of the zeros of $(Y, \eta)/J$ and since neither is a ramification point, we have that either $f$ is the identity map or $f$ is a double cover and $a = n-a$. However, this latter case violates the condition that $\mathrm{gcd}(a, n) = 1$ and the assumption that $a > 1$.

Suppose now that $\infty$ is a double ramification point. Since $\{0, \infty\}$ are the only preimages of the zeros of $(Y, \eta)/J$ it follows that there is a single zero on $(Y, \eta)/J$ and that $f$ is a triple cover. In particular, $2a = (n-a)$. As before this violates the condition that $\mathrm{gcd}(a, n) = 1$ and the assumption that $a > 1$. 

Therefore, as before, $f$ is the identity as desired. 
%
\end{proof}

\subsection{Unfoldings of rational isosceles triangles}

In this subsection we will complete the proof of Theorem \ref{T:Main:Unfolding} by showing the following. 

\begin{thm}\label{T:Unfolding:IsoscelesDetailed}
Let $T$ be a rational isosceles triangle with angles $\left( \frac{a}{n}, \frac{a}{n}, \frac{b}{n} \right)\pi$ where $a, b$ and $n$ are positive integers with $\mathrm{gcd}(a, b, n) =1$. Then the orbit closure $\cN$ of the unfolding of $T$ is one of the following, when $n$ is odd:
\begin{enumerate}
    \item If $a=1$ (resp. $b=1$) then $\cN$ is the double regular $n$-gon (resp. regular $2n$-gon) locus.
    \item If $\min(a,b) \ne 1$, then $\cN$ is the quadratic double of $\cQ(2a-2, b-2, -1^n)$.
\end{enumerate}
When $n$ is even, the unfolding belongs to the quadratic double\footnote{A similar observation about when the unfoldings of isosceles triangles are double covers was made in Hubert-Schmidt \cite{HubertSchmidt-polygonal}.} of the (unique) hyperelliptic component $\cQ$ of a stratum of quadratic differentials that is a locus of double covers of $\cQ(a-2, \frac{b}{2}-2, -1^{\frac{n}{2}})$. Moreover,
\begin{enumerate}
    \item If $\min(a,\frac{b}{2}) = 1$ (resp. $b=4$) then $\cN$ is the double regular $n$-gon (resp. a double cover of the double regular $\frac{n}{2}$-gon) locus.
    \item If $\min(a,\frac{b}{2}) \ne 1$ and $b \ne 4$, then $\cN$ is the quadratic double of $\cQ$.
\end{enumerate}
\end{thm}

We will reduce the proof of Theorem \ref{T:Unfolding:IsoscelesDetailed} to Theorem \ref{T:Unfolding:RightDetailed}. Begin by fixing a rational isosceles triangle $T$ with angles $(\frac{a}{e}, \frac{a}{e}, \frac{b}{e})\pi$ where $a, b$ and $e$ are positive integers with $\mathrm{gcd}(a,b,e) = 1$. By taking the perpendicular bisector through the angle of size $\frac{b \pi}{e}$ we see that the triangle is tiled by two isometric right triangles $R$ with angles $(\frac{a}{e}, \frac{b}{2e}, \frac{1}{2})\pi$. 

\begin{prop}\label{P:IsoscelesCovers}
If $e$ is odd, then the unfolding of $T$ is the same as the unfolding of $R$. If $e$ is even, then the unfolding of $T$ is a translation double cover of the unfolding of $R$.

Moreover, when $e$ is even, the unfolding of $T$ belongs to a quadratic double of a hyperelliptic component $\cQ$ of a stratum of quadratic differentials. The double cover from the unfolding of $T$ to the unfolding of $R$ is induced from the double cover of a generic surface in $\cQ$ to a generic surface in $\cQ(a-2, \frac{b}{2}-1, -1^{\frac{n}{2}})$.
%
\end{prop}
\begin{proof}
Let $D(T)$ and $D(R)$ denote the pillowcase doubles of $T$ and $R$ respectively. Since $T$ is tiled by two copies of $R$, there is a double branched cover $\pi: D(T) \ra D(R)$ that is a local isometry of the flat metric away from cone points. Let $V(R)$ be the three cone points on $D(R)$ and let $V(T)$ be the four points in its preimage on $D(T)$. Since $\pi: D(T) - V(T) \ra D(R) - V(R)$ is a double cover, $\pi_*: \pi_1(D(T) - V(T)) \ra \pi_1(D(R) - V(R))$ is an injection and its image is a normal index two subgroup $N$ of $\pi_1(D(R) - V(R))$. 

Let $\rho_R: D(R) - V(R) \ra \mathbb{Z}/d \mathbb{Z}$ be the holonomy map. Note that $d = \frac{2e}{\mathrm{gcd}(e, 2)}$. The holonomy map on $D(T) - V(T)$ is given by $\rho_T := \rho_R \circ \pi_*$. Let $K$ be the kernel of $\rho_R$, which is a normal index $d$ subgroup. Since the kernel of $\rho_T$ is an index $e$ subgroup of $N$ we have that $K \cap N$ is an index $2e$ subgroup of $\pi_1(D(R) - V(R))$. Since $d = 2e$ when $e$ is odd, $K = K \cap N$ in this case. This shows that $\pi_*(\ker(\rho_T)) = \ker \rho_R$ when $e$ is odd, which yields the first claim.


Suppose now that $e$ is even. In this case, $K \cap N$ is a normal index two subgroup of $K$. This implies that the unfolding of $D(T)$ is a translation double cover of the unfolding of $D(R)$. We note that since $2a + b = e$ and $\mathrm{gcd}(a,b,e)=1$, we have that $a$ is odd and $b$ is even. 

Let $K'$ denote the preimage under $\rho_R$ of the two-element subgroup generated by $\frac{e}{2}$. This subgroup corresponds to the cover of $D(R)$ given by its partial unfolding to a surface in the genus zero stratum $\cQ(a-2, \frac{b}{2}-2, -1^{e/2})$. Therefore, $K' \cap N$ corresponds to a double cover of this surface that is branched over the poles (since the pullback of the flat structure on $D(T)$ to the cover corresponding to $K' \cap N$ has no poles) and unbranched over the zero of order $a-2$ (since the cone point of cone angle $\frac{a\pi}{e}$ has two preimages on the map from $D(T)$ to $D(R)$). Since a double branched cover of the sphere is necessarily branched over an even number of points, the zero of order $\frac{b}{2}-2$ is a branch point if and only if $\frac{e}{2}$ is odd (equivalently, if $\frac{b}{2}-2$ is even). Therefore, the  branched cover corresponding to the inclusion of $K' \cap N$ into $K'$ is precisely the one that sends a generic surface in a hyperelliptic component of a stratum of quadratic differentials to a genus zero surface (see Apisa-Wright \cite[Section 8]{ApisaWrightDiamonds}).
\end{proof}

Theorem \ref{T:Unfolding:IsoscelesDetailed} now follows immediately from Theorem \ref{T:Unfolding:RightDetailed} and Proposition \ref{P:IsoscelesCovers}.

\section{Counting in rational right and isosceles triangles - Proof of Theorems \ref{T:Counting:ClosedTrajectories} and \ref{T:Counting:GeneralizedDiagonals}}

In this section, we will summarize the work of Athreya-Eskin-Zorich \cite{AEZ} and use it to deduce Theorems \ref{T:Counting:ClosedTrajectories} and \ref{T:Counting:GeneralizedDiagonals}. ``Singularity" will refer to zeros, poles, and marked points, all of which will be assumed to be labelled. All saddle connections will be assumed to be unoriented. 

\begin{defn}\label{D:Configurations}
Given a half-translation surface in a genus zero stratum $\cQ(k_1, k_2, 0^m, -1^{k_1+k_2+4})$ where $k_1, k_2 > 0$, we will define three configurations of saddle connections whose numbering corresponds to that in \cite{AEZ}.
\begin{enumerate}
    \item (Configuration $1$): Let $\{p_1, p_2\}$ be a set of two distinct singularities not consisting of two poles. A saddle connection $s$ joining $p_1$ to $p_2$ is said to belong to configuration $1$. The length of the configuration is the length of $s$. Letting $d_i$ be the order of the singularity at $p_i$, define
    \[ c_{p_1, p_2} := \frac{(d_i+d_j+2)!!(d_i+1)!!(d_j+1)!!}{(d_i+d_j+1)!!d_i!!d_j!!}\cdot \begin{cases}
                                  \frac{2}{\pi^2} & \text{ $d_i$ and $d_j$ both odd} \\
                                  \frac{1}{2} & \text{ either $d_i$ or $d_j$ even} 
  \end{cases}\]
    \item (Configuration $3$): Suppose that $m=0$. Let $\{p_1, p_2\}$ be a set of two distinct poles and $r$ a zero. A saddle connection $s_1$ joining $p_1$ to $p_2$ together with a saddle connection $s_2$ joining $r$ to itself that bound a cylinder $C$ is said to belong to configuration $3$. The length of the configuration is the circumference of $C$. Letting $d$ be the order of $r$, set
    \[ c_{env, r} = \frac{d+1}{2\pi^2(k_1+k_2+2)} \]
    \item (Configuration $4$): Suppose that $m=0$. A simple separating cylinder $C$ together with the set of $k_1+2$ poles that belong to one component of the complement of $C$ is said to specify a configuration of type $4$. The length of the configuration is the circumference of $C$.  Set
    \[ c_{simp} = \frac{1}{2\pi^2\binom{k_1+k_2+2}{k_1+1}} \]
\end{enumerate}
\end{defn}

By Eskin-Mirzakhani-Mohammadi \cite[Theorem 2.12]{EMM} and Athreya-Eskin-Zorich \cite[Theorems 4.3, 4.5, and 4.8]{AEZ} we have the following:

\begin{thm}\label{T:AEZ}
Let $(Y, q)$ be any half-translation surface with dense orbit in $\cQ(k_1, k_2, 0^m, -1^{k_1+k_2+4})$ where $k_1, k_2 > 0$. Given a configuration with associated Siegel-Veech constant $c_*$ let $N_{*, (Y, q)}(L)$ be the number of configurations on $(Y, q)$ of length less than $L$. Then
\[ N_{*, (Y, q)}(L) ``\sim" c_* \frac{\pi L^2}{\mathrm{area}(Y, q)}. \]
\end{thm}

Given a flat surface $(Y, q)$ we will let $N_{cyl, (Y, q)}(L)$ (resp. $N_{simp, (Y, q)}(L)$, $N_{env, (Y, q)}(L)$) denote the number of cylinders (resp. simple cylinders, resp. envelopes) on $(Y, q)$.

\begin{cor}\label{C:T:AEZ}
If $(Y, q)$ is any half-translation surface with dense orbit in $\cQ(k_1, k_2, 0^m, -1^{k_1+k_2+4})$ where $k_1, k_2 > 0$, then 
\[ N_{*, (Y, q)}(L) ``\sim" c_* \frac{\pi L^2}{\mathrm{area}(Y, q)} \]
where $*$ is ``cyl", ``env", or ``simp" and 
\begin{equation*}
\begin{split}
    c_{cyl} &= \frac{(k_1+k_2+4)(k_1+k_2+3)}{4\pi^2}\left( 1 + \frac{2}{(k_1+2)(k_2+2)} \right) \\
    c_{env} &= \frac{1}{2\pi^2} \binom{k_1+k_2+4}{2} \\
    c_{simp} &= \frac{1}{2\pi^2}\binom{k_1+k_2+4}{2} \frac{2}{(k_1+2)(k_2+2)}.
\end{split}
\end{equation*}
\end{cor}
\begin{proof}
By Eskin-Mirzakhani-Mohammadi \cite[Theorem 2.12]{EMM}, once area is normalized, the weak asymptotics of $N_{*, (Y, q)}(L)$ are independent of which surface $(Y, q)$ with dense orbit is chosen. We will therefore assume that $(Y, q)$ belongs to the full measure subset of the stratum for which every cylinder is either a simple envelope or a simple cylinder (see Masur-Zorich \cite{MZ} or Apisa-Wright \cite[Section 4.1]{ApisaWrightDiamonds} for an explanation of why these two cylinder types are the only ones that generically occur in genus zero strata). We therefore have $c_{cyl} = c_{env} + c_{simp}$.

Counting simple cylinders is equivalent to counting configurations of type 4. Part of the data of the configuration is specifying $k_1+2$ poles. There are $\binom{k_1+k_2+4}{k_1+2}$ different sets of choices. By Theorem \ref{T:AEZ}, it follows that 
\[ c_{simp} =  \frac{1}{2\pi^2\binom{k_1+k_2+2}{k_1+1}}\binom{k_1+k_2+4}{k_1+2} = \frac{1}{2\pi^2}\binom{k_1+k_2+4}{2} \frac{2}{(k_1+2)(k_2+2)}.\]

Similarly, counting simple envelopes is equivalent to counting configurations of type 3. Part of the data of the configuration is specifying $2$ poles and the order of a zero. By Theorem \ref{T:AEZ}, it follows that 
\[ c_{env} =  \frac{1}{2\pi^2(k_1+k_2+2)}\binom{k_1+k_2+4}{2}\left( (k_1+1) + (k_2+1) \right) = \frac{1}{2\pi^2}\binom{k_1+k_2+4}{2}.\]
\end{proof}

%


The following is an immediate corollary.

\begin{thm}\label{T:AEZ:Corollary}
Let $(Y, q)$ be any half-translation surface with dense orbit in $\cQ(k_1, k_2, 0^m, -1^{k_1+k_2+4})$ where $k_1, k_2 > 0$. Let $(X_1, \omega_1)$ (resp. $(X_2, q_2)$) be the holonomy double cover (resp. corresponding surface in a hyperelliptic component of a stratum of quadratic differentials) that covers $(Y, q)$. Suppose too that preimages of poles are unmarked. Then 
\[ N_{cyl, (X_1, \omega_1)}(L) ``\sim" \frac{1}{2\pi^2} \binom{k_1+k_2+4}{2} \left( 1 + \frac{4}{(k_1+2)(k_2+2)}\right) \frac{2\pi L^2}{\mathrm{area}(X_1, \omega_1)} \]
and 
\[ N_{cyl, (X_2, q_2)}(L) ``\sim" \frac{1}{2\pi^2} \binom{k_1+k_2+4}{2} \left( 1 + \frac{1}{2(k_1+2)(k_2+2)}\right) \frac{2\pi L^2}{\mathrm{area}(X_2, \omega_2)}. \]
\end{thm}
The double cover of $(Y, q)$ that belongs to a hyperelliptic component of a stratum of quadratic differentials was originally described by Lanneau \cite{LanneauHyp}, see also Apisa-Wright \cite[Section 9]{ApisaWright}.
\begin{proof}
On $(X_1, \omega_1)$, there are two kinds of cylinders, those that come from envelopes and pairs that come from simple cylinders.  This shows that
\[ N_{cyl, (X_1, \omega_1)}(L) = N_{env, (Y,q)}(L) + 2N_{simp, (Y,q)}(L) \]
Hence the formula follows from Corollary \ref{C:T:AEZ}. 

Similarly, on $(X_2, q_2)$ there are two kinds of cylinders. First, there are cylinders that are preimages of an envelope $C'$; these are simple cylinders of the same length as $C'$ (this follows since the condition that $k_1, k_2 > 0$ implies that the map from $(X_2, q_2)$ to $(Y, q)$ is branched over every pole). Second, there are cylinders that are preimages of simple cylinders $C'$; these are complex cylinders of twice the length of $C'$ (this follows since every cylinder on $(X_2, q_2)$ is fixed by the hyperelliptic involution; see Apisa-Wright \cite[Section 4.1]{ApisaWright} for a definition of ``complex cylinder"). This show that
\[ N_{cyl, (X_2, q_2)}(L) = N_{env, (Y,q)}(L) + N_{simp, (Y,q)}\left( \frac{L}{2} \right) \]
and so the formula follows from Corollary \ref{C:T:AEZ}. 
\end{proof}

\begin{defn}\label{D:Double}
Recall that, given a polygon $P$, its \emph{pillowcase double} $D(P)$, is the flat structure on the sphere formed by taking $P$ and its reflection $P'$ across an edge and identifying corresponding edges. There is an antiholomorphic involution $J$ on $D(P)$ that exchanges $P$ and $P'$ by reflection. Letting $V(P)$ denote the cone points of $D(P)$, a \emph{cylinder} on $D(P)$ will refer to a maximal collection of homotopic closed geodesic loops on $D(P) - V(P)$.

If $P$ is a rational polygon, then $D(P)$ is a holomorphic $k$-differential, and there is a canonical cover, called the \emph{holonomy cover}, $\pi: (X, \omega) \ra D(P)$ under which the $k$-differential pulls back to the $k$th power of a holomorphic one-form $\omega$ on a Riemann surface $X$. If $\pi$ factors through a (half) translation cover $\tau: (X, \omega) \ra (Y, q)$ then $\tau$ has degree one or two and we will call $(Y, q)$ the \emph{partial unfolding of $P$}. Notice that if the degree of $\tau$ is $1$, then $(Y, q) = (X, \omega)$.
\end{defn}

\begin{lem}\label{L:UnfoldingsVsPolygons}
Let $P$ be a rational polygon. Using the notation of Definition \ref{D:Double}, there are only finitely many $J$-invariant saddle connections and cylinders on $D(P)$.

Moreover, if $d$ is the degree of the induced cover from $(Y, q)$ to $D(P)$, then $\displaystyle{\lim_{L \ra \infty} \frac{N_{P}(L)}{\frac{1}{2d} N_{cyl, (Y, q)}(L)}} = 1$. Similarly, if $p_1$ and $p_2$ are points on $P$ with preimages $P_1$ and $P_2$ on $(Y, q)$, then $\displaystyle{\lim_{L \ra \infty} \frac{N_{P; p_1, p_2}(L)}{\frac{1}{2d} N_{\mathcal{C}, (Y, q)}(L)}} = 1$, where $\cC$ is the configuration of saddle connections joining a point in $P_1$ to a point in $P_2$.
%
\end{lem}

Recall that $N_P(L)$ and $N_{P; p_1, p_2}(L)$ are defined in the introduction. 

\begin{proof}
Suppose that $s$ is either a saddle connection or the core curve of a cylinder. Suppose too that $s$ is fixed by $J$. Note that $s$ contains a point $p$ that belongs to the interior of an edge $e$ of $P$, which is embedded in $D(P)$. The point $p$ is necessarily a fixed point of $J$. Suppose without loss of generality that $e$ is horizontal and that one of the branches of $s$ through $p$ travels in the direction of $\pm e^{i \alpha}$. Let $v$ denote this unit tangent vector at $p$. Since $s$ is invariant by $J$ (and since $p$ is a fixed point), another branch must pass through $p$ traveling in the direction of $\pm e^{-i\alpha}$. Let $\overline{v}$ denote this unit tangent vector at $p$. Since $P$ is a rational polygon, there is an integer $n$ so that the arguments of the preimages of $v$ (resp. $\overline{v}$) on the unfolding $(X, \omega)$ are a subset of $\alpha + \frac{k \pi}{n}$ (resp. $-\alpha + \frac{k \pi}{n}$) where $k$ is an integer. The lift of $s$ to any preimage of $p$ is a straight line that must contain one preimage of $v$ and one preimage of $\overline{v}$, which implies that $\alpha$ is an integral multiple of $\frac{\pi}{2n}$. This implies that there is a finite set of directions on $(X, \omega)$ to which the lift of $s$ might be parallel. Since there is a finite collection of cylinders and saddle connections in these directions, the first claim follows.

For the second claim we will begin with the following sublemma. Continuing to adopt the notation of Definition \ref{D:Double}, let $T$ be a generator of the deck group of $\pi: (X, \omega) \ra D(P)$. If the order of $T$ is $n$, then $T^* \omega = \zeta \omega$ for an $n$th root of unity $\zeta$ (see for instance Mirzakhani-Wright \cite[Section 6]{MirWri2}). Therefore, if $T^k$ fixes a saddle connection or cylinder, it follows that either $k = n$ or $n = 2k$. In the latter case, the quotient of $(X, \omega)$ by $T^k$ is $(Y, q)$. Therefore, we see that the induced action of the deck group on $(Y, q)$ fixes no cylinder or saddle connection. 

If $N_{cyl, D(P)}(L)$ denotes the number of cylinders (resp. saddle connections) of length at most $L$ on $D(P)$ we see that
\[ d \cdot N_{cyl, D(P)}(L) = N_{cyl, (Y, q)}(L). \]
Recall that periodic directions on the billiard table $P$ correspond to a $J$-orbit of cylinders on $D(P)$. Since all but finitely many $J$-orbits of cylinders contain two cylinders it follows that $\displaystyle{\lim_{L \ra \infty} \frac{N_{P}(L)}{\frac{1}{2d} N_{cyl, (Y, q)}(L)}} = 1$. The final claim follows from identical arguments.
\end{proof}

\begin{proof}[Proof of Theorems \ref{T:Counting:ClosedTrajectories} and \ref{T:Counting:GeneralizedDiagonals}:]
Suppose first that $P$ is a rational right triangle with angles $\left( \frac{a}{2n}, \frac{b}{2n}, \frac{1}{2}\right)\pi$ where $a, b$ and $n$ are positive integers with $\mathrm{gcd}(a,b,2n) =1$ and $\min(a,b) > 2$. The  partial unfolding $(Y, q)$ of $P$ has dense orbit in $\cQ(a-2, b-2, -1^n)$ (by Theorem \ref{T:Main:Unfolding}) and is tiled by $2n$ isometric copies of $P$.  Therefore, by Corollary \ref{C:T:AEZ} and Lemma \ref{L:UnfoldingsVsPolygons},
\[ N_{P}(L) ``\sim" \frac{1}{(2n) 2\pi^2} \left(\frac{n(n-1)}{2} \right) \left( 1 + \frac{2}{ab}\right) \frac{\pi L^2}{(2n)\mathrm{area}(P)}, \]
which simplifies to $\frac{1}{16\pi} \left(1- \frac{1}{n} \right) \left( 1 + \frac{2}{ab}\right) \frac{L^2}{\mathrm{area}(P)}$. 

Suppose next that $p_1$ and $p_2$ are two distinct points on $P$. Let $P_1$ and $P_2$ be the preimages on $(Y, q)$. Let $\cC$ be the configuration consisting of saddle connections joining a point in $P_1$ to a point in $P_2$. Let $(Y, q; P_1 \cup P_2)$ be $(Y, q)$ with the points in $P_1 \cup P_2$ marked. By Apisa-Wright \cite{ApisaWright}, since $(Y, q)$ has dense orbit in $\cQ(a-2, b-2, -1^n)$, $(Y, q; P_1 \cup P_2)$ has dense orbit in $\cQ(a-2, b-2, 0^m, -1^n)$ where $m$ is the number of points in $P_1 \cup P_2$ that are not cone points. Let $q_i$ be a point in $P_i$. By Theorem \ref{T:AEZ}, 
\[ N_{\cC, (Y, q)}(L) ``\sim" |P_1||P_2| c_{q_1, q_2} \frac{\pi L^2}{\mathrm{area}(Y, q)}. \]
Theorem \ref{T:Counting:GeneralizedDiagonals}  now follows as in the previous case by Lemma \ref{L:UnfoldingsVsPolygons} (note that in the notation of Theorem \ref{T:Counting:GeneralizedDiagonals}, $n_i := |P_i|$ and $d_i$ is the order of vanishing of $q$ at a point in $P_i$). 


Suppose finally that $T$ is a rational isosceles triangle with angles $\left( \frac{a}{n}, \frac{a}{n}, \frac{b}{n} \right)\pi$ where $a, b$ and $n$ are positive integers with $\mathrm{gcd}(a, b, n) =1$. Then, when $n$ is odd and $\min(a,b) \ne 1$, the orbit closure $\cN$ of the unfolding $(X_1, \omega_1)$ of $T$ is the quadratic double of $\cQ(2a-2, b-2, -1^n)$ and is tiled by $2n$ copies of $T$ (by Theorem \ref{T:Unfolding:IsoscelesDetailed}). Since $n$ is odd $(X_1, \omega_1)$ is the partial unfolding of $D(P)$. By Theorem \ref{T:AEZ:Corollary} and Lemma \ref{L:UnfoldingsVsPolygons},
\[N_{T}(L) ``\sim" \frac{1}{(2n)2\pi^2} \left(\frac{n(n-1)}{2} \right) \left( 1 + \frac{4}{2ab}\right) \frac{2\pi L^2}{(2n)\mathrm{area}(T)}\]
which simplifies to $\frac{1}{8\pi} \left(1- \frac{1}{n} \right) \left( 1 + \frac{2}{ab}\right) \frac{L^2}{\mathrm{area}(P)}$
Suppose now that $n$ is even and $a \ne 1$ and $b \notin \{2, 4\}$. The partial unfolding $(X_2, q_2)$ of $D(T)$ is tiled by $n$ copies of $T$ and its orbit closure is the (unique) hyperelliptic component of a stratum of quadratic differentials that is a locus of double covers of $\cQ(a-2, \frac{b}{2}-2, -1^{\frac{n}{2}})$. By Theorem \ref{T:AEZ:Corollary} and Lemma \ref{L:UnfoldingsVsPolygons},
\[ N_{T}(L) ``\sim" \frac{1}{(n)2\pi^2} \left(\frac{\frac{n}{2}(\frac{n}{2}-1)}{2} \right) \left( 1 + \frac{1}{ab}\right) \frac{2\pi L^2}{(n)\mathrm{area}(T)}, \]
which simplifies to $\frac{1}{8\pi}  \left( 1 + \frac{1}{ab}\right)\left( 1-\frac{2}{n} \right) \frac{L^2}{\mathrm{area}(T)}$.
\end{proof}

\section{Finite blocking - Proof of Theorems \ref{T:Main:FiniteBlocking} }\label{S:FB}

We will begin by summarizing some relevant background on the finite blocking problem. The \emph{finite blocking problem for a translation surface} $(X, \omega)$ asks for the list of all pairs of finitely blocked points, i.e. points $\{p_1, p_2\}$ on $(X, \omega)$ so that there is a finite collection of points $B \subseteq (X, \omega) - \{p_1, p_2\}$ so that any line from $p_1$ to $p_2$ that does not pass through a singular point in its interior must pass through a point in $B$. If $(X, \omega)$ is the unfolding of a polygon $P$, then two points $\{q_1, q_2\}$ are finitely blocked if every preimage of $q_1$ is finitely blocked from every preimage of $q_2$. 

For any translation surface $(X, \omega)$ that is not a translation cover of a torus, M\"oller \cite[Theorem 2.6]{M2} showed that there is a unique maximal degree translation cover, denoted $\pi_{X_{min}}: (X, \omega) \ra (X_{min}, \omega_{min})$ with domain $(X, \omega)$. Similarly, by Apisa-Wright \cite[Lemma 3.3]{ApisaWright}, there is a quadratic differential $(Q_{min}, q_{min})$ and a degree one or two (half)-translation cover $\pi: (X_{min}, \omega_{min}) \rightarrow (Q_{min}, q_{min})$ so that any half-translation cover is a factor of $\pi_{Q_{min}} := \pi \circ \pi_{X_{min}}$. 

\begin{thm}[Apisa-Wright \cite{ApisaWright} Theorem 3.6 and Lemma 3.8]\label{T:FiniteBlockingBackground}
If $(X, \omega)$ is a translation surface that is not a translation cover of a torus then the only finitely blocked pairs of points are the following:
\begin{enumerate}
    \item Two periodic or singular points blocked by the other periodic and singular points.
    \item Two points with the same image, call it $p$, under $\pi_{Q_{min}}$ blocked by periodic and singular points and the other preimages of $p$.
\end{enumerate}
\end{thm}

Recall that a point $p$ on $(X, \omega)$ is called \emph{periodic} if it is a non-singular point and the orbit closure of $(X, \omega)$ and $(X, \omega; \{p\})$ have the same dimension. A non-singular point that is not a periodic point is called \emph{generic}.

Therefore, to solve the finite blocking problem on a rational polygon it is useful to know the minimal half-translation cover and the periodic points on its unfolding. 

\begin{lem}\label{L:MinimalCoverUnfolding}
The only rational right triangles that unfold to torus covers are the $45-45-90$ and $30-60-90$ triangles. All other rational right triangles have their minimal degree half-translation cover given by the quotient by the hyperelliptic involution. 

The only rational isosceles triangles that unfold to torus covers are the $45-45-90$, $60-60-60$, and $30-30-120$ triangles. All other rational isosceles triangles with angles $\left( \frac{a}{n}, \frac{a}{n}, \frac{b}{n} \right)\pi$ with $a$, $b$, and $n$ positive integers with $\mathrm{gcd}(a,b,n) = 1$ have their minimal degree half-translation cover given by the quotient by the hyperelliptic involution (when $n$ is odd) and a degree four map to a genus zero surface (when $n$ is even). 
\end{lem}
\begin{proof}
By Theorems \ref{T:Unfolding:RightDetailed} and \ref{T:Unfolding:IsoscelesDetailed}, the only rational and isosceles triangles that unfold to a translation surface $(X, \omega)$ with a rank one orbit closure have the property that $(X, \omega)$ is one of the following: the regular $n$-gon (for $n$ even), the double regular $n$-gon locus (for $n$ odd), or a double cover of one of those loci. The only such surfaces that are torus covers are the regular $4$-gon, regular $6$-gon (equivalently the double regular $3$-gon), and double covers thereof. The claim about which triangles unfold to torus covers now follows from Theorems \ref{T:Unfolding:RightDetailed} and \ref{T:Unfolding:IsoscelesDetailed}.

Since a surface with dense orbit in a genus zero stratum of rank at least two cannot be the domain of a nontrivial half-translation cover (Apisa-Wright \cite[Lemma 4.5]{ApisaWright}) the other claims follow by Theorems \ref{T:Unfolding:RightDetailed} and \ref{T:Unfolding:IsoscelesDetailed} when the unfolding is neither Veech nor a double cover of a Veech surface.  However, in those cases, the unfolding is either a cover of the regular $n$-gon or double regular $n$-gon locus, for which the minimal half-translation cover is the quotient by the hyperelliptic involution (by Apisa-Saavedra-Zhang \cite[Lemma 2.7]{ApisaSaavedraZhang}).
\end{proof}

The following will allow us to compute the periodic points whenever an unfolding is not a cover of a Veech surface.

\begin{thm}[Apisa-Wright \cite{ApisaWright} Theorem 1.3]\label{T:QuadraticPeriodicBackground}
Suppose that $\cQ$ is a component of a stratum of quadratic differentials of rank at least two. If $\cQ$ is a hyperelliptic component, then the only periodic points are fixed points of the hyperelliptic involution. Otherwise, there are no periodic points.
\end{thm}

We will use this theorem to show the following:

\begin{prop}\label{P:FB:Zero}
Suppose that $(X, \omega)$ is a holonomy double cover of a genus zero surface $(Y, q)$ with dense orbit in a component $\cQ$ of a stratum of quadratic differentials of rank at least two and with at most two zeros. Then the only pairs of finitely blocked points on $(X, \omega)$ are pairs of nonsingular points exchanged by the hyperelliptic involution, blocked by the Weierstrass points.
\end{prop}
\begin{proof}
Let $p_1$ and $p_2$ be two not necessarily distinct points on $(Y, q)$. Let $P$ be the possibly empty set consisting of the non-singular points in $\{p_1, p_2\}$. By Apisa-Wright \cite{ApisaWright}, since $(Y, q)$ has dense orbit in $\cQ$, $(Y, q; P)$ has dense orbit in the stratum containing it. By considering the configurations in Definition \ref{D:Configurations}, Eskin-Mirzakhani-Mohammadi \cite[Theorem 2.12]{EMM} and Athreya-Eskin-Zorich \cite[Theorems 4.3, 4.5, and 4.8]{AEZ} (as recalled in Theorem \ref{T:AEZ}) imply that $p_1$ and $p_2$ are finitely blocked from each other if and only if $p_1 = p_2$ is a pole (note that the case where $p_1 = p_2$ is a zero is best understood using configuration 2 in \cite{AEZ}). This shows that the only finite blocking that occurs on $(Y, q)$ is the blocking of a pole from itself.

Since non-singular points on $(X, \omega)$ are blocked from their image under the hyperelliptic involution by Weierstrass points, it only remains to consider blocking between two preimages of zeros of even order on $(Y, q)$ by Theorems \ref{T:FiniteBlockingBackground} and \ref{T:QuadraticPeriodicBackground}. (Note that odd order zeros have a single preimage on $(X, \omega)$ and hence we can solve the blocking problem by lifting paths from $(Y, q)$ to $(X, \omega)$.) The blocking set between two such points must be a subset of the collection of Weierstrass points by Theorems \ref{T:FiniteBlockingBackground} and \ref{T:QuadraticPeriodicBackground}. 

First suppose that $z$ is an even order zero on $(Y, q)$ whose preimages are $z_1$ and $z_2$ on $(X, \omega)$. To see that $z_i$ is not blocked from itself we note that by Athreya-Eskin-Zorich \cite[Theorem 4.5]{AEZ} there are simple envelopes on $(Y, q)$, one boundary of which consists of a saddle connection joining $z$ to itself. The preimage of this envelope on $(X, \omega)$ is a simple cylinder bounded by one saddle connection joining $z_1$ to itself and another joining $z_2$ to itself (neither of which passes through poles). This establishes that $z_i$ is not blocked from itself.

We will now show that $z_1$ and $z_2$ are not blocked from each other. By Athreya-Eskin-Zorich \cite[Theorem 4.4]{AEZ} there is a saddle connection $s$ on $(Y, q)$ that joins $z$ to itself and so that the there is a component of $(Y, q) - s$ whose only cone points are an odd number of poles. The homology class of $s$ on $Y - C(q)$, where $C(q)$ denotes the cone points of $q$, can be identified with the sum of loops around these poles. Since the poles are branch points it follows that the preimage of $s$ on $(X, \omega)$ is a single closed loop (not two closed loops). This implies that a lift of the saddle connection $s$ to $(X, \omega)$ joins $z_1$ and $z_2$ and does not pass through Weierstrass points, as desired.

Finally, suppose that $w$ is an even order zero on $(Y, q)$ that is distinct from $z$ and whose preimages on $(X, \omega)$ are $w_1$ and $w_2$. We will show that $z_i$ and $w_j$ are not finitely blocked from each other for any $i$ and $j$. By Athreya-Eskin-Zorich \cite[Theorem 4.4]{AEZ} there is a simple cylinder on $(Y, q)$ whose boundary consists of one saddle connection joining $z$ to itself and one joining $w$ to itself. The preimage of this cylinder on $(X, \omega)$ consists of two cylinders, $C_1$ and $C_2$, so that, up to relabelling, the boundary of $C_i$ consists of a saddle connection joining $z_i$ to itself and another joining $w_i$ to itself. This demonstrates that $z_i$ and $w_i$ are not finitely blocked from each other. 

It suffices now to show that $z_1$ is not blocked from $w_2$. By perturbing $(X, \omega)$ we may assume that it is cylindrically stable in the direction of $\{C_1, C_2\}$ (by Apisa-Wright \cite[Lemma 7.10 (6)]{ApisaWrightHighRank}). Suppose without loss of generality that $C_1$ is horizontal and that its top boundary consists of a saddle connection joining $z_1$ to itself. Let $D$ denote the horizontal cylinder that borders $C_1$ along this saddle connection. Since $\cQ$ has at most two zeros, $C_1$ and $C_2$ are the only horizontal cylinders that are not fixed by the hyperelliptic involution. In particular, $z_2$ belongs to the top boundary of $D$. By shearing $D$ and equally shearing the cylinders in $\{C_1, C_2\}$, it is straightforward to produce a surface on which the vertical separatrix that travels down from $z_2$ crosses the core curve of $D$ and $C_1$ once and then terminates at $w_1$ without passing through a Weierstrass point. 
\end{proof}

\begin{lem}\label{L:Qhyp-cyl}
Suppose that $(X, \omega)$ is a holonomy double cover of a surface $(Y, q)$ with dense orbit in a hyperelliptic component $\cQ$ of a stratum of quadratic differentials of rank at least two. Let $T$ be the translation involution on $(X, \omega)$ and $J$ the holonomy involution. Suppose that the surfaces in $\cQ$ are double covers of surfaces belonging to the genus zero stratum $\cQ(a, b, -1^{a+b+4})$ where $a \geq b \geq -1$ are integers. The following subequivalence classes of generic cylinders occur on $(X, \omega)$ unless indicated otherwise. 
\begin{enumerate}
    \item A pair of simple cylinders that are exchanged by $J$ and $T$.
    \item A single complex cylinder that is fixed by $J$ and $T$. (This only occurs if $b = -1$)
    \item A pair of complex cylinders that are exchanged by $J$ and $JT$. (This only occurs if $b \ne -1$)
\end{enumerate}
Moreover, it is possible to produce a pair of simple cylinders (resp. a single complex cylinder) that contains the $T$-orbit of two (resp. one) specified fixed point of $JT$.
\begin{figure}[h]\centering
\includegraphics[width=\linewidth]{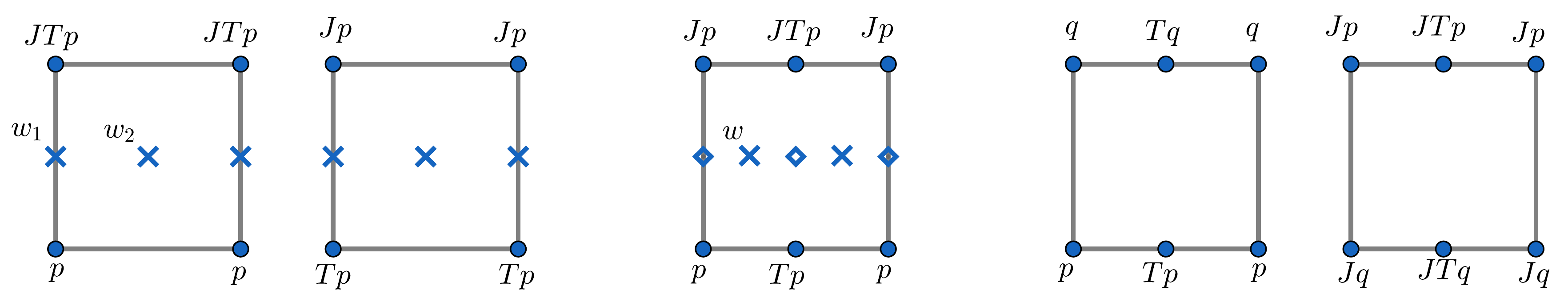}
\caption{These three horizontal cylinders are the three possible subequivalence classes of generic cylinders appearing on a surface in the locus of holonomy double covers of a hyperelliptic component of a stratum of quadratic differentials. The points marked with an ``x" (resp. diamond) are fixed points of $JT$ (resp. $J$); the point $p$ is a cone point and the point $q$ is a cone point that does not belong to the same $\{\mathrm{id}, J, JT, T\}$-orbit.}
\label{F:CylinderTypesQHyp}
\end{figure}
\end{lem}
\begin{proof}
Let $\tau$ denote the hyperelliptic involution on $(Y, q)$. The only subequivalence classes of generic cylinders on $(Y, q)/\tau$ are simple envelopes and simple cylinders (since $(Y, q)/\tau$ belongs to a genus zero stratum). Let $C$ be a such a cylinder and let $\bfC$ be its preimage on $(X, \omega)$.

If $C$ is a simple cylinder, then $\bfC/T$ is a pair of simple cylinders and $\bfC/J$ is a complex cylinder. It follows that $\bfC$ consists of a pair of complex cylinders that are exchanged by $J$ and $T$. Note that surfaces in $\cQ(a, b, -1^{a+b+4})$ (where $a \geq b \geq -1$) contain generic simple cylinders if and only if $b \ne -1$. 

If $C$ is a simple envelope, then suppose first that both poles on its boundary are branch points for the map $(Y, q) \ra (Y, q)/\tau$ (this is always the case if $b \ne -1$). In this case, $\bfC/J$ and $\bfC/T$ both consist of a single simple cylinder and hence $\bfC$ is a pair of simple cylinders that are exchanged by $J$ and $T$.

If $C$ is a simple envelope and a pole on its boundary is a branch point, then $b = -1$ and $\bfC/J$ is a complex envelope. It follows that $\bfC$ is a complex cylinder that is fixed by $J$ and $T$.

Finally, notice that by Athreya-Eskin-Zorich \cite[Theorem 4.5]{AEZ}, we can find $C$ on $(Y, q)/\tau$ that is an envelope with any pair of poles we like on its boundary. This implies the final claim.
\end{proof}

\begin{prop}\label{P:FB:Hyp}
Under the same assumptions as Lemma \ref{L:Qhyp-cyl},
\begin{enumerate}
    \item\label{I:FB:Hyp:General} A generic point $p$ is only finitely blocked from $Jp$ and $JTp$. 
    \item\label{I:FB:Hyp:PeriodicPeriodic} Each fixed point of $JT$ is not finitely blocked from at least one other such point
    \item\label{I:FB:Hyp:SingularPeriodic} Every singular point on $(X, \omega)$ is not finitely blocked from at least one nonsingular fixed point of $J$ and $JT$.
    \item\label{I:FB:Hyp:SingularSingular} No zero (not including marked points) on $(X, \omega)$ is finitely blocked from any other.
\end{enumerate}
\end{prop}
\begin{proof}
First, let $p$ be a generic point on $(X, \omega)$. Since $p/T$ is finitely blocked from $(Jp)/T$, it follows that $p$ is finitely blocked from $Jp$ and $JTp$. By Theorems \ref{T:FiniteBlockingBackground} and \ref{T:QuadraticPeriodicBackground}, the only other points from which $p$ may be blocked are $p$ and $Tp$, in which case the blocking set would consist of points in $\{p, Tp, Jp, JTp\}$, fixed points of $J$, and preimages of fixed points of $\tau$ (this set consists of fixed points of $JT$). Since $p$ is generic, it suffices to show that the generic point on a generic surface in the locus $\wt{\cQ}$ of holonomy double covers of surfaces in $\cQ$ is not blocked from itself or its image under $T$. This is immediate from Lemma \ref{L:Qhyp-cyl} (see for instance the cylinder on the left in Figure \ref{F:CylinderTypesQHyp}). This establishes \eqref{I:FB:Hyp:General}. Notice that \eqref{I:FB:Hyp:PeriodicPeriodic} and \eqref{I:FB:Hyp:SingularPeriodic} follow similarly from Lemma \ref{L:Qhyp-cyl} (see for instance the cylinders in the left and middle of Figure \ref{F:CylinderTypesQHyp}).

If $p$ is a zero, then it is not finitely blocked from $p$ or $Tp$ by Lemma \ref{L:Qhyp-cyl} (see for instance the cylinder on the left in Figure \ref{F:CylinderTypesQHyp}). We will show additionally that $p$ is not finitely blocked from $Jp$ or $JTp$. Let $\alpha+1$ be the order of vanishing of $\omega$ at $p$. If $\alpha$ is odd, then $p = Jp$ and there is nothing to prove. If $\alpha$ is even, then $p = Tp$ and by Proposition \ref{P:FB:Zero} we can pull back a saddle connection from $p/T$ to $Jp/T$ that does not pass through Weierstrass points to $(X, \omega)$. (See Apisa-Wright \cite[Section 9.1]{ApisaWrightDiamonds} for a review of the connection between the parity of $\alpha$ and the $J$ and $T$ invariance of $p$.). 

Suppose now that there are two zeros or marked points on $\cQ$. Let $p$ and $q$ be a preimage of each one on $(X, \omega)$. We may choose the preimages so that $p$ is not finitely blocked from $q$ or $Tq$ by Lemma \ref{L:Qhyp-cyl} (see for instance the cylinder on the right in Figure \ref{F:CylinderTypesQHyp}). Let $\alpha+1$ (resp. $\beta+1$) be the order of vanishing of $\omega$ at $p$ (resp. $q$). If $\beta$ is odd, then $Jq = q$ and so $p$ is not finitely blocked from $Jq$ or $JTq$. Suppose therefore that $\beta$ is even, which implies that $Tq = q$. If $\alpha$ is odd, then $Jp = p$ and hence the fact that $p$ is not finitely blocked from $q$ or $Tq$ implies that $p$ is not finitely blocked from $Jq$ or $JTq$.  Suppose therefore that $\alpha$ is also even, which implies that $Tp = p$. By Proposition \ref{P:FB:Zero} we can pull back a saddle connection from $p/T$ to $Jq/T$ that does not pass through Weierstrass points to $(X, \omega)$ to produce the desired saddle connection on $(X, \omega)$.  This establishes \eqref{I:FB:Hyp:SingularSingular}.
\end{proof}

The following result is a more explicit version of Theorem \ref{T:Main:FiniteBlocking}.

\begin{thm}\label{T:Main:FiniteBlocking:Detailed}
The only pairs of finitely blocked points on a rational right triangle belong to the following list:
\begin{enumerate}
    \item Any two points on the $30-60-90$ or $45-45-90$ triangle.
    \item The vertex of angle $\frac{\pi}{n}$ is finitely blocked from itself on the triangle with angles $( \frac{1}{n}, \frac{n-2}{2n}, \frac{1}{2})\pi$ for $n$ even. The blocking set may be taken to be empty.
\end{enumerate}
The only pairs of finitely blocked points on a rational isosceles triangle belong to the following list:
\begin{enumerate}
    \item In the $45-45-90$, $60-60-60$, or $30-30-120$ triangles any two points are finitely blocked from each other.
    \item In the $( \frac{1}{n}, \frac{1}{n}, \frac{n-2}{n})\pi$ triangle with $n$ even, any two (resp. any two distinct) vertices of angle $\frac{\pi}{n}$ are blocked from each other when $n$ is even (resp. odd). 
    \item In the $(\frac{n-1}{2n}, \frac{n-1}{2n}, \frac{1}{n})\pi$ triangle, if $n$ is odd, then the vertex of angle $\frac{\pi}{n}$ is blocked from itself. 
    \item In the $(\frac{n-1}{n}, \frac{n-1}{n}, \frac{2}{n})\pi$ triangle, if $n$ is even, then the vertex of angle $\frac{2\pi}{n}$ is blocked from itself. 
\end{enumerate}
In all but the first case the blocking set may be taken to be the midpoint of the edge joining the two vertices of equal angle. 
\end{thm}
\begin{proof}
Let $P$ be the right triangle with angles $\left( \frac{a}{2n}, \frac{b}{2n}, \frac{1}{2}\right)\pi$ with $\mathrm{gcd}(a,b,n) = 1$. If the unfolding is a torus (i.e. if $n = 2, 3$; see Lemma \ref{L:MinimalCoverUnfolding}), then any two points are finitely blocked from each other by Gutkin \cite{GutkinBlocking1}. If the unfolding is Veech, but $n \notin \{2, 3\}$, then the claim holds by Apisa-Saavedra-Zhang \cite[Corollary 1.7]{ApisaSaavedraZhang}. Therefore, suppose that $\min(a,b) \notin \{1,2\}$. The unfolding has dense orbit in the quadratic double of $\cQ(a-2,b-2,-1^n)$ (by Theorem \ref{T:Unfolding:RightDetailed}) and so no points are finitely blocked from each other (by Proposition \ref{P:FB:Zero}).

Let $T$ be the isosceles triangle with angles $\left( \frac{a}{n}, \frac{a}{n}, \frac{b}{n}\right)\pi$ with $\mathrm{gcd}(a,b,n) = 1$. Let $(X, \omega)$ denote the unfolding. If the unfolding is a torus (i.e. if $n = 3$, $4$, or $6$, by Lemma \ref{L:MinimalCoverUnfolding}), then any two points are finitely blocked from each other by Gutkin \cite{GutkinBlocking1}. Suppose now that $n \notin \{3, 4, 6\}$. Let $p$ denote the midpoint of the edge of $T$ that joins the two vertices of equal angle. Let $P$ be the preimages of $p$ on the unfolding. By Theorem \ref{T:Unfolding:IsoscelesDetailed}, $(X, \omega)$ is contained in the quadratic double of a stratum $\cQ_0$, where $\cQ_0 = \cQ(2a-2, 2b-2, -1^n)$ when $n$ is odd and where $\cQ_0$ is the hyperelliptic component of a stratum that covers $\cQ(a-2, \frac{b}{2} -2, -1^{\frac{n}{2}})$ when $n$ is even. Moreover, the points in $P$ are precisely the preimages of Weierstrass points under the quotient of the unfolding by the holonomy involution.


If $(X, \omega)$ is not a cover of a Veech surface then its orbit closure is the quadratic double of $\cQ_0$. Since there are $n \geq 5$ preimages of every point on $T$, excluding vertices, on $(X, \omega)$, it follows from Propositions \ref{P:FB:Zero} and \ref{P:FB:Hyp} \eqref{I:FB:Hyp:General} that only vertices or $p$ are finitely blocked from other points. These points can only be blocked from other vertices or from $p$ by Theorems \ref{T:FiniteBlockingBackground} and \ref{T:QuadraticPeriodicBackground}. The solution of the finite blocking problem in this case now follows from Proposition \ref{P:FB:Zero} and \ref{P:FB:Hyp} \eqref{I:FB:Hyp:PeriodicPeriodic}, \eqref{I:FB:Hyp:SingularPeriodic}, and \eqref{I:FB:Hyp:SingularSingular}. 

%

Suppose now that $(X, \omega)$ is a cover of a Veech surface. Suppose additionally that $n$ is odd. It follows that $(X, \omega)$ is the double regular $n$-gon (if $a = 1$) or the regular $2n$-gon (if $b=1$). By Apisa-Saavedra-Zhang \cite[Corollary 1.6]{ApisaSaavedraZhang}, the two distinct vertices of angle $\frac{\pi}{n}$ are blocked from each other if $a=1$ and the vertex of angle $\frac{\pi}{n}$ is blocked from itself if $b=1$ and there is no other blocking. 

Suppose now that $n$ is even. Since $(X, \omega)$ covers a Veech surface it is a double cover of the regular $n$-gon (if $b \ne 4$) or the double regular $\frac{n}{2}$-gon (if $b = 4$). By Apisa-Saavedra-Zhang \cite[Theorem 1.3]{ApisaSaavedraZhang}, the only periodic points are preimages of Weierstrass points. The points on $T$ whose preimage on $(X, \omega)$ are singular or periodic points are precisely vertices of $T$ and $p$. Since four is the degree of the minimal half-translation cover with domain $(X, \omega)$, it follows that a generic point is blocked from at most four others by Theorem \ref{T:FiniteBlockingBackground}. Since each point on $T$ has at least $n \geq 8$ preimages on $(X, \omega)$, it follows from Theorem \ref{T:FiniteBlockingBackground} that only vertices of $T$ and $p$ are finitely blocked from other points and these other points must be other vertices or the point $p$ with the blocking set consisting of $p$. By dropping a perpendicular from $p$ to another edge of $T$, we see that $p$ is not finitely blocked from itself (and hence from any point on $T$). 

Note that, by lifting a saddle connection from the regular $n$-gon or the double regular $\frac{n}{2}$-gon, we see that the vertex of angle $\frac{b \pi}{n}$ is not blocked from itself unless $b = 2$ by Apisa-Saavedra-Zhang \cite[Corollary 1.6]{ApisaSaavedraZhang}.

It remains to consider blocking between the two vertices of angle $\frac{a \pi}{n}$. When $a = 1$, the preimages of these vertices is precisely the preimage of a Weierstrass point on the regular $n$-gon. Since such a point is finitely blocked from itself it follows that when $n$ is even and $a = 1$ any two vertices of angle $\frac{\pi}{n}$ are finitely blocked from each other. 

Suppose now that $a \ne 1$. In particular $a \geq 3$, $n \geq 8$ and $b = 2, 4$. Since $\frac{b \pi }{n}$ and $\frac{a \pi}{n}$ are both acute angles, we see that dropping a perpendicular from a vertex of angle $\frac{a \pi}{n}$ to the opposite side shows that the vertex is not finitely blocked from itself. As illustrated in Figure \ref{F:IsoscelesBilliardTrajectory}, the two vertices of equal angle are not finitely blocked from each other when $a > b$. 

\begin{figure}[h]\centering
\includegraphics[width=.4\linewidth]{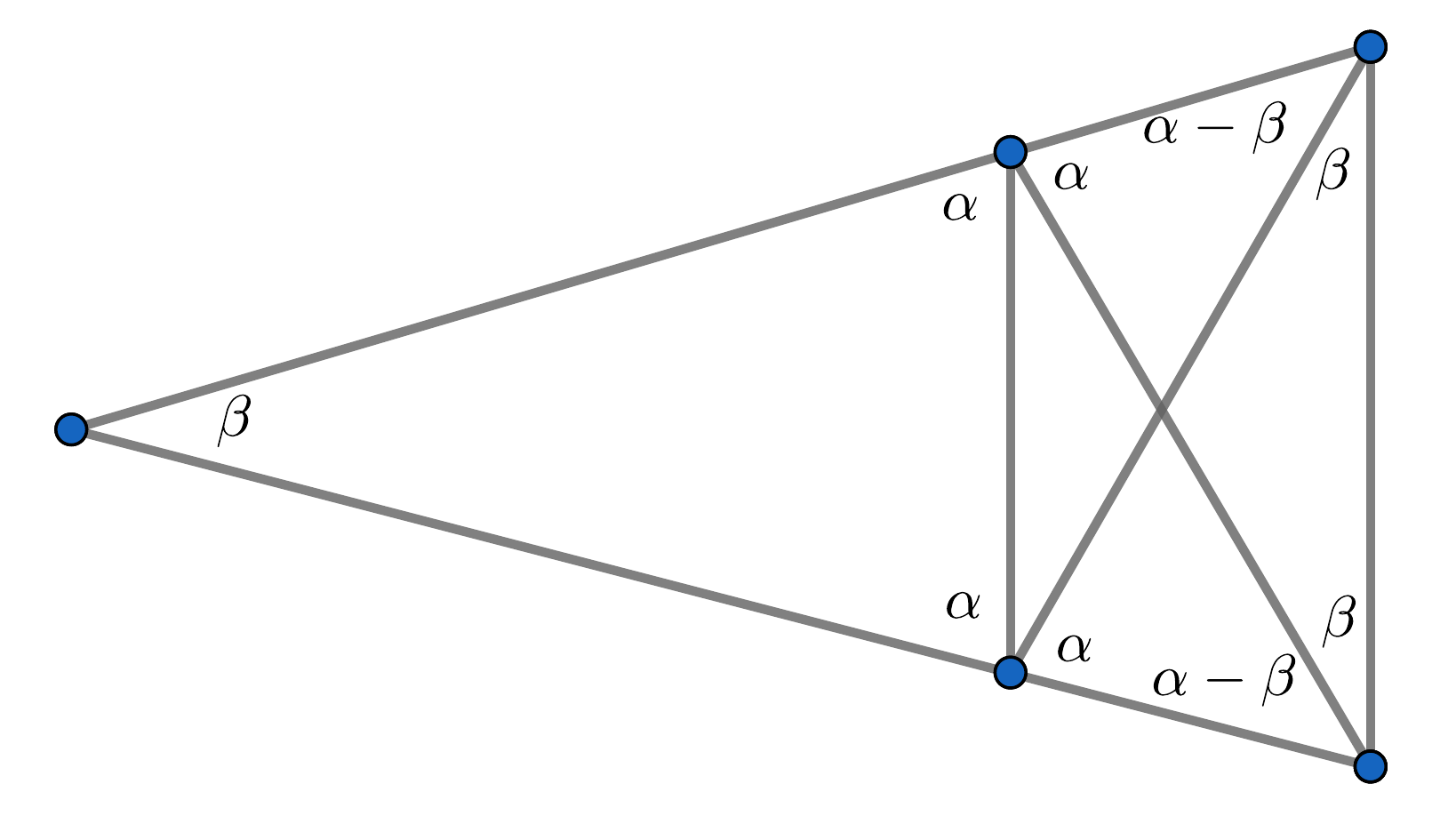}
\caption{Let $\alpha = \frac{a \pi}{n}$ and $\beta = \frac{b \pi}{n}$. When $a > b$ the two vertices of equal angle are not finitely blocked from each other.}
\label{F:IsoscelesBilliardTrajectory}
\end{figure}

This accounts for all the remaining isosceles triangles except for the $\left( \frac{3}{10}, \frac{3}{10}, \frac{4}{10} \right) \pi$ triangle, whose unfolding is depicted in Figure \ref{F:SpecialIsoUnfolding}. This unfolding covers the double pentagon locus and is branched over two marked points that are exchanged by the hyperelliptic involution. Since these marked points are not periodic points the saddle connections labelled $4$ and $5$ may be deformed (as long as they have the same holonomy), while remaining in the orbit closure of $(X, \omega)$. After such a perturbation it is clear that the two distinct zeros illuminate each other on $(X, \omega)$ as desired.
\begin{figure}[h]\centering
\includegraphics[width=.5\linewidth]{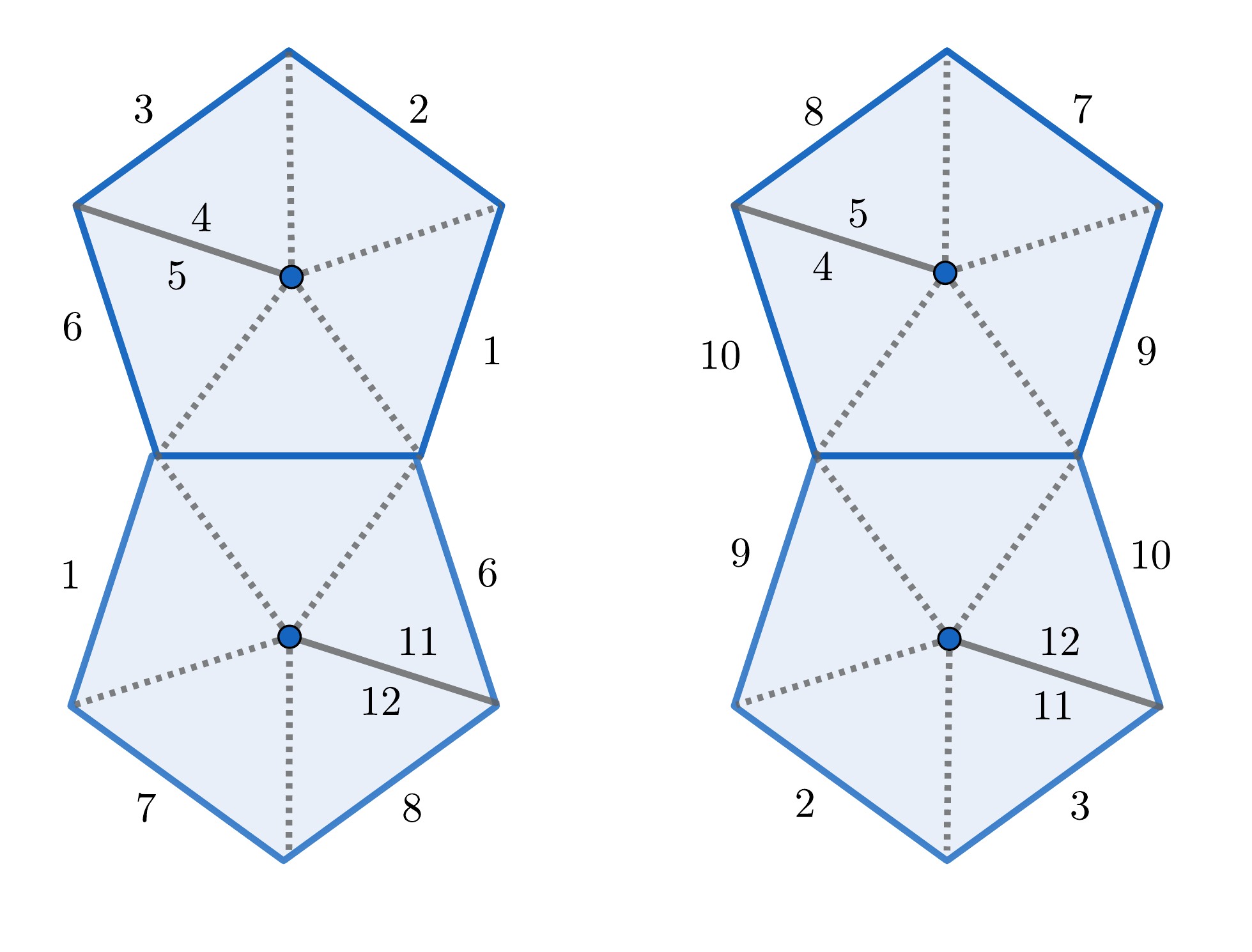}
\caption{This translation surface is the unfolding of the $\left( \frac{3}{10}, \frac{3}{10}, \frac{4}{10} \right) \pi$ triangle $T$. The dotted lines indicate the tiling of the unfolding by copies of $T$.}
\label{F:SpecialIsoUnfolding}
\end{figure}
\end{proof}

\section{Further applications to billiards - Proof of Theorem \ref{T:Unfolding:Parallelogram} and Theorem \ref{T:Main:SpecificSphere}}

In this section we will determine the orbit closure of all but a countable set of rational parallelograms, isosceles trapezoids, and right trapezoids.

\begin{prop}\label{P:Parallelogram}
If $\theta := \left( \frac{a}{n}, \frac{b}{n}, \frac{a}{n}, \frac{b}{n} \right)\pi$, where $a, b$ and $n$ are positive integers with $\mathrm{gcd}(a,b,n) = 1$ and $a+b = n$, then $\cM(\theta)$ is the quadratic double of $\cQ(2a-2, 2b-2, -1^{2n})$ if $n$ is odd and of $\cQ^{hyp}((a-2)^2, (b-2)^2)$ if $n$ is even.
\end{prop}
\begin{proof}
The case of $n = 2$ is immediate and the case of $n = 3$ follows from Mirzakhani-Wright \cite[Theorem 1.3]{MirWri2}. Therefore, suppose that $n \notin \{2, 3\}$.

Let $P$ be the rhombus with angles in $\theta$ and let $(X, \omega)$ be its unfolding, which is tiled by $2n$ copies of $P$. The unfolding belongs to $\cH((a-1)^2, (b-1)^2)$ where the superscripts indicate multiplicity. Let $T$ denote a generator of the cyclic deck group.

Since every parallelogram has a rotation by $\pi$ symmetry, the unfolding of every parallelogram has this symmetry as well, call this $J$ on $(X, \omega)$ (see Figure \ref{F:ParallelogramSymmetry}). Notice that $J$ does not fix any zero and fixes each copy of $P$ that tiles $(X, \omega)$. This shows that $J$ fixes $2n$ points (the midpoint of each copy of $P$ that tiles $(X, \omega)$). Consequently, $J$ is the hyperelliptic involution and $(X, \omega)/J \in \cQ := \cQ(2a-2, 2b-2, -1^{2n})$. 
\begin{figure}[h]\centering
\includegraphics[width=0.75\linewidth]{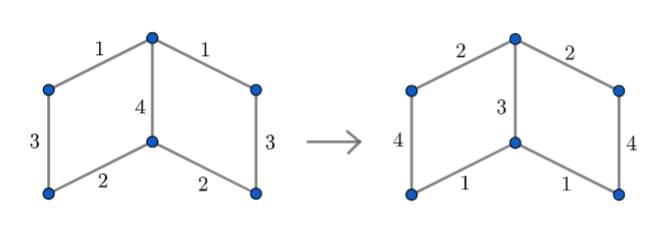}
\caption{The indicated map is the symmetry that $J$ induces on the double of $P$.}
\label{F:ParallelogramSymmetry}
\end{figure}

Let $\cM'$ be the  invariant subvariety corresponding to $\cM(\theta)$ in $\cQ$. By the ``trivial rank bound" of Mirzakhani-Wright \cite[Lemma 7.1]{MirWri2}, $\cM'$ has rank at least two and hence is a locus of covers by Theorem \ref{T:Main}.

\begin{sublem}\label{SL:BigSublocus}
$\cM'$ contains a full locus of covers of $\cQ(a-2,b-2, -1^n)$.
\end{sublem}
\begin{proof}
Since $P$ was taken to be a rhombus, there is an additional symmetry of its double, $D(P)$, which is the holomorphic map that fixes two singular points of $D(P)$ and exchanges the other two. This map will be called $S$ and is pictured in Figure \ref{F:RhombusSymmetry}.

\begin{figure}[h]\centering
\includegraphics[width=0.75\linewidth]{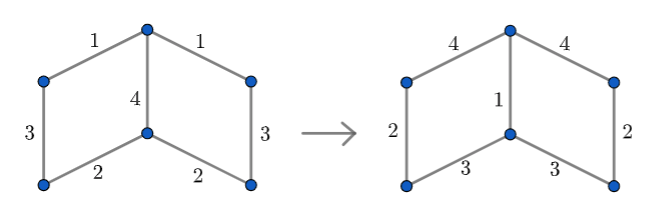}
\caption{The surface on the left is $D(P)$ and the indicated map is $S$.}
\label{F:RhombusSymmetry}
\end{figure}

Letting $J_D$ be the symmetry that $J$ induces on $D(P)$ we see that $J_D$ and $S$ generate a group of holomorphic symmetries of $D(P)$ that is isomorphic to $\mathbb{Z}/2\mathbb{Z} \times \mathbb{Z}/2\mathbb{Z}$. The isomorphism can be exhibited by noticing that the symmetries belong to the subgroup of $\mathrm{Sym}(4)$ that take cone points of a given cone angle to cone points of the same cone angle. Letting $R$ denote the right triangle with cone angles $\left( \frac{a}{2n}, \frac{b}{2n}, \frac{1}{2}\right)\pi$ and letting $D(R)$ denote its double, we see that $D(P)/\langle J_D, S \rangle = D(R)$. Let $(Y, \eta)$ denote the unfolding of $D(R)$. 

Let $f: D(P) \ra D(R)$ be the indicated map between $D(P)$ and $D(R)$. Let $V(R)$ be the singular points on $D(R)$ and let $V(P)$ be their preimages on $D(P)$. Let $\rho: \pi_1(D(R) - V(R)) \ra \mathbb{Z}/2n\mathbb{Z}$ be the holonomy map and notice that the holonomy map on $D(P) - V(P)$ is given by $\rho \circ f_*$. Therefore, there is an induced translation cover $F: (X, \omega) \ra (Y, \eta)$, which has degree at most four. 

By Lemma \ref{L:MinimalCoverUnfolding}, since $n \notin \{2, 3\}$, $\pi_{Y_{min}}$ is the quotient by the hyperelliptic involution and so $(X_{min}, \omega_{min})$ belongs to $\cQ(a-2,b-2, -1^n)$. In particular, $\cM'$ contains a full locus of covers of $\cQ(a-2,b-2, -1^n)$. 
\end{proof}

Suppose first that $n$ is even. Then $(X, \omega)$ is contained in the quadratic double of $\cQ((a-2)^2, (b-2)^2)$. Let $J'$ be the involution on $(X, \omega)$ whose quotient belongs to this stratum. Notice that $J$ and $J'$ commute since $J$ is the hyperelliptic involution. Since $J$ fixes no zero, $(X, \omega)/\langle J, J' \rangle \in \cQ(a-2, b-2, -1^n)$. The fact that $a+b=n$, $n$ is even, and $\mathrm{gcd}(a,b,n) =1$ implies that $a$ and $b$ are odd. Since the hyperelliptic component of $\cQ((a-2)^2, (b-2)^2)$ consists precisely of the double covers of surfaces in $\cQ(a-2, b-2, -1^n)$ that are branched over poles and even order zeros, of which there are none, (see Lanneau \cite{Lconn} and the discussion in Apisa-Wright \cite[Section 9]{ApisaWrightDiamonds}), it follows that $(X, \omega)$ belongs to the quadratic double of $\cQ^{hyp}((a-2)^2, (b-2)^2)$, which is a full locus of covers of $\cQ(a-2,b-2, -1^n)$. Therefore, $\cM(\theta) = \cQ^{hyp}((a-2)^2, (b-2)^2)$ by Sublemma \ref{SL:BigSublocus}.

Suppose now that $n$ is odd. Our goal is to show that $\cM'$ coincides with the stratum of quadratic differentials containing it. Equivalently, since $\cM'$ is a locus of covers, it suffices to show that the generic surface in $\cM'$ is the domain of a minimal half-translation cover that is degree one. Suppose in order to derive a contradiction that this is not the case. By Sublemma \ref{SL:BigSublocus}, the minimal half-translation cover with domain $(X, \omega)/J$ has degree two. Therefore, $\cM' \subseteq \cQ(2a-2, 2b-2, -1^{2n})$ is a locus of double covers of  $\cQ(a-2, b-2, -1^n)$.


Let $S'$ denote the involution on $(X, \omega)/J$ whose quotient is the double cover. Notice that $S'$ is an isometry that fixes each element in the set $\Sigma$ consisting of the two singularities of order $2a-2$ and $2b-2$. Since $T$ and $J$ commute, it follows that $T$ induces an isometry $T'$ on $(X, \omega)/J$ that has the same order as $T$ (since $n$ is odd). Since $T$ fixes the zeros of $(X, \omega)$, $T'$ fixes the elements of $\Sigma$. Since $T'$ and $S'$ are both holomorphic self-maps of $(X, \omega)/J$ that are isometries of the flat metric and that fix the points in $\Sigma$, it follows that they commute in a neighborhood of the points in $\Sigma$ and hence, since they are holomorphic, commute. Since $T'$ has order $n$, which is odd, $T'$ and $S'$ generate a group of $\mathbb{Z}/2n\mathbb{Z}$ symmetries.

Therefore, $S'$ induces an isometry on $(X, \omega) / \langle J, T' \rangle$ that fixes the two cone points of cone angle different than $\pi$ and that necessarily exchanges the two cone points of cone angle equal to $\pi$. Notice that $S$ also induces such a map on $D(P)/J_D$. This induced map is precisely the map induced by $S'$ since M\"obius transformations are determined by the image of three distinct points. 

\begin{figure}[h]\centering
\includegraphics[width=0.75\linewidth]{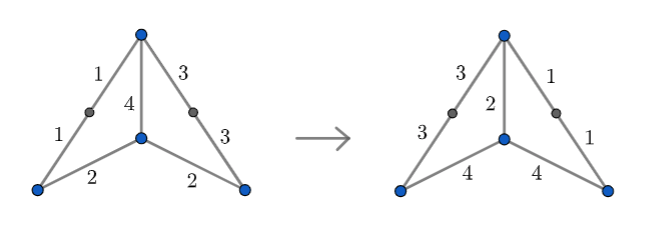}
\caption{The surface on the left is $D(P)/J_D$. The indicated map is the symmetry that $S$ induces.}
\label{F:HalfRhombusSymmetry}
\end{figure}

The induced map is depicted in Figure \ref{F:HalfRhombusSymmetry}, where it is easy to see that the induced map must exchange the saddle connections labelled $2$ and $4$. This deduction applies to the unfolding of any parallelogram with angles $\theta$. However, the indicated saddle connections only have equal lengths when the parallelogram is a rhombus, which is a contradiction. 
\end{proof}

\begin{prop}\label{P:Discrete}
Suppose that $\theta$ is an ordered $4$-tuple of positive rational multiples of $\pi$ that sum to $2\pi$. Let $\cP(\theta)$ be the collection of polygons whose vertices have angles, listed in clockwise order, belonging to $\theta$. Suppose that for any polygon in $\cP(\theta)$ it is possible to change the edge lengths of two parallel edges by equal amounts, while fixing the other two and while remaining in $\cP(\theta)$. Suppose finally that $\cM(\theta)$ has rank at least three and that no polygon in $\cP(\theta)$ is Gaussian or Eisenstein. Then the set of unit area polygons in $\cP(\theta)$ whose unfolding has orbit closure strictly contained in $\cM(\theta)$ is discrete.
\end{prop}

Recall that a polygon is called \emph{Gaussian} (resp. \emph{Eisenstein}) if it can be tiled by reflections of isometric $45-45-90$ (resp. $30-60-90$) triangles (see Apisa-Wright \cite{ApisaWright}).

\begin{proof}
Let $P$ be a polygon in $\cP(\theta)$. Let $D(P)$ be its double and $(X, \omega)$ its unfolding. Since it is possible to change two edges by equal amounts, suppose the two edges are horizontal and let $s$ be an arc in $P$ joining these two edges. The preimage of $s$ on $D(P)$ is a closed loop $\gamma$. The preimages of $\gamma$ consist of loops $\gamma_1, \hdots, \gamma_m$ so that the holonomy of $\gamma_i$ is a positive real multiple of $\zeta_i$ where $\zeta_i$ is a root of unity. Define
\[ v := i\sum_{i=1}^m \zeta_i \gamma_i^* \]
where $\gamma_i^*$ is the Poincare dual of $\gamma_i$. Notice that $(X, \omega) + tv$ is the unfolding of $P$ once the lengths of the two edges that $s$ crosses are increased by length $t$. Up to rescaling, the unfolding of every unit area element of $\cP(\theta)$ belongs to $L$, the line $(X, \omega) + tv$ where $t \in \mathbb{R}$.


Since no polygon in $\cP(\theta)$ is Gaussian or Eisenstein, no polygon in $\cP(\theta)$ has an unfolding that is a torus cover by Apisa-Wright \cite[Proposition 3.11]{ApisaWright}. By Eskin-Filip-Wright \cite{EFW}, there is a finite collection invariant subvarieties of $\cM(\theta)$ of rank at least two and, excluding loci of torus covers, all but finitely many rank one invariant subvarieties are contained in an invariant subvariety of $\cM(\theta)$ of rank exactly two.  Since $\cM(\theta)$ has rank at least three, it follows that there is a finite collection of proper invariant subvarieties $\{\cM_i\}$ so that if the orbit closure of the unfolding of a polygon in $\cP(\theta)$ is not $\cM(\theta)$ then it is contained in $\cM_i$ for some $i$.

Since $L$ is linear in local period coordinates and not contained in any $\cM_i$, each of which is defined in period coordinates by linear equations, it follows that the intersection of $L$ and $\cM_i$ is discrete for each $i$, as desired.
%
%
%
%
\end{proof}

 


\begin{proof}[Proof of Theorem \ref{T:Main:SpecificSphere}:]
Suppose that $P$ is a polygon whose angles appear in the cyclic order $\theta := \left( \frac{a}{n}, \frac{b}{n}, \frac{1}{2}, \frac{1}{2} \right)\pi$ where $a, b$ and $n$ are positive integers with $\mathrm{gcd}(a,b,n) = 1$ and $a+b = n$. The unfolding of $P$ belongs to the quadratic double of $\cQ(2a-2, 2b-2, -1^{2n})$ if $n$ is odd and of $\cQ(a-2, b-2, -1^n)$ when $n$ is even.

Since each right trapezoid may be doubled to form an isosceles trapezoid and since, when $n > 2$, each isosceles trapezoid is tiled by two congruent right trapezoids, there is a bijective correspondence between right and isosceles trapezoids. As we have noted before (for instance by the argument in the first two paragraphs of Proposition \ref{P:IsoscelesCovers}) this implies that, letting $\theta' = \left(\frac{a}{n}, \frac{b}{n}, \frac{a}{n}, \frac{b}{n}, \right)\pi$, $\cM(\theta')$ is a full locus of covers of $\cM(\theta)$. Since $\cM(\theta')$ is a full locus of covers of $\cQ(2a-2, 2b-2, -1^{2n})$ if $n$ is odd and of $\cQ(a-2, b-2, -1^n)$ if $n$ is even (by Proposition \ref{P:Parallelogram}). This implies the claim in the first paragraph of Theorem \ref{T:Main:SpecificSphere}. The claim in the second paragraph follows from Proposition \ref{P:Discrete} (note that if $n \ne 2, 3, 4$ or $6$, then a polygon with angles in $\theta$ is neither Gaussian nor Eisenstein; moreover, the rank of $\cM(\theta)$ is at least three). The claim about the finite blocking problem in the third paragraph is immediate from Proposition \ref{P:FB:Zero}. The claim about counting problems in the final paragraph follows immediately from Corollary \ref{C:T:AEZ} and Lemma \ref{L:UnfoldingsVsPolygons} as in the proof of Theorem \ref{T:Counting:ClosedTrajectories}.
%
%
\end{proof}

 


\begin{proof}[Proof of Theorem \ref{T:Unfolding:Parallelogram}:]
The claims in the first two paragraph are given by Propositions \ref{P:Parallelogram} and \ref{P:Discrete}. The claims about counting problems in the final paragraph follows immediately from Theorem \ref{T:AEZ:Corollary} and Lemma \ref{L:UnfoldingsVsPolygons} as in the proof of Theorem \ref{T:Counting:ClosedTrajectories}. It remains to analyze the finite blocking problem.


Suppose that $P$ is a polygon with angles $\theta := \left( \frac{a}{n}, \frac{b}{n}, \frac{a}{n}, \frac{b}{n} \right)\pi$ where $a, b$ and $n$ are positive integers with $\mathrm{gcd}(a,b,n) = 1$ and $a+b = n$. Let $(X, \omega)$ be the unfolding of $P$ and suppose that it has orbit closure equal to $\cM(\theta)$. Let $D(P)$ denote the double of $P$. The double is invariant by an involution $J$ whose fixed points are the points corresponding to the midpoint of $P$ (when $P$ is a parallelogram) and to the midpoint of the two parallel edges (when $P$ is an isosceles trapezoid). We will call these points and vertices the \emph{special points of $P$}. As noted above (see Figure \ref{F:ParallelogramSymmetry}, when $P$ is a parallelogram $J$ lifts to the hyperelliptic involution on $(X, \omega)$ and the preimage of special points are Weierstrass points. The same conclusion holds when $P$ is an isosceles trapezoid. Since $\cM(\theta)$ is a quadratic double, it follows from Theorem \ref{T:QuadraticPeriodicBackground} that the preimages of the special points contain the periodic points on $(X, \omega)$.


When $n$ is odd, $(X, \omega)$ has dense orbit in the quadratic double of $\cQ(2a-2,2b-2, -1^{2n})$ and in this case the solution of the finite blocking problem follows from Proposition \ref{P:FB:Zero}. 

When $n$ is even, $(X, \omega)$ covers a surface with dense orbit in $\cQ^{hyp}((a-2)^2,(b-2)^2)$. By Proposition \ref{P:FB:Hyp}, the only pairs of finitely blocked points must consist of special points and the blocking set must also consist of special points. Another application of Proposition \ref{P:FB:Hyp} shows that no vertex is finitely blocked from any other point unless it has angle $\frac{\pi}{n}$ (and then it is only potentially finitely blocked from another vertex of the same angle). If $T$ denotes the unique translation involution on $(X, \omega)$, then these points map to a single Weierstrass point on $(X, \omega)/T$ (see Apisa-Wright \cite[Section 9.1]{ApisaWrightDiamonds}), which is finitely blocked from itself. It follows that two (not necessarily distinct) vertices of angle $\frac{\pi}{n}$ are finitely blocked from each other. 


It remains to consider blocking between the special points that are not vertices. In the isosceles trapezoid case it is clear that the the two special points are not blocked from each other. Therefore, in that case, it remains to show that they are not blocked from themselves. Since an isosceles trapezoid is tiled by two congruent right trapezoids (where the special points correspond to the right angles of the right trapezoid), this is immediate in the isosceles trapezoid case by the solution of the finite blocking problem in the right trapezoid case. In the parallelogram case, the special point is the midpoint, and any trajectory that strikes an edge perpendicularly will return to the midpoint without passing through a special point. 
\end{proof}

\section{Cyclic covers of $k$-differentials on the sphere - Proof of Theorem \ref{T:Main:GeneralSphere}}\label{S:AlmostRightPolygons}

%
%
%

A \emph{generalized polygon} is an immersed disk in $\mathbb{R}^2$ whose boundary is piecewise linear, i.e. if $f: \mathbb{D} \ra \mathbb{R}^2$ is the immersion there is a cyclically ordered set of points $(p_1, \hdots, p_k)$ on the boundary of $\mathbb{D}$ so that $f$ takes the boundary arc joining $p_i$ to $p_{i+1}$ to a straight line. The set $\{ f(p_i) \}$ are called \emph{vertices} of the generalized polygon. Two vertices are said to be \emph{adjacent} if one is $f(p_i)$ and the other is $f(p_{i+1})$ for some $i$. A generalized polygon $P$ has a \emph{double} $D(P)$ which is formed by taking $P$ and its mirror image $P'$ and identifying sides that are mirror images of each other to form a flat sphere with cone points. The polygon is said to be \emph{rational} if all of the cone points have cone angles that are rational multiples of $\pi$. The \emph{angle of a vertex} is defined to be half the cone angle of the corresponding cone point on $D(P)$. 

We will say that a generalized polygon is \emph{almost-right} if it is rational and all but exactly two angles are integer multiples of $\frac{\pi}{2}$. Similarly, we will say that a stratum of $k$-differentials on the sphere is almost-right if all but at most two cone points have cone angles that are not integer multiples of $\pi$. The \emph{signature} of an almost-right polygon (resp. $k$-differential) will be the list of angles (resp. cone angles) recorded with multiplicity.

We will now describe a construction, which we call \emph{gluing in a square}, for producing one almost-right polygon from another (see Figure \ref{F:GluingSquares} for examples). Let $P$ be a generalized almost-right polygon with an edge $\gamma$ of length $\ell$ that connects a vertex of angle $\frac{\pi}{2}$ to another vertex (or to a marked point, i.e. a vertex of angle $\pi$) $p$ of angle $\theta_p$. Gluing one edge of a square with side length $\ell$ into $\gamma$ produces a new generalized almost-right polygon with one more vertex of angle $\frac{\pi}{2}$ and where the vertex $p$ now has angle $\theta_p + \frac{\pi}{2}$.

\begin{lem}\label{L:AlmostRightConstruction}
An almost-right polygon of any signature can be constructed by successively applying the gluing-in-a-square construction to a right triangle. 
\end{lem}

By taking the doubles of the polygons in Lemma \ref{L:AlmostRightConstruction} we have produced $k$-differentials in any almost-right stratum (note that all strata of $k$-differentials on the sphere are connected). 

\begin{proof}
Let $\kappa$ be the signature of the almost-right polygon. Let $\alpha_1$ and $\alpha_2$ be the two angles that are not multiples of $\frac{\pi}{2}$ and let $\beta_1$ and $\beta_2$ be these angles taken modulo $\frac{\pi}{2}$. Set $\kappa_{>\frac{\pi}{2}}$ to be the multi-set of elements in $\kappa$ that are either equal to $\alpha_i$ or that are strictly greater than $\frac{\pi}{2}$. Since the sum of the elements in $\kappa$ is $(|\kappa|-2)\pi$, if we produce an almost-right polygon with signature $\kappa'$ so that $\kappa'_{>\frac{\pi}{2}} = \kappa_{>\frac{\pi}{2}}$ then we will have produced an almost-right polygon of signature $\kappa$.

Let $T$ be the right triangle with angles $\left( \beta_1, \beta_2, \frac{\pi}{2} \right)$ with vertex $v_i$ having angle $\beta_i$. If $\alpha_i > \frac{\pi}{2}$, then apply the gluing construction to the edge joining $v_i$ to a vertex of angle $\frac{\pi}{2}$. Let $p$ be any vertex of angle $\frac{\pi}{2}$ that is not adjacent to a vertex $v_i$ with $\alpha_i > \frac{\pi}{2}$.

Iteratively apply the gluing construction to an edge joining $v_i$ and an adjacent vertex of angle $\frac{\pi}{2}$ until $v_i$ has angle $\alpha_i$. These gluings never involve an edge with endpoint $p$, which therefore still has angle $\frac{\pi}{2}$.


Let $m$ be the number of elements of $\kappa$, excluding $\alpha_1$ and $\alpha_2$, that are greater than $\frac{\pi}{2}$. Apply the gluing construction to the segment of an edge beginning at a marked point and traveling clockwise to $p$. This replaces $p$ with three vertices whose angles, appearing in clockwise order, are $\frac{3\pi}{2}, \frac{\pi}{2}$, and $\frac{\pi}{2}$. We will relabel the third vertex as $p$. Repeating this construction $m-1$ more times, produces a series of vertices that alternate ($m$ times) between $\frac{3\pi}{2}$ and $\frac{\pi}{2}$.


Since each vertex of angle $\frac{3\pi}{2}$ is succeeded (in the clockwise orientation of the boundary) by one of angle $\frac{\pi}{2}$ we may iteratively apply the gluing construction to these edges until we have created an almost-right polygon of signature $\kappa'$ so that $\kappa'_{>\frac{\pi}{2}} = \kappa_{>\frac{\pi}{2}}$, as desired. 
%
\end{proof}


\begin{proof}[Proof of Theorem \ref{T:Main:GeneralSphere}:]
By Apisa \cite[Lemmas 4.2 and 4.4]{Apisa-PeriodicPointsG=2}, given Theorems \ref{T:Main}, \ref{T:Unfolding:Parallelogram}, and \ref{T:Main:SpecificSphere}, it only remains to consider the orbit closures of unfoldings of almost-right polygons.

Let $P$ be an almost-right polygon with $k$ sides and signature $\kappa$. Notice that $P$ has a partial unfolding to a genus zero quadratic differential in a stratum $\cQ$. Let $\cM'(\kappa)$ be the smallest invariant subvariety containing these partial unfoldings. It suffices to show that $\cM'(\kappa) = \cQ$. When $k = 4$ this follows from Theorem \ref{T:Main:SpecificSphere}, so suppose that $k \geq 5$.  By the ``trivial rank bound" of Mirzakhani-Wright \cite[Lemma 7.1]{MirWri2}, $\cM'(\kappa)$ has rank at least three. In the sequel, it suffices to ignore vertices of angle $\pi$ by Apisa-Wright \cite{ApisaWright} (see Theorem \ref{T:QuadraticPeriodicBackground}).

Let $\alpha_1$ and $\alpha_2$ be the two angles that are not multiples of $\frac{\pi}{2}$. In the sequel, we will say that a \emph{cylinder} on $P$ is a maximal embedded rectangle and three of whose sides are subsets of edges of $P$. The cylinder will be called \emph{special} if the two parallel edges belonging to the boundary of $P$ are comprised of a full edge from a vertex of angle $\frac{\pi}{2}$ to another vertex $v$ of angle $\alpha$ and a segment joining a vertex of angle $\frac{\pi}{2}$ to a marked point. Sending the length of these two parallel edges to zero will be called \emph{collapsing the (special) cylinder} and changes the signature $\kappa$ by deleting a vertex of angle $\frac{\pi}{2}$ and replace $\alpha$ with $\alpha - \frac{\pi}{2}$. 

\begin{sublem}\label{SL:FewObtuse1}
Suppose that all angles aside from possibly $\alpha_1$ and $\alpha_2$ are not obtuse, then $\cM'(\kappa) = \cQ$.
\end{sublem}
\begin{proof}
Proceed by induction with the result for $k = 4$ forming the base case. By Theorem \ref{T:Main}, $\cM'(\kappa)$ is a locus of covers. It suffices to show that the degree $d$ of these covers is one. By the construction in Lemma \ref{L:AlmostRightConstruction}, there is an almost-right polygon $P'$ with signature $\kappa$ that contains a special cylinder $R$ that contains the vertex of angle $\alpha_i$, for $i \in \{1, 2\}$, in its boundary. Collapsing the special cylinder produces an almost-right polygon whose signature $\kappa'$ still satisfies the hypotheses of the sublemma. Since $\cM'(\kappa')$ is contained in the boundary of $\cM'(\kappa)$ and consists of a full stratum (by the induction hypothesis), it follows that $d=1$. 
\end{proof}

\begin{sublem}\label{SL:FewObtuse2}
If $\kappa$ contains exactly one obtuse angle which is equal to $\frac{3\pi}{2}$, then $\cM'(\kappa) = \cQ$.
\end{sublem}
\begin{proof}
Since $\alpha_1$ and $\alpha_2$ is acute their sum is $\frac{\pi}{2}$. Since the sum of angles in $\kappa$ is $(|\kappa|-2)\pi$, it follows that $\kappa = \left( \frac{a}{2m}, \frac{b}{2m}, \frac{1}{2}, \frac{1}{2}, \frac{3}{2} \right)\pi$, where $a$, $b$, and $m$ are positive integers such that $a + b = m$ and $\mathrm{gcd}(a,b,2m) = 1$. Without loss of generality $a$ is odd and $D(P)$ partially unfolds to a surface in $\cQ(a-2, b-2, -1^{2m}, 1^m)$. The rank is $\lceil \frac{3m-1}{2} \rceil \leq \frac{3m}{2}$. By the construction in Lemma \ref{L:AlmostRightConstruction}, we may replace $P$ with a polygon $P'$ for which the angles appear in clockwise order as indicated in $\kappa$ (see Figure \ref{F:5gonDegeneration}). We can move the angle of size $\frac{3\pi}{2}$ into the acute angle $\frac{b\pi}{2m}$ to form a right trapezoid whose unfolding has dense orbit in $\cQ(a-2, b+m-2, -1^{2m})$, which has rank $m$ (see Figure \ref{F:5gonDegeneration}). 

\begin{figure}[h]\centering
\includegraphics[width=0.45\linewidth]{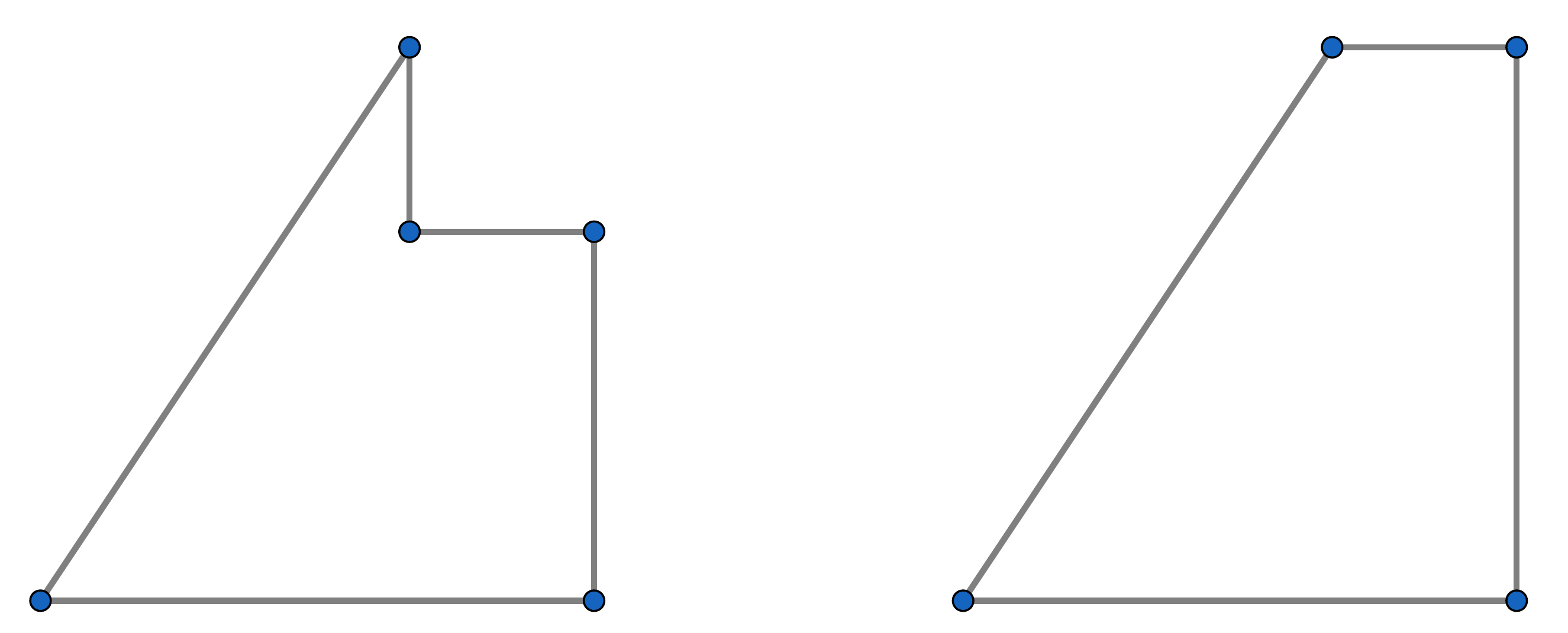}
\caption{The degeneration of $P'$ to a trapezoid.}
\label{F:5gonDegeneration}
\end{figure}

Note that $\cM'(\kappa)$ is high rank if $m \geq \frac{3m}{4} + 1$, i.e. if $m \geq 4$. In this case it is a locus of covers (by Apisa-Wright \cite{ApisaWrightHighRank}), but we have exhibited a degeneration where the minimal cover is the identity, so $\cM'(\kappa) = \cQ$ if $m \geq 4$. The same argument works when $m = 2, 3$ since the rank of $\cQ$ is $3$ and $4$ respectively and we have already noted that $\cM'(\kappa)$ has rank at least $3$. 
%
\end{proof}



\begin{sublem}\label{SL:StrangeDiamond}
Suppose that $P'$ is an almost-right polygon of signature $\kappa$ with disjoint special cylinders $C_1$ and $C_2$ whose preimage on the partial unfolding of $D(P)$ are denoted $\bfC_1$ and $\bfC_2$. Suppose that collapsing $C_i$ results in an almost-right polygon $P_i$ of signature $\kappa_i$ so that $\cM'(\kappa_i)$ is a stratum for $i \in \{1, 2\}$. Then $\cM'(\kappa) = \cQ$. 
\end{sublem}
\begin{proof}
Collapsing $C_i$ produces a path in the space of almost-right polygons of signature $\kappa$ to $P_i$. This path has a corresponding path in $\cM'(\kappa)$. Since $\bfC_j$ persists on the unfolding of $D(P_i)$ (for $i \ne j$), which is contained in the stratum $\cM'(\kappa_i)$, it follows that $\bfC_j$ is a collection of free cylinders, i.e. cylinders that can be dilated and sheared while fixing the rest of the surface and while remaining in $\cM'(\kappa)$. Since $\cM'(\kappa_j)$ is a full stratum that is formed by collapsing a collection of free simple cylinders in $\cM'(\kappa)$, it follows that $\cM'(\kappa)$ is itself a full stratum. 
\end{proof}

\begin{sublem}\label{SL:FewObtuse3}
If $\kappa$ contains exactly one obtuse angle, then $\cM'(\kappa) = \cQ$.
\end{sublem}
\begin{proof}
We will proceed by induction on the number of sides $k$. We have already noted that the $k = 4$ case has been established. Sublemmas \ref{SL:FewObtuse1} and \ref{SL:FewObtuse2} allow us to assume that the obtuse angle is equal to $\frac{\ell \pi}{2}$ where $\ell \geq 4$. The construction in Lemma \ref{L:AlmostRightConstruction} shows that we can find an almost-right polygon $P'$ with signature $\kappa$ that has two disjoint special cylinders (see Figure \ref{F:GluingSquares} (left)). The induction hypothesis and Sublemma \ref{SL:StrangeDiamond} allow us to conclude. 
\end{proof}

\begin{figure}[h]\centering
\includegraphics[width=0.75\linewidth]{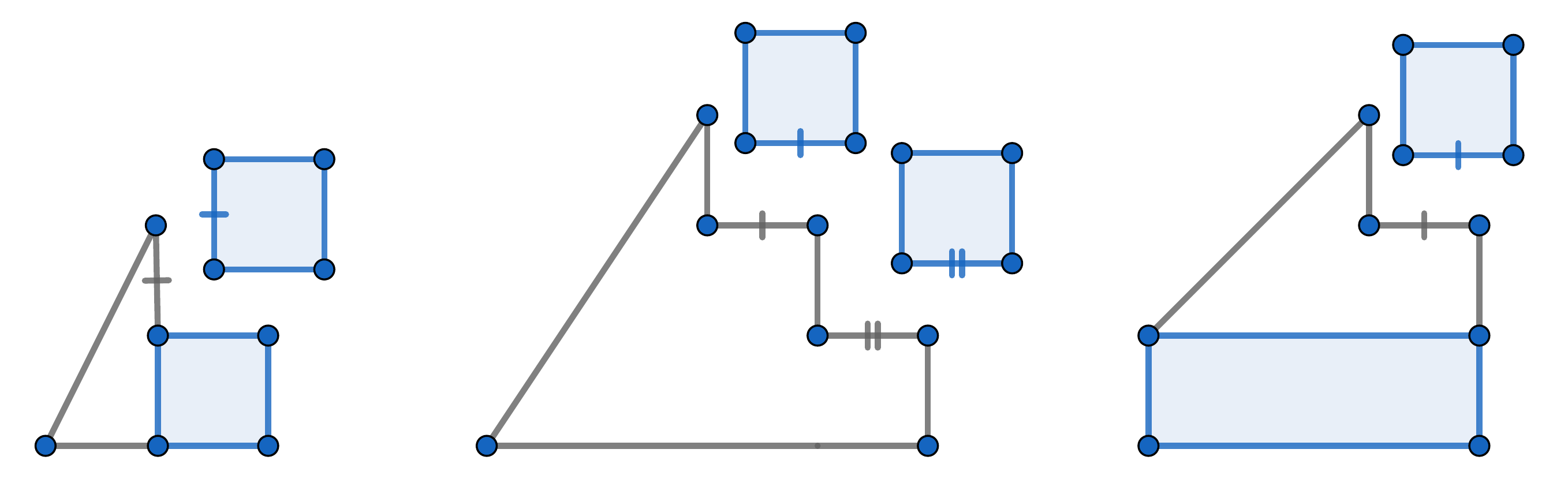}
\caption{Disjoint special cylinders on almost right polygons with one obtuse angle of the form $\frac{\ell \pi}{2}$ for $\ell \geq 4$ (left), two obtuse angles that are integral multiples of $\frac{\pi}{2}$ (middle), and one obtuse angle that is an integral multiple of $\frac{\pi}{2}$ and one of $\alpha_1$ and $\alpha_2$ also obtuse (right).}
\label{F:GluingSquares}
\end{figure}

To conclude the proof we will proceed by induction on the number of sides $k$. We have already noted that the $k = 4$ case has been established. Sublemmas \ref{SL:FewObtuse1} and \ref{SL:FewObtuse3} allow us to assume that there are at least two obtuse angles in $\kappa$ and that at least one is an integral multiple of $\frac{\pi}{2}$. The construction in Lemma \ref{L:AlmostRightConstruction} shows that we can find an almost-right polygon $P'$ with signature $\kappa$ that has two disjoint special cylinders (see Figure \ref{F:GluingSquares} (middle and right)). The induction hypothesis and Sublemma \ref{SL:StrangeDiamond} allow us to conclude. 
\end{proof}

\bibliography{mybib}{}
\bibliographystyle{amsalpha}
\end{document}